\documentclass{amsart}
\pdfoutput=1

\usepackage{amssymb}
\usepackage{amsthm}
\usepackage{amsfonts}
\usepackage{amsmath}
\usepackage{pmboxdraw}
\usepackage{verbatim} % for comment
\usepackage{graphicx}
\usepackage{color}
\usepackage[colorlinks=true, citecolor=blue, filecolor=black, linkcolor=black, urlcolor=black]{hyperref}
\usepackage{cite}
\usepackage[normalem]{ulem}
\usepackage{subcaption}
\usepackage{bbm}
\usepackage{bm}
\usepackage{mathtools}
\usepackage{todonotes}
\usepackage{array}
\usepackage{multirow}
\usepackage{multirow,makecell}  % table with merged rows/columns and cells with multiple lines
\allowdisplaybreaks
\usepackage{colortbl}
\newcommand{\RN}[1]{%
	\textup{\uppercase\expandafter{\romannumeral#1}}%
}

\def\ss{\smallskip}

\def\C{\mathbb{C}}

\def\R{\mathbb{R}}
\def\N{\mathbb{N}}

\def\tt{\mathsf{t}}
\def\ss{\mathsf{s}}
\def\TT{\mathsf{T}}

\newcommand{\re}{\operatorname{Re}}

	\def\Xint#1{\mathchoice
		{\XXint\displaystyle\textstyle{#1}}%
		{\XXint\textstyle\scriptstyle{#1}}%
		{\XXint\scriptstyle\scriptscriptstyle{#1}}%
		{\XXint\scriptscriptstyle\scriptscriptstyle{#1}}%
		\!\int}
	\def\XXint#1#2#3{{\setbox0=\hbox{$#1{#2#3}{\int}$ }
			\vcenter{\hbox{$#2#3$ }}\kern-.59\wd0}}
	
	\def\dashint{\Xint-}
\theoremstyle{plain}
\newtheorem{thm}{Theorem}[section]

\newtheorem{cor}[thm]{Corollary}
\newtheorem{lem}[thm]{Lemma}

\newtheorem{prop}[thm]{Proposition} 

\theoremstyle{remark}

\newtheorem{rem}{Remark}
\newtheorem{assump}[thm]{Assumption}

\numberwithin{equation}{section}

\begin{document}

\title[Precise large deviations in geometric last passage percolation]{Precise large deviations in geometric last passage percolation}

%%%%%%%%%%%%%%%%%%%%%%%%%%%%% author %%%%%%%%%%%%%%%%%%%%%%%%%%%%
\author{Sung-Soo Byun}
\address{Department of Mathematical Sciences and Research Institute of Mathematics, Seoul National University, Seoul 151-747, Republic of Korea}
\email{sungsoobyun@snu.ac.kr}

\author{Christophe Charlier}
\address{ Research Institute in Mathematics and Physics, UCLouvain, 1348 Louvain-La-Neuve, Belgium}
\email{christophe.charlier@uclouvain.be}

\author{Philippe Moreillon}
\address{Research Institute of Mathematics, Seoul National University, Seoul 151-747, Republic of Korea}
\email{phmoreil@snu.ac.kr}

\author{Nick Simm}
\address{Department of Mathematics, University of Sussex, Brighton, BN1 9RH, United Kingdom}
\email{n.j.simm@sussex.ac.uk}
%%%%%%%%%%%%%%%%%%%%%%%%%%%%% author %%%%%%%%%%%%%%%%%%%%%%%%%%%%

\begin{abstract}
We study the last passage time in geometric last passage percolation (LPP). As the system size increases, we derive precise large deviation probabilities---up to and including the constant terms---for both the lower and upper tails. A key step in proving these results is to establish a \textit{duality formula} that reformulates the LPP problem in terms of the largest eigenvalue in the Jacobi unitary ensemble (JUE). In addition, we establish a second duality formula, which relates the LPP problem to the truncated unitary ensemble (TUE). Using this, we also derive asymptotics for the moments of the absolute value of characteristic polynomials of the TUE, which may be of independent interest.
\end{abstract}

\maketitle

%\tableofcontents

\section{Introduction and main results}\label{sec:Intro}

Consider the positive quadrant of the integer lattice $\mathbb{Z}_{+}^{2}$ and on each site put independent and identically distributed geometric random variables $\omega_{i,j}$ with parameter $q^2 \in (0,1)$, i.e. the probability mass function of each $\omega_{i,j}$ is
\begin{equation}
\label{weights}
	\mathbb{P}(\omega_{i,j} = k) = (1-q^2)q^{2k}, \qquad k=0,1,2,\ldots
\end{equation}
Now consider the collection $\Pi$ of all up-right paths $\pi$ that connect the site $(1,1)$ to $(n,m)$. The last passage time is the random variable
\begin{equation}
	G_{n,m} = \max_{\pi \in \Pi}\sum_{(i,j) \in \pi}\omega_{i,j}. \label{gnm}
\end{equation}
The geometric last passage percolation (LPP) model is defined by the random path that attains the maximum total weight among all up-right paths, as illustrated in Figure~\ref{Fig_LPP}. 

In this work, we study the asymptotic behaviour of the probability $\mathbb{P}(G_{n,m} \leq \ell)$ in regimes where $n,m,\ell$ are large. The analysis of such probabilities has played a central role in the development of the Kardar--Parisi--Zhang (KPZ) universality class and the early advances in integrable probability theory.

\begin{figure}[t]
	\begin{subfigure}{0.32\textwidth}
	\begin{center}	
		\includegraphics[width=\textwidth]{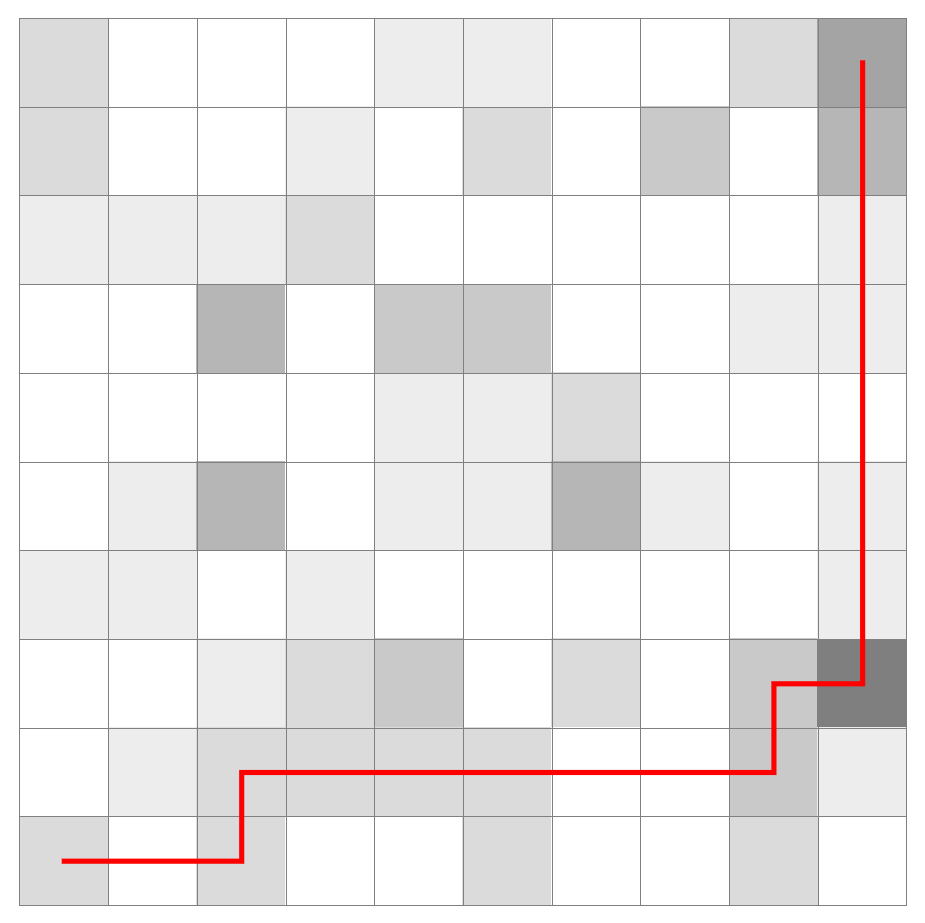}
	\end{center}
	\subcaption{$n=m=10$}
\end{subfigure}	
\begin{subfigure}{0.32\textwidth}
	\begin{center}	
		\includegraphics[width=\textwidth]{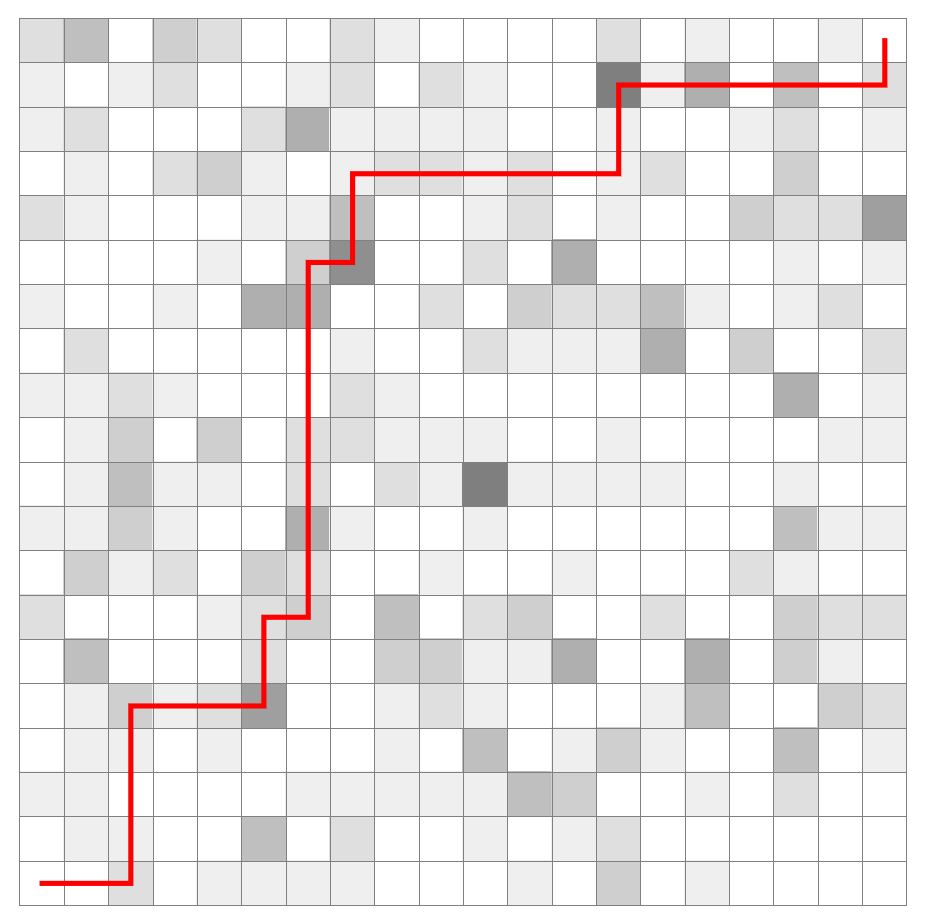}
	\end{center}
	\subcaption{$n=m=20$}
\end{subfigure}	
	\begin{subfigure}{0.32\textwidth}
	\begin{center}	
		\includegraphics[width=\textwidth]{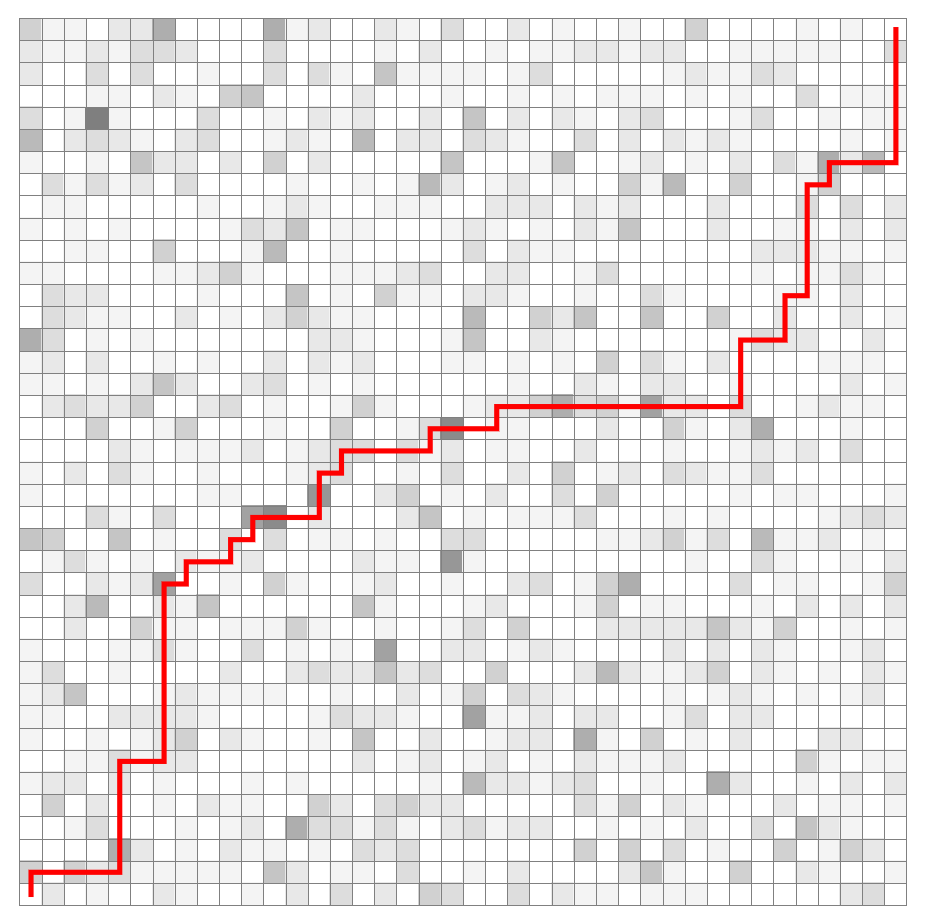}
	\end{center}
	\subcaption{$n=m=40$}
\end{subfigure}	
	\caption{Plots of the last passage paths (red lines) from $(0,0)$ to $(n,m)$ with $q=1/\sqrt{2}$. The panels correspond to (A) $n=m=10$, (B) $n=m=20$, and (C) $n=m=40$. Lighter shading indicates smaller values of the weights $\omega_{i,j}$.  } \label{Fig_LPP}
\end{figure}

A law of large numbers type result for \eqref{gnm} was investigated in the seminal work \cite[Theorem 1.1]{Joh00} by Johansson, where it is shown that for $\gamma \geq 1$, 
\begin{equation} \label{LLN for LPP}
	\lim_{N \to \infty} \frac{1}{N} \mathbb{E} [ G_{\lfloor\gamma N \rfloor, N} ] = \omega(\gamma,q):= \frac{ (1+q\sqrt{\gamma } )^2 }{ 1-q^2}-1. 
\end{equation}
Here, $\lfloor x\rfloor$ is the largest integer less or equal to $x$. (Note that $q^{2}$ here corresponds to $q$ in \cite{Joh00}.)

It is also known from \cite[Theorem 1.2]{Joh00} that the fluctuations of $G_{\lfloor \gamma N \rfloor, N}$ around the typical value \(\omega(\gamma,q)\) are described by the Tracy--Widom distribution: more precisely,
\begin{equation} \label{def of scaled TW}
\mathbb{P}( G_{\lfloor \gamma N \rfloor, N} \leq x ) =   F_{\mathrm{TW}}\Big( \frac{x - N \omega(\gamma, q)}{\sigma(\gamma, q) N^{1/3}} \Big) (1+o(1)),
\end{equation}
as $N \to \infty,$ where
\begin{align*}
\sigma(\gamma,q) = \frac{q^{1/3} \gamma^{-1/6}}{1 - q^2} ( \sqrt{\gamma} + q )^{2/3} ( 1 + q\sqrt{ \gamma} )^{2/3},
\end{align*}
and $F_{\rm TW}(s)$ is the Tracy-Widom distribution, which we recall is given by
\begin{equation*}
F_{\rm TW}(s)=\exp\bigg(-\int_s^{+\infty}(x-s)u(x)^2 dx\bigg),
\end{equation*}
with $u$ the Hastings-McLeod solution of the Painlev\'e II equation, see also Figure \ref{Fig_rare events}. We note that the left-hand side of \eqref{def of scaled TW} admits a Fredholm determinant representation~\cite{BO00}, see also~\cite[Appendix]{CLW16} where this formula is used to derive deviation probabilities, and~\cite{FR05,BS25} for Fredholm determinant representations associated with variants of the LPP models.

Beyond fluctuations around the typical values, a natural next question concerns the probabilities of rare events. These can be divided into two cases: 
\begin{itemize}
    \item the \emph{lower tail probability}, where \( G_{m,n} \) is significantly smaller than its typical value, and 
    \smallskip 
    \item the \emph{upper tail probability}, where \( G_{m,n} \) is significantly larger,
\end{itemize}  
see again Figure~\ref{Fig_rare events}. Due to the typical fluctuation described in \eqref{def of scaled TW} and the well-known tail behaviour of the Tracy--Widom distribution, one can expect that the lower tail probability is significantly smaller---and much harder to analyze---than the upper tail probability.
This expectation is confirmed by \cite[Theorem 1.1]{Joh00}, which shows that there are functions $\ell,i : (0,+\infty) \to (0,+\infty)$ such that for any fixed $\epsilon >0$,
\begin{align}
	& \lim_{ N \to \infty } \frac{1}{N^2} \log \mathbb{P}\Big[ G_{\lfloor\gamma N\rfloor, N} \le N ( \omega(\gamma,q)-\epsilon ) \Big]  = - \ell(\epsilon), \label{lowtail}
	\\
	&  \lim_{ N \to \infty } \frac{1}{N} \log \mathbb{P}\Big[ G_{\lfloor\gamma N \rfloor, N} \ge N ( \omega(\gamma,q)+\epsilon ) \Big]  = - i(\epsilon). \label{uptail}
\end{align}
An explicit formula for the upper tail rate function $i(x)$ is known and given in \cite[Theorem 2.2 and (2.21)]{Joh00} and \cite[Theorem 4.4]{Sep98}, see also \cite{BDMMZ01} and Remark~\ref{Rem_Johansson's upper tail} below. In contrast to the upper tail rate function, it is written in \cite[p.451]{Joh00} that  
\emph{``We will not discuss the explicit form of the lower tail rate function.''}  
In fact, even though Johansson's work \cite{Joh00} is 25 years old, an explicit formula for $\ell(x)$ still appears to be lacking in the literature. In the recent paper \cite[p.29]{DZ22}, the authors write  
    \emph{``The only impediment is that the Johansson result appears in a variational form.''} (In addition, we note that the variational characterisation of $\ell$ in \cite{Joh00} contains a minor error, which is addressed and clarified in detail in the recent work \cite[Theorem~7.14]{DD22}.)

\begin{figure}[t]
\begin{center}
\begin{tikzpicture}[scale=2, every node/.style={align=center}]  
\node at (0, 0) {\includegraphics[width=0.65\textwidth]{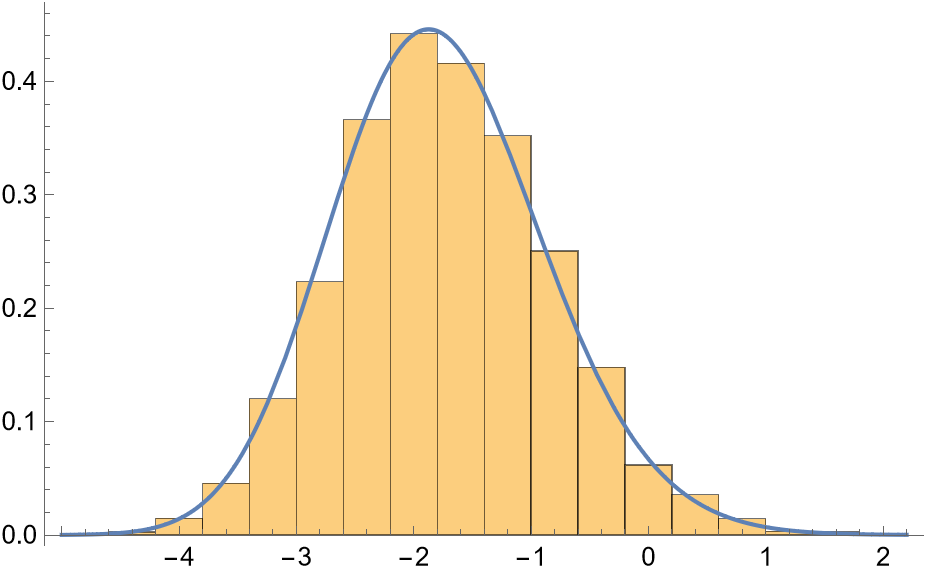}};    
\node[below] at (-2.2, -1.8) {\emph{\textbf{Lower tail}}};
\node[below] at (2.2, -1.8) {\emph{\textbf{Upper tail}}};
\end{tikzpicture}
\end{center}
    \caption{The histogram was made by computing $10^{4}$ samples of $\frac{G_{N, N} - N \omega(\gamma, q)}{\sigma(\gamma, q) N^{1/3}}$ with $N=400$. The blue curve is the density $F_{\mathrm{TW}}'(x)$ of the Tracy-Widom distribution.}  \label{Fig_rare events}
\end{figure} 

\medskip 

Our main goal in this paper is to obtain asymptotic formulas for both the lower and upper tail probabilities appearing in \eqref{lowtail} and \eqref{uptail}, respectively. For the lower-tail probability, our proof requires to treat the following two cases separately.
\begin{itemize}
    \item \textbf{Almost-square lattice} (see~Theorem~\ref{Thm_LPP weak}): This corresponds to the regime where \( n - m \) remains bounded as \( n \to \infty \).
    \smallskip
    \item \textbf{Rectangular lattice} (see~Theorem~\ref{th:lpp-lt}): This corresponds to the regime where the ratio \( n/m \) stays bounded away from \(1\) as \( n \to \infty \).
\end{itemize}
This distinction is only needed for technical reasons that will become apparent below; our results, however, suggest that no critical transition takes place between these two regimes, see also Remark \ref{remark:no critical transition}.

\begin{thm}[\textbf{Lower tail probability in an almost-square lattice}]  \label{Thm_LPP weak}  
Let $q \in (0,1)$, $0< \delta < \omega(1,q)$ and $\mathsf{n}\in \N:=\{0,1,\ldots\}$ be fixed. Then, as $N \to \infty$,     
	\begin{equation}
			\log \mathbb{P} \Big[ G_{  N+ \mathsf{n} , N  } \le \delta N \Big] = \mathsf{L}_1 N^2 + \mathsf{L}_2 N +\mathsf{L}_{3} \log N +\mathsf{L}_4 +O\Big(\frac{\log N}{N}\Big) ,
		\end{equation} 
        where 
        \begin{align*}
     \begin{split}
        \mathsf{L}_1 &= \log \big(\tfrac{1-q}{2}\big) +  ( 1+\delta )^2   \log\big(\tfrac{1+q}{2}\big)  - \frac{\delta^{2}}{2}\log q
          -(1+\delta)^{2}\log(1+\delta) +\frac{\delta^{2}}{2}\log\delta + \frac{(2+\delta)^{2}}{2}\log(2+\delta),
        \end{split}
        \\
        \begin{split}
        \mathsf{L}_2 &= \mathsf{n} \bigg[  \log \big(\tfrac{1-q^2}{4}\big) -   \delta \log \big(\tfrac{2}{ 1+q } \big) -   (1+\delta)\log(1+\delta) +   (2+\delta)\log(2+\delta) \bigg] , 
        \end{split}
        \\
        \begin{split}
        \mathsf{L}_3 &= -\frac{1}{12}, 
        \end{split}
        \\
        \begin{split}
        \mathsf{L}_4 &=  \zeta'(-1)+\frac{1}{12} \log \Big(\tfrac{ 2 (1+\delta)^2 }{ \delta (2+\delta) } \Big)  - \frac{1}{8} \log  \Big( \tfrac{  ( 2q+\delta q-\delta ) (  2+\delta-\delta q ) }{ 4 q } \Big)    
    + \frac{\mathsf{n}^2}{2} \log \Big(\tfrac{ (2+\delta-\delta q)(2+\delta) }{ 4 (1+\delta) } \Big). 
        \end{split}
        \end{align*}
Here, $\zeta$ is the Riemann zeta function. 
\end{thm}

To state our results in a general rectangular lattice, it is convenient to introduce the parameters  
\begin{equation} \label{x pm lpp}
	a = \bigg( \frac{  \sqrt{ \gamma(\gamma+\delta) } -\sqrt{\delta+1}   }{ \gamma+\delta+1 } \bigg)^2, \qquad  b =\bigg( \frac{  \sqrt{\gamma(\gamma+\delta) } +\sqrt{\delta+1}   }{\gamma+\delta+1} \bigg)^2
\end{equation}
for $\gamma\geq1$ and $0< \delta <\omega(\gamma,q)$. 
We note that these values correspond to the edges of the support of the Wachter distribution (see \eqref{def of Wachter distribution} below), once the parameters are properly chosen.
Furthermore, let $\mathfrak{m}\in (0,a)$ be the unique solution to the equation 
\begin{align}\label{a eq 4}
2 = (\gamma-1) \bigg( \frac{\sqrt{1-q^2}}{\sqrt{ \mathfrak{m} }}-1 \bigg) + \delta \bigg( \frac{q}{\sqrt{1-\mathfrak{m}  }}-1 \bigg).
\end{align}
Below, the parameter $\mathfrak{m}$ serves as the left edge of the \textit{constrained} Wachter distribution in Proposition~\ref{Prop_constrained Wacther}. See also Remark~\ref{Rem_explicit m} for another expression of $\mathfrak{m}$.

\begin{thm}[\textbf{Lower tail probability in a rectangular lattice}] \label{th:lpp-lt}
Let \( q \in (0,1) \), \( \gamma > 1 \) and \( 0 < \delta < \omega(\gamma, q) \) be fixed.   Let \( \mathsf{n}=\mathsf{n}_N \) be in a fixed compact subset of \( \mathbb{R} \), chosen so that \( \gamma N + \mathsf{n} \in \mathbb{N}_{>0} \) for all $N\in\mathbb{N}_{>0}$.  
Then, as \( N \to \infty \),
		\begin{equation}
			\log \mathbb{P} \Big[ G_{  \gamma N + \mathsf{n} , N  } \le \delta N \Big] = L_1 N^2 + L_2 N +L_{3} \log N +L_4 +O\Big(\frac{\log N}{N}\Big) , 
		\end{equation} 
where
\begin{align*}
\begin{split}
 L_{1} & =  \log \big(\tfrac{1-q^2-\mathfrak{m} }{4}\big)  +  ( 1+\gamma+\delta ) \bigg[ (\gamma-1)\log\Big(\tfrac{\sqrt{ \mathfrak{m} }+\sqrt{1-q^2 }}{2}\Big) + \delta \log\Big(\tfrac{\sqrt{1- \mathfrak{m} }+ q }{2}\Big) \bigg]  
 \\
&  \quad  - (\gamma-1)\delta\log\Big(\tfrac{\sqrt{ \mathfrak{m} q^2}+\sqrt{(1-q^2)(1-\mathfrak{m} )}}{2}  \Big)  -\frac{\gamma^{2}}{2}\log\gamma - \frac{(1+\delta)^{2}}{2}\log(1+\delta) - \frac{(\gamma+\delta)^{2}}{2}\log(\gamma+\delta)
\\
&\quad + \frac{(\gamma-1)^{2}}{4}\log\Big( \tfrac{(\gamma-1)^{2}}{ \mathfrak{m} (1-q^2)}\Big) + \frac{\delta^{2}}{4}\log\Big( \tfrac{\delta^{2}}{(1-\mathfrak{m} )q^2}\Big) + \frac{(1+\gamma+\delta)^{2}}{2}\log(1+\gamma+\delta), \end{split}
\\
\begin{split}
L_2 &=   \mathsf{n} \bigg[ (1+\gamma+\delta) \log\Big(\tfrac{\sqrt{ \mathfrak{m} }+\sqrt{1-q^2} }{2}\Big)  + (\gamma-1) \log \Big(   \tfrac{ \sqrt{\mathfrak{m}}+\sqrt{1-q^2} }{ 2\sqrt{ \mathfrak{m}(1-q^2) } }  \Big) - \delta \log \Big(\tfrac{\sqrt{  \mathfrak{m}  q^2}+\sqrt{(1-q^2)(1- \mathfrak{m}  ) } }{\sqrt{  1-\mathfrak{ m }   }+q }\Big)
\\
& \qquad   -  \gamma\log\gamma - (\gamma+\delta)\log(\gamma+\delta) + (2+\delta)\log(1+\gamma+\delta) + (\gamma-1) \log\big((\gamma-1)(1+\gamma+\delta)\big) \bigg],  
\end{split}
\\
L_3 &=-\frac{1}{12}, 
\\
\begin{split}
L_{4} & = \zeta^{\prime}(-1)-\frac{1}{6} \log \big( \tfrac{1-q^2- \mathfrak{m} }{4} \big)- \frac{1}{8}\log \Big( \tfrac{\gamma-1}{\sqrt{ \mathfrak{m} (1-q^2)}} - \tfrac{\delta}{\sqrt{(1- \mathfrak{m} )q^2} } \Big)  - \frac{1}{24}\log \Big( (\gamma-1) \tfrac{\sqrt{1-q^2}}{ \mathfrak{m} ^{3/2}} - \delta \tfrac{ q }{(1- \mathfrak{m} )^{3/2}} \Big)
\\
&\quad   - \frac{1}{12}\log  \Big(\tfrac{(\gamma-1)\delta(1+\gamma+\delta)}{\gamma(1+\delta)(\gamma+\delta)}\Big)
 + \frac{  \mathsf{n}^{2} }{2} \log \Big( \tfrac{(\sqrt{\mathfrak{m}}+\sqrt{1-q^{2}})^{2}}{4\sqrt{1-q^{2}}} \tfrac{\gamma - 1}{\sqrt{\mathfrak{m}}} \tfrac{1+\gamma+\delta}{\gamma(\gamma+\delta)} \Big) .
\end{split}
\end{align*}
Here, $\zeta$ is the Riemann zeta function. 
\end{thm}

\begin{rem}\label{remark:no critical transition}
A direct analysis using \eqref{a eq 4} shows that
\begin{align}\label{asymptotics of mfrak}
\mathfrak{m} = \frac{1-q^{2}}{(2+\delta-\delta q)^{2}}(\gamma-1)^{2} + O\big( (\gamma-1)^{3} \big), \qquad \mbox{as } \gamma \to 1.
\end{align}
Using \eqref{asymptotics of mfrak}, it is now easy to verify that $\lim_{\gamma\to 1} L_{j} = \mathsf{L}_j$ for $j=1,2,3,4$. This suggests that the statement of Theorem \ref{th:lpp-lt} in fact holds uniformly for $\gamma$ in compact subsets of $[1,+\infty)$.
\end{rem}
 
\begin{rem} \label{Rem_explicit m} 
Using the result in \cite[Section 4]{RKC12}, we can express $\mathfrak{m}$ more explicitly as 
\begin{equation}
\mathfrak{m}= \frac14 \bigg( \mathsf{B}+\sqrt{x_0} -\sqrt{ \mathsf{B}^2-2\mathsf{C}^2+2 -x_0-\tfrac{2}{ \sqrt{x_0}  }  \mathsf{B}(\mathsf{C}^2+1) } \bigg)^2,
\end{equation}
where  
\begin{align*}
& \mathsf{B}= \frac{ (\gamma-1) \sqrt{1-q^2} }{ 1+\gamma+\delta }, \qquad \mathsf{C}= \frac{\delta q }{ 1+\gamma+\delta }, \\
& x_0= -\frac{2\mathsf{C}^2-\mathsf{B}^2-2 }{3} +( \mathsf{R}+\sqrt{\mathsf{Q}^3+\mathsf{R}^2} )^{ 1/3 } +( \mathsf{R}-\sqrt{\mathsf{Q}^3+\mathsf{R}^2} )^{ 1/3 }, \\
& \mathsf{Q}= -\frac19 (1-\mathsf{B}^2-\mathsf{C}^2)^2,
	\qquad 
	\mathsf{R} = \frac{1}{27} \Big( \mathsf{B}^6-3\mathsf{B}^4(1-\mathsf{C}^2)+3\mathsf{B}^2(1+16\mathsf{C}^2+\mathsf{C}^4)-(1-\mathsf{C}^2)^3\Big) .
\end{align*} 
\end{rem}

\bigskip 

Next, we discuss the asymptotic expansion for the upper tail large deviation probability. In order to describe it, for $q \in (0,1)$, $\gamma \geq 1$ and $ \delta > \omega(\gamma,q)$, we define 
\begin{align}
\begin{split} \label{utexpan}
\Phi_{\gamma,\delta}(x) & := -\log\bigg(\frac{2x-a-b+2\sqrt{(x-b)(x-a)}}{b-a}\bigg)+(\gamma-1)\log\bigg(\frac{\sqrt{ab}+x-\sqrt{(x-b)(x-a)}}{(\sqrt{a}+\sqrt{b})\sqrt{x}}\bigg)
\\
&\quad +\delta\log\bigg(\frac{\sqrt{(1-a)(1-b)}-x+1+\sqrt{(x-b)(x-a)}}{(\sqrt{1-a}+\sqrt{1-b})\sqrt{1-x}}\bigg),   
\end{split}
\end{align} 
where $a$ and $b$ are given by \eqref{x pm lpp}.

\begin{thm}[\textbf{Upper tail probability}]
\label{th:lpp-ut}
Let \( q \in (0,1) \), \( \gamma \ge 1 \) and \( \delta > \omega(\gamma, q) \) be fixed.   Let \( \mathsf{n}=\mathsf{n}_N \) be in a fixed compact subset of \( \mathbb{R} \), chosen so that \( \gamma N + \mathsf{n} \in \mathbb{N}_{>0} \) for all $N\in\mathbb{N}_{>0}$. 
Then, as \( N \to \infty \), 
\begin{equation}
\label{lpp-ut}
	\log \mathbb{P} \Big[ G_{  \gamma N+ \mathsf{n} , N  } \geq \delta N \Big] = U_1 N + U_2 \log N +U_{3} + o(1),
	\end{equation}
where 
\begin{equation}
U_{1} = -2\Phi_{\gamma,\delta}(1-q^2), \qquad U_{2}=-1
\end{equation}
and
\begin{equation}
U_{3} =  \log \bigg( \frac{b-a}{8\pi}\,\frac{1}{(1-q^2-a)(1-q^2-b)}\frac{1}{2\Phi'_{\gamma,\delta}(1-q^2)} \bigg)  -  2  \mathsf{n} \frac{d}{d \gamma} \Phi_{\gamma,\delta}(1-q^2)  .
\end{equation}
Here, $\Phi_{\gamma,\delta}$ is given by \eqref{utexpan} and $\frac{d}{d \gamma}\Phi_{\gamma,\delta}$ denotes the total derivative of $\Phi_{\gamma,\delta}$ with respect to $\gamma$ (taking into account that $a$ and $b$ depend on $\gamma$).
\end{thm}

\begin{rem}[Comparison with Johansson's formula] \label{Rem_Johansson's upper tail}
In \cite[Eq. (2.21)]{Joh00}, it was shown that the upper tail rate function in \eqref{uptail} can be written as $i(\epsilon) = 2J(\mathsf{b}+\epsilon)$ where
\begin{equation}
J(t) = \frac{ \mathsf{b}-\mathsf{a} }{ 8q \sqrt{\gamma}}\int_{1}^{x}(x-y)\Big(\frac{\gamma-q^{2}}{y+B}+\frac{1-q^{2}\gamma}{y+D}\Big)\frac{dy}{\sqrt{y^{2}-1}}.
\end{equation}
Here,  
$\mathsf{a}=\frac{(1-q\sqrt{\gamma})^{2}}{1-q^2}$, $\mathsf{b}=\frac{(1+q\sqrt{\gamma})^{2}}{1-q^2}=\omega+1$, $x=\frac{2(t-\mathsf{a})}{\mathsf{b}-\mathsf{a}}-1$, $B = \frac{\gamma+q^2}{2q\sqrt{\gamma}}$ and $D=\frac{1+q^2\gamma}{2q\sqrt{\gamma}}$, where $\omega$ is given in \eqref{LLN for LPP}. It turns out that the integral above can be computed explicitly: for $x\geq1$ and $|B|\geq 1$, we have
\begin{equation}
\int_{1}^{x}\frac{x-y}{\sqrt{y^{2}-1}(y+B)} \, dy =  \frac{B+x}{\sqrt{B^{2}-1}}\log \bigg( \frac{1+Bx+\sqrt{(B^{2}-1)(x^{2}-1)}}{B+x} \bigg) + \frac{1}{2}\log \bigg( \frac{x-\sqrt{x^{2}-1}}{x+\sqrt{x^{2}-1}} \bigg).
\end{equation}
With $I(x,B)$ denoting the above integral, we have
\begin{equation*}
J(t) =  \frac{ \mathsf{b}-\mathsf{a} }{8q\sqrt{\gamma}}\Big( (\gamma-q^2)I(x,B)+(1-q^2\gamma)I(x,D)\Big).
\end{equation*}
It is now direct to check that the rate function \eqref{utexpan} is in agreement with $J(t)$ above for the appropriate choice of parameters, namely
\begin{equation}
 J(t) = \Phi_{\gamma,t-1}(1-q^2). 
\end{equation}
This formula confirms the coefficient $U_1$ in Theorem~\ref{th:lpp-ut}.  
    \end{rem}
 
\medskip 
 
The last passage percolation model is closely connected to numerous other integrable probability models. Here, we provide a brief and non-exhaustive overview of the literature on large deviation probabilities—both lower and upper tails—for such models, with an emphasis on recent developments.

One of the foundational contributions is due to Logan and Shepp \cite{LS77}, who computed the equilibrium measure and rate function for Young tableaux under constraints. This corresponds to lower tail large deviations for the length of the longest increasing subsequence in a uniform random permutation. Subsequent works \cite{DZ99, Sep98a} established large deviation principles for the longest increasing subsequence and for the height function in the polynuclear growth model. Johansson \cite{Joh00} then used the Meixner ensemble to establish a large deviation principle for the integrated current in the Totally Asymmetric Simple Exclusion Process (TASEP).  

More recently, Corwin and Ghosal \cite{CG20} derived upper and lower bounds on the lower tail of the KPZ equation, capturing the qualitative behaviour of the rate function, though without sharp estimates. The exact lower tail large deviation was subsequently derived by Tsai \cite{Tsai22}. Further refinements appeared in \cite{CC22, CCR22}, where the Riemann--Hilbert method was employed to obtain asymptotic expansions of the lower tail. A variational approach was developed in \cite{DLM23, DLM24} to establish lower tail large deviations for discrete growth models, including the $q$-polynuclear growth model and the stochastic six-vertex model. In \cite{BGS21}, a speed-$N^2$ upper tail large deviation principle was proved for first passage percolation---corresponding to the lower tail regime in LPP---where the model is not exactly solvable and the rate function is characterised via a variational formula.

At the process level, the scaling limit of passage times was constructed in \cite{DOV22}. In addition, the large deviation behaviour of midpoint transversal fluctuations in general LPP models was analysed in \cite{ABGS25, AC21}.

Turning to the upper tail, exact large deviation principles for the one-point distribution of the KPZ equation and the ASEP with narrow wedge initial condition were established in \cite{DT21, DZ22}. A process-level upper tail large deviation principle for TASEP with general initial data was derived in \cite{QT21}, while the directed landscape---viewed as the full scaling limit of LPP---was analysed at the metric level in \cite{DDV24}, yielding upper tail large deviations. In the setting of exponential LPP, \cite{BBBK24} obtained optimal bounds for the one-point distribution in both upper and lower tail regimes, capturing the correct leading-order behaviour within the typical fluctuation scale. We also refer the reader to \cite{LMS16,CC2025,HMO25} and references therein for further literature.

\subsection*{Outlook of the paper} 
The rest of this paper is organised as follows.
\begin{itemize}
    \item In Section~\ref{Section_duality and results}, we present additional results and reformulate those from Section \ref{sec:Intro} using duality relations (Theorems~\ref{th:trunc-lpp} and~\ref{th:jue-lpp}). These dualities connect the probability $\mathbb{P}(G_{n,m} \leq \ell)$ with non-Hermitian and Hermitian random matrix theories. Specifically, $\mathbb{P}(G_{n,m} \leq \ell)$ can be expressed in terms of the characteristic polynomial of the truncated unitary ensemble (TUE), and in terms of large deviation probabilities for the Jacobi unitary ensemble (JUE). These duality formulas allow us to reformulate Theorems \ref{Thm_LPP weak}, \ref{th:lpp-lt} and \ref{th:lpp-ut} as results for the TUE (Theorems~\ref{Thm_TUE moment weak} and~\ref{Thm_TUE moment strong}) and the JUE (Theorems~\ref{Thm_JUE LDP weak} and~\ref{Thm_JUE LDP}). 
    \smallskip
    \item In Section~\ref{Section_dualities}, using Schur function expansion, we prove the duality relations (Theorems~\ref{th:trunc-lpp} and~\ref{th:jue-lpp}) that connect $\mathbb{P}(G_{n,m} \leq \ell)$ with both the JUE and the TUE.
    \smallskip
    \item The proofs of Theorems \ref{Thm_LPP weak}, \ref{th:lpp-lt}, \ref{th:lpp-ut}, \ref{Thm_TUE moment weak}, \ref{Thm_TUE moment strong}, \ref{Thm_JUE LDP weak}, \ref{Thm_JUE LDP} proceed in the other direction than the presentation: in  Section~\ref{Section_large deviation JUE}, we will use results from \cite{CG21,Fo12} to establish precise large deviation for the largest eigenvalue distribution of the JUE (Theorems~\ref{Thm_JUE LDP weak} and~\ref{Thm_JUE LDP}). The work \cite{CG21} is used for the lower tail, and \cite{Fo12} for the upper tail. We then combine these results with the duality formulas to obtain new results on the TUE (Theorems~\ref{Thm_TUE moment weak} and~\ref{Thm_TUE moment strong}) and on $\mathbb{P}(G_{n,m} \leq \ell)$ (Theorems \ref{Thm_LPP weak}, \ref{th:lpp-lt}, \ref{th:lpp-ut}). 
\end{itemize}

\medskip 

A summary of the interrelated models and corresponding results is provided in Table~\ref{Table_model summary}.

 \begin{table}[h!]
 	\begin{center}
 		{\def\arraystretch{2.5}
 			\begin{tabular}{|cc||c|c|c||}
 				\hline
 				\multicolumn{2}{|c||}{Model} &   \cellcolor{gray!10} \textbf{LPP} &  \cellcolor{gray!10} \textbf{TUE}    &  \cellcolor{gray!10}  \textbf{JUE} 
 				\\
 				\hline 
 				\multicolumn{2}{|c||}{Observables} &  \hspace{2.5em} $	\mathbb{P}(G_{n,m} \leq \ell)$ \hspace{2.5em} & \hspace{1.5em} $ \mathbb{E}  | \det ( \TT_{N,M} -z ) |^{2k}  $ \hspace{1.5em}   &        $ \mathbb{P}( x_{ \rm max }^{ (n) }(\lambda_1,\lambda_2)   < d )  $  
                \\
 					\hline \rule{0pt}{0.1cm} & & \multicolumn{3}{c||}{} \\[-0.9cm]
 					\multicolumn{2}{|c||}{\hspace{-0.1cm}\makecell{Dualities}}  & \multicolumn{3}{c||}{ \makecell[l]{
 Thm.~\ref{th:trunc-lpp} (LPP $\leftrightarrow$ TUE): $\mathbb{P}(G_{n,m} \leq \ell) = c_{\ell,n,m}(1-q^{2})^{nm}\mathbb{E}\left(|\det(  \mathsf{T}_{ \ell,\ell+n-m  } -q)|^{2m}\right)$ \\[0.2cm] Thm.~\ref{th:jue-lpp} (LPP $\leftrightarrow$ JUE): $\mathbb{P}(G_{n,m} \leq \ell) = \mathbb{P}(x^{(m)}_{\mathrm{max}}(n-m,\ell) \leq 1-q^2)$
}
 }      \\[0.3cm]
 				\hline\hline  \rule{0pt}{0.1cm} & & & & \\[-0.9cm]
 				%%%%%%%%%%%%%%%%%%%%%%%%% Weak non-unitarity regime %%%%%%%%%%%%%%%%%%%%%%%%%
% 				\multirow{3}{*}{\rotatebox[origin=c]{90}{ \emph{\CCC{Almost-s}quare lattice}  } } 	& Post-criticality &   \makecell{$n=m+O(1)$, \smallskip $ \delta > \omega(1,q)$ \\ Upper tail }   &    \makecell{$M-N=O(1)$, \smallskip \\ Droplet: full unit circle }      & \makecell{$\alpha=O(1)$, $\beta=O(n)$, \smallskip \\ Pulled regime}   \\
% 				\cline{2-5}
% 				& Pre-criticality &    \makecell{$n=m+O(1)$, \smallskip $  \delta < \omega(1,q)$ \\ Lower tail }  &  \makecell{$M-N=O(1)$, \smallskip \\ Droplet: arc of the unit circl\CCC{e}  
% 				}        &  \makecell{$\alpha=O(1)$, $\beta=O(n)$, \smallskip \\ Pushed regime}    \\
% 				\cline{2-5}  
 					& \hspace{-0.3cm}\makecell[l]{Lower tail } &  
\makecell[l]{
$\bullet$ Thm.~\ref{Thm_LPP weak}: 
\\
\hspace{0.5em} almost-square lattice 
\\
\hspace{0.5em} $n-m \asymp 1$ \\[0.1cm]
$\bullet$ Thm.~\ref{th:lpp-lt}:
\\
\hspace{0.5em} rectangular lattice 
\\
\hspace{0.5em} $n-m \asymp n$ 
}  &   
\makecell[l]{
\hspace{-0.5em}$\bullet$ Thm.~\ref{Thm_TUE moment weak} (b): 
\\
\hspace{0em} weak non-unitarity 
\\
\hspace{0em} $M-N \asymp 1$ \\[0.1cm]
\hspace{-0.8em} $\bullet$ Thm.~\ref{Thm_TUE moment strong} (b): 
\\
\hspace{0em} strong non-unitarity 
\\
\hspace{0em} $M-N \asymp n$ 
}   & 
\makecell[l]{
\hspace{-0.25em} $\bullet$ Thm.~\ref{Thm_JUE LDP weak} (b):  
\\
\hspace{0.5em} one hard, one soft edge
\\
\hspace{0.5em} $\lambda_{1}\asymp 1, \, \lambda_{2} \asymp n$ \\[0.1cm]
\hspace{-0.25em} $\bullet$ Thm.~\ref{Thm_JUE LDP} (b): \\
\hspace{0.5em} two soft edges 
\\
\hspace{0.5em} $\lambda_{1}\asymp n, \, \lambda_{2} \asymp n$ 
}     \\[0.55cm]
 				\cline{2-5}  
 				\hline \hline \rule{0pt}{0.1cm} & & & & \\[-0.9cm]
             	& \hspace{-0.3cm}\makecell[l]{Upper tail } &  
\makecell[l]{
\hspace{-1.5em}$\bullet$ Thm.~\ref{th:lpp-ut} 
}  &   
\makecell[l]{
\hspace{-0.25em}$\bullet$ Thm.~\ref{Thm_TUE moment weak} (a): 
\\
\hspace{0.25em} weak non-unitarity 
\\
\hspace{0.25em}  $M-N \asymp 1$ \\[0.1cm]
\hspace{-0.75em} $\bullet$ Thm.~\ref{Thm_TUE moment strong} (a): \\
\hspace{0.25em} strong non-unitarity 
\\
\hspace{0.25em}  $M-N \asymp n$ 
}   & 
\makecell[l]{
$\bullet$ Thm.~\ref{Thm_JUE LDP weak} (a):
\\
\hspace{0.5em} one hard, one soft edge
\\
\hspace{0.5em} $\lambda_{1}\asymp 1, \, \lambda_{2} \asymp n$ \\[0.1cm]
$\bullet$ Thm.~\ref{Thm_JUE LDP} (a):\\
\hspace{0.5em} two soft edges
\\
\hspace{0.5em} $\lambda_{1}\asymp n, \, \lambda_{2} \asymp n$ 
}     \\
 				\cline{2-5}  
 				\hline
 				%%%%%%%%%%%%%%%%%%%%%%%%% Weak non-unitarity regime %%%%%%%%%%%%%%%%%%%%%%%%%
 			\end{tabular}
 		}
 	\end{center}
 	\caption{Summary of our results.} \label{Table_model summary}
 \end{table}

\section{Dualities and further results} \label{Section_duality and results}

The underlying algebraic structures in integrable probability models sometimes lead to surprising connections between seemingly different models. These connections are exact identities which we will refer to as \emph{duality formulas}. In this work, we will find two new duality formulas relating the probability $\mathbb{P}(G_{n,m} \leq \ell)$ with two random matrix ensembles (see \eqref{momtrunc} and \eqref{lppjue}).

The approach in \cite{Joh00} to deriving \eqref{lowtail} and \eqref{uptail} is based on a duality relation between $\mathbb{P}(G_{n,m} \leq \ell)$ and the distribution of the largest particle in the \textit{Meixner ensemble}.
This ensemble is defined by a discrete joint probability mass function on tuples $(h_{1},\ldots,h_{m}) \in \mathbb{N}^{m}$ proportional to
\begin{equation}
\prod_{1 \leq i < j \leq m}(h_{i}-h_{j})^{2}\prod_{i=1}^{m}\binom{h_{i}+m-n}{h_{i}}q^{2h_i}. \label{meixpdf}
\end{equation}
Let $h^{(m)}_{\mathrm{max}}(n,q):=\max\{h_{1},\ldots,h_{m}\}$, where $h_{1},\ldots,h_{m}$ are drawn with respect to \eqref{meixpdf}. It then follows from \cite[Proposition 1.3]{Joh00} that for any $\ell\geq 0$, 
\begin{equation} 
\mathbb{P}(G_{n,m} \leq \ell) = \mathbb{P}(h^{(m)}_{\mathrm{max}}(n,q) \leq \ell+m-1). \label{meix-max-intro}
\end{equation}

In this paper we establish two additional duality formulas. We begin with the one relating $\mathbb{P}(G_{n,m} \leq \ell)$ with the TUE.

\subsection{Truncated unitary ensemble: moments of the characteristic polynomial}

Let $U$ be drawn from the \textit{Circular Unitary Ensemble} (CUE) \cite[Chapter 2]{Fo10}, i.e. $U$ is an $M \times M$ unitary matrix sampled with respect to the Haar measure on the unitary group. Let $\TT_{N,M}$ be the top-left $N \times N$ submatrix of $U$. We then say that $\TT_{N,M}$ is drawn from the \textit{Truncated Unitary Ensemble} (TUE) (see \cite{ZS2000} or \cite[Section 2.6]{BF25} for more background).   
The joint probability density of the eigenvalues $\{ \lambda_j \}_{j=1}^N$ of $\TT_{N,M}$ is given by
\begin{equation} \label{def of tUE}
\frac{1}{ Z_{N,M}^{ \rm trunc } }\prod_{1\le j<k \le N} |\lambda_j-\lambda_k|^2 \prod_{j=1}^N (1-|\lambda_j|^2)^{M-N-1} \,dA(\lambda_j), \qquad |\lambda_j| \le 1.  
\end{equation}
Here, $dA(z)=d^2z/\pi$ is the area measure, and the normalizing constant is given explicitly as
\begin{align}
\begin{split}
\label{def of ZNM trunc}
Z_{N,M}^{ \rm trunc } = N! \prod_{k=0}^{N-1} \frac{k! \, \Gamma(M-N)}{ \Gamma(M-N+k+1) } = N! \,  \Gamma(M-N)^N \, \frac{ G(N+1)G(M-N+1) }{ G(M+1) },  
\end{split}
\end{align} 
where $G$ is the Barnes $G$-function defined by 
\begin{equation} \label{def of Barnes G}
	G(z+1)=\Gamma(z)G(z),\qquad G(1)=1,
\end{equation}
see e.g. \cite[Section 5.17]{NIST}.  
Then we have the following connection.
\begin{thm}[\textbf{Duality between LPP and TUE}]
\label{th:trunc-lpp}
Consider the last passage time $G_{n,m}$ in a geometric environment with parameter $q^2$ and assume without loss of generality that $n \geq m$. Then, for any $\ell \in \N$,
\begin{equation}
 \mathbb{P}(G_{n,m} \leq \ell) = c_{\ell,n,m}(1-q^{2})^{nm}\mathbb{E}(|\det(  \mathsf{T}_{ \ell,\ell+n-m  } -q)|^{2m}), \label{momtrunc}
\end{equation} 
where 
\begin{equation}\label{clnm}
c_{\ell,n,m} := \prod_{j=1}^{m}\frac{(\ell+n-m+j-1)!(j-1)!}{(n-m+j-1)!(\ell +j-1)!} = \frac{G(\ell+1)G(m+1)G(\ell+n+1)G(n-m+1)}{G(\ell+m+1)G(\ell+n-m+1)G(n+1)}.
\end{equation}
\end{thm}

As a consequence of Theorem~\ref{th:trunc-lpp}, we can reformulate the result in the context of the asymptotic behaviour of the TUE. In studying TUE, one needs to distinguish between the following two regimes.

\begin{itemize}
	\item \textbf{Weak non-unitarity:} This regime corresponds to $N\to + \infty$ while $M-N$ remains bounded. We will use the notation
	\begin{equation} \label{def of MN weak non-uni}
	M-N= \ss, \qquad \ss \in \N_{>0}. 
	\end{equation} 
   In this regime,  most of the eigenvalues are in a $O(1/N)$ neighbourhood of the unit circle with high probability, see Figure~\ref{Fig_TUE} (A). 
	\smallskip 
	\item \textbf{Strong non-unitarity:} This regime corresponds to $N\to + \infty$ with  $M-N$ being proportional to $N$, and we therefore use the notation
	\begin{equation} \label{def of MN strong non-uni}
		M= (\rho+1)N+\tt, \qquad \rho > 0, \quad \tt  \in \mathbb{R}, \quad M\in \N.
	\end{equation}
   In this regime, the eigenvalues accumulate on the disc centred at $0$ and of radius $1/\sqrt{1+\rho}$, see Figure~\ref{Fig_TUE} (B). It is therefore convenient to rescale the eigenvalues as 
    \begin{equation} \label{def of rescaling}
    	z_j= \sqrt{1+\rho} \, \lambda_j.
    \end{equation}
     The droplet in the $z_{j}$-plane is the unit disc, and the associated equilibrium measure is given by
    \begin{equation}
    	\label{rho strong non-unitary}
      \frac{ \rho		(1+\rho) }{ (1+\rho-|z|^2)^2} \cdot \mathbbm{1}_{\{ |z|<1 \} }\, dA. 
    \end{equation} 
    Note that as $\rho \to \infty$, the measure \eqref{rho strong non-unitary} tends to $\mathbbm{1}_{\{ |z|<1 \} }\, dA$, which is the circular law.
\end{itemize}

\begin{figure}[t]
	\begin{subfigure}{0.35\textwidth}
	\begin{center}	
		\includegraphics[width=\textwidth]{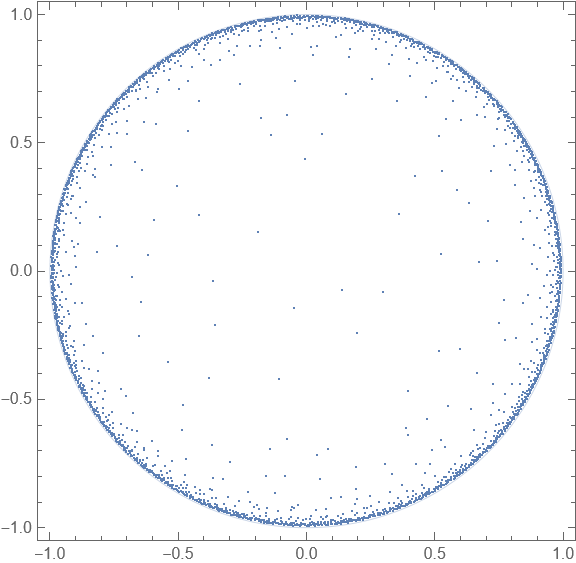}
	\end{center}
	\subcaption{$N=2560, M=2592$}
\end{subfigure}	 \qquad 
\begin{subfigure}{0.35\textwidth}
	\begin{center}	
		\includegraphics[width=\textwidth]{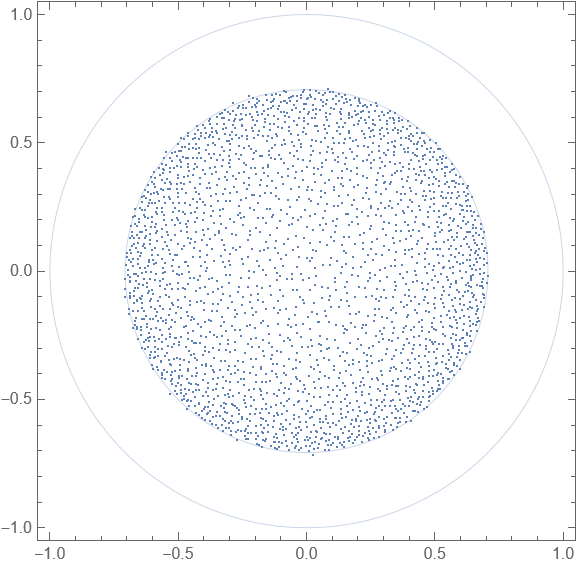}
	\end{center}
	\subcaption{$N=2560, M=5120$}
\end{subfigure}	
	\caption{In each plot, the points are the eigenvalues of a TUE matrix with the indicated values of $N$ and $M$. The thin curves are circles centred at $0$ of radii $1$ and $1/\sqrt{1+\rho}$.} \label{Fig_TUE}
\end{figure}

The asymptotic behaviour of moments of the absolute value of the characteristic polynomial of the Ginibre ensemble has been studied in \cite{WW19, DS22,Byun25} for the case where the exponent is fixed, and in \cite{BSY24} for the case where the exponent is of order $N$. More recently, the corresponding problem for the TUE was investigated in \cite{DMMS25} in the fixed exponent regime. In our subsequent results, we extend this analysis to the regime where the exponent is of order $N$.
In the context of two-dimensional Coulomb gas theory, such problems are referred to as point charge insertions. The associated planar orthogonal polynomials have been studied extensively \cite{BBLM15, KKL25, BGM17, BEG18, LY17, LY23, BFKL25, KLY25, BY23,BSY24}.

For each of the weak and strong non-unitary regimes, one needs to distinguish two regimes in the large $N$ behaviour of $\mathbb{E}  | \det (  \TT_{N,M} -z ) |^{2cN}$, depending on whether $|z|$ is smaller or larger than a critical value. The ``small" and ``large" $|z|$-regimes correspond to the upper and lower tail LPP probabilities, respectively, see also Table \ref{Table_model summary}. This phenomenon can also be interpreted as a topological phase transition of the associated droplet and will be studied in detail in the follow-up paper \cite{BCMSEquilibrium} (see also \cite{BY25, ABK21, CK22, BFL25} for various results on topological phase transitions).

\begin{thm}[\textbf{Moments of the characteristic polynomial of TUE matrices at weak non-unitarity}] \label{Thm_TUE moment weak}
	Suppose $M-N=: \ss \in \N$ is fixed.
	\begin{itemize}
		\item[(a)] \textbf{\textup{(Post-critical regime)}}  Suppose 
$0<|z|< \frac{ 1 }{2c+1}.$ Then as $N \to +\infty$, we have
\begin{align}
\begin{split}
 \mathbb{E}  | \det (  \TT_{N,M} -z ) |^{2cN}  = N^{ \frac{ \ss^2 }{2} }  \mathsf{G}^{ \rm post } \exp\Big( \mathsf{H}_{ 0 }^{ \rm post } N^2+ \mathsf{H}_{ 1 }^{ \rm post } N  \Big) \Big( 1+ \sum_{k=1}^\infty \frac{ d_k }{ N^k } +O( e^{ -\epsilon N } ) \Big),
 \end{split}
\end{align} 
for some $\epsilon >0$, where 
\begin{align*}
& \mathsf{G}^{ \rm post }  = \Big( \frac{c}{c+1} \Big) ^{ \frac{\ss^2}{2} } (2\pi)^{ \frac{ \ss }{2} } \frac{1}{ G(\ss+1) }, \\
& \mathsf{H}_{  0 }^{ \rm post } = -c^2 \log(1-|z|^2), \qquad  \mathsf{H}_{  1 }^{ \rm post } = \ss \log \Big( \frac{1}{(1-|z|^2)^ c } \frac{c^c}{ (c+1)^{c+1} } \Big). 
\end{align*} 
Here, $G$ is Barnes' $G$-function and the $d_k$'s are $N$-independent constants.
	   \smallskip 
		\item[(b)] \textbf{\textup{(Pre-critical regime)}} Suppose 
$\frac{ 1 }{2c+1}<|z|<1.$  Then as $N \to +\infty$, we have 
		\begin{align}
			\begin{split}
				\mathbb{E}  | \det ( \TT_{N,M} -z ) |^{2cN} 
				=   N^{ -\frac{1}{12}+ \frac{ \ss^2 }{2} }  e^{ \zeta'(-1) }  \mathsf{G}^{ \rm pre }  \exp\Big( \mathsf{H}_{  0 }^{ \rm pre } N^2+ \mathsf{H}_{ 1 }^{ \rm pre } N  \Big) \Big( 1+ O( \frac{1}{N} ) \Big),
			\end{split}
		\end{align} 
		where 	\begin{align*}
		 & \mathsf{G}^{ \rm pre }  =  \mathsf{G}^{ \rm post }  \Big(  \frac{ (  1+2c-1/|z|  )( 1+2c-|z| ) }{ 4c^2 }  \Big)^{ -\frac{ 1  }{8} }    \Big( \frac{(1+2c-|z|)(1+2c)}{4c(1+c)} \Big)^{ \frac{ \ss^2 }{2} }    \Big(\frac{ 2(1+c)^2 }{c( 1+2c ) } \Big)^{ \frac{1}{12} }, \\
        & \mathsf{H}_{ 0 }^{ \rm pre }  = -c^2 \log(1-|z|^2) + c^2 \,\mathsf{L}_1, \qquad  	\mathsf{H}_{ 1 }^{ \rm pre }  = \ss \log \Big( \frac{1}{(1-|z|^2)^ c } \frac{c^c}{ (c+1)^{c+1} } \Big) + c\, \mathsf{L}_2.
          \end{align*}
	   Here, $\mathsf{L}_1$ and $\mathsf{L}_2$ are as in Theorem~\ref{Thm_LPP weak} with $\mathsf{n}=\ss$, $\delta=1/c$ and $q=|z|$.  
	\end{itemize} 
\end{thm}

\begin{rem}
In the post-critical regime, the constants $d_k$ can be obtained by combining \eqref{momtrunc} and \eqref{clnm} with Theorem \ref{th:lpp-ut} and the asymptotic expansion of the Barnes \( G \)-function:
\begin{equation}  \label{Barnes G asymp}
\log G(z+1) \sim \frac{z^2 \log z}{2} -\frac34 z^2+\frac{ \log(2\pi) z}{2}-\frac{\log z}{12}+\zeta'(-1)  + \sum_{k=1}^\infty \frac{ B_{2k+2} }{ 4k(k+1)  } \frac{1}{z^{2k}} , \qquad \mbox{as } z\to + \infty
\end{equation}
where $B_k$'s are the Bernoulli numbers, see e.g. \cite[(5.17.5) and (5.11.1)]{NIST}. 
\end{rem}

Due in part to their connection with the Riemann zeta function \cite{KS00}, moments of the absolute value of the CUE characteristic polynomial have been extensively studied, see e.g. \cite{BCL2024,DIK,SF25,NSW20,SW24,AKW22} and the references therein. However, to our knowledge, Theorem~\ref{Thm_TUE moment weak} (b) is new even for the CUE, which corresponds to $\ss = 0$ (Theorem~\ref{Thm_TUE moment weak} (a) with $\ss=0$ can be deduced from \cite{BCL2024}, see Remark \ref{remark:consistency CUE blackstone} below). For convenience, the case $\ss=0$ is stated explicitly in the following corollary.

\begin{cor}[\textbf{Moments of the characteristic polynomial of CUE}] \label{Thm_CUE moment weak} \;
	\begin{itemize}
		\item[(a)] \textbf{\textup{(Post-critical regime)}} Suppose 
$0<|z|< \frac{ 1 }{2c+1}.$  Then as $N \to +\infty$, we have
		\begin{align}
			\begin{split}
			\mathbb{E}  | \det (  \TT_{N,N} -z ) |^{2cN} 
				  =  (1-|z|^2)^{-c^2N^2} \Big( 1+ \sum_{k=1}^\infty \frac{ d_k }{ N^k } +O( e^{ -\epsilon N } ) \Big), 
			\end{split}
		\end{align} 
	for some $\epsilon >0$ and where the $d_k$'s are $N$-independent constants.
		\smallskip 
		\item[(b)] \textbf{\textup{(Pre-critical regime)}}  Suppose 
$\frac{ 1 }{2c+1}<|z|<1.$  Then as $N \to +\infty$, we have 
		\begin{align}
			\begin{split}
				\mathbb{E}  | \det ( \TT_{N,N} -z ) |^{2cN} 
				&=   N^{ -\frac{1}{12}  }  e^{ \zeta'(-1) }   (1-|z|^2)^{-c^2N^2} e^{ c^2 \, \mathsf{L}_1 N^2   }   
				\\
				&\quad \times  \Big(  \frac{ (  1+2c-1/|z|  )( 1+2c-|z| ) }{ 4c^2 }  \Big)^{ -\frac{ 1  }{8} }       \Big(\frac{ 2(1+c)^2 }{c( 1+2c ) } \Big)^{ \frac{1}{12} }  \Big( 1+ O( \frac{1}{N} ) \Big). 
			\end{split}
		\end{align} 
	 Here, $\mathsf{L}_1$ is as in Theorem~\ref{Thm_LPP weak} with $\delta=1/c$ and $q=|z|$. 
	\end{itemize} 
\end{cor}

\begin{rem}(Consistency check between Corollary \ref{Thm_CUE moment weak} (a) and \cite[Theorem 1.1]{BCL2024})\label{remark:consistency CUE blackstone}
In \cite[Theorem 1.1 (a)]{BCMSEquilibrium}, we show that the equilibrium measure associated with the potential \( V(\zeta) = -2c \log |\zeta - z| \) is given by
\begin{equation} \label{def of eq msr 1D post}
			d\mu_{ \mathsf{W} }(\theta) = \frac{1}{2\pi}  \frac{ 1+(1+2c)z^2-2(1+c)z \cos \theta }{ 1+z^2-2z \cos\theta  }   \mathbbm{1}_{ (-\pi,\pi] }(\theta) \,d\theta . 
		\end{equation}
On the other hand, \cite[Theorem 1.1]{BCL2024} implies for $0< |z| < \frac{ 1 }{2c+1}$ that
\begin{align*}
\mathbb{E}  | \det (  \TT_{N,N} -z ) |^{2cN} = \exp \bigg(  \bigg[ -\frac{1}{2}\int_{-\pi}^{\pi} V(e^{i\theta})\frac{d\theta}{2\pi} -\frac{1}{2}\int_{-\pi}^{\pi} V(e^{i\theta})d\mu_{ \mathsf{W} }(\theta) \bigg] n^{2} + O(n^{-1}) \bigg).
\end{align*}
A direct residue computation shows that $\int_{-\pi}^{\pi} V(e^{i\theta})\frac{d\theta}{2\pi}=0$ and that 
\begin{align*}
-\frac{1}{2}\int_{-\pi}^{\pi} V(e^{i\theta}) \, d\mu_{ \mathsf{W} }(\theta) = -c^{2}\log(1-|z|^{2}),
\end{align*}
which is consistent with Corollary \ref{Thm_CUE moment weak} (a).
\end{rem}

Next, we consider the asymptotic expansion in the strong non-unitarity regime. 

\begin{thm}[\textbf{Moments of the characteristic polynomial of TUE matrices at strong non-unitarity}] \label{Thm_TUE moment strong}
Let $\rho>0$ be fixed, and let $\tt=\tt_{N}  \in \mathbb{R}$ be bounded and such that $M= (\rho+1)N+\tt \in \N$ for $N\in \N$.
	\begin{itemize}
		\item[(a)] \textbf{\textup{(Post-critical regime)}} Suppose 
\begin{align*}
0<\frac{|z|}{\sqrt{1+\rho}}< \dfrac{  \sqrt{(\rho+c+1)(c+1)}-\sqrt{(\rho+c)c} }{\rho+2c+1} .
\end{align*}	
Then as $N\to + \infty$, we have
\begin{equation}
	 	\mathbb{E}  | \det ( \sqrt{1+\rho}\, \TT_{N,M} -z ) |^{2cN} 
= N^{ \frac{1}{12} } e^{ -\zeta'(-1) } \mathcal{G}^{ \rm post }  \exp\Big( \mathcal{H}_0^{ \rm post } N^2+ \mathcal{H}_1^{ \rm post } N  \Big) \Big( 1+ \sum_{k=1}^\infty \frac{ b_k }{N^k} + O(e^{-\epsilon N})  \Big)  ,
\end{equation}
		for some $\epsilon >0$, where 
        \begin{align*}
			\mathcal{G}^{ \rm post } & = \Big( \frac{c}{c+1}\Big)^{ \frac{1}{12} }   \Big( \frac{\rho}{\rho+1} \Big)^{ -\frac{6\tt^2-1}{12} } \Big( \frac{(\rho+1)(\rho+c)}{\rho(\rho+c+1) } \Big)^{ \frac{6\tt^2-1}{12} }, \\
			\mathcal{H}_0^{ \rm post } & = - c (\rho+c) \log\Big( 1- \frac{|z|^2}{1+\rho} \Big) + c \log (1+\rho) 
			\\
			&\quad +\frac12 \bigg(  \log\Big( \frac{(c+1)^{(c+1)^2} }{ c^{c^2} }\Big)  +  \log\Big( \frac{(\rho+1)^{(\rho+1)^2} }{ \rho^{\rho^2} }\Big) - \log\Big( \frac{ (c+\rho+1)^{ (c+\rho+1)^2 } }{ (c+\rho)^{ (c+\rho)^2 } } \Big) \bigg),
			\\
			\mathcal{H}_1^{ \rm post } &= -\tt \bigg(  c \log\Big( 1- \frac{|z|^2}{1+\rho} \Big) +  \log\Big( \frac{ \rho^{ \rho } }{ (\rho+1)^{ \rho+1 } } \Big) -\log\Big( \frac{ (\rho+c)^{ \rho+c } }{ (\rho+c+1)^{ \rho+c+1 } } \Big) \bigg) .   
		\end{align*}
        Here, the $b_k$'s are $N$-independent constants.
         \smallskip 
		\item[(b)] \textbf{\textup{(Pre-critical regime)}}  Suppose 
\begin{align*}
  \dfrac{  \sqrt{(\rho+c+1)(c+1)}-\sqrt{(\rho+c)c} }{\rho+2c+1} < \frac{|z|}{\sqrt{1+\rho}} <1 .
\end{align*}
Then as $N\to + \infty$, we have 
			\begin{align}
			\begin{split}
			\mathbb{E}  | \det ( \sqrt{1+\rho}\, \TT_{N,M} -z ) |^{2cN} 
				 = 	\mathcal{G}^{ \rm pre } \exp\Big( \mathcal{H}_0^{ \rm pre } N^2+ \mathcal{H}_1^{ \rm pre } N  + O(\frac{1}{N}) \Big)  ,
			\end{split}
		\end{align} 
	     where 
\begin{align*}
\mathcal{G}^{ \rm pre } = c^{ -\frac{1}{12} } \mathcal{G}^{ \rm post } e^{ L_4 }, \qquad  \mathcal{H}_0^{ \rm pre } = \mathcal{H}_0^{ \rm post } +c^2\,L_1,\qquad  \mathcal{H}_1^{ \rm pre } = \mathcal{H}_1^{ \rm post }  +c\,L_2,
\end{align*}
and $L_1,L_2,L_4$ are as in Theorem~\ref{th:lpp-lt} with 
          \begin{equation}
        \gamma=\frac{\rho}{c}+1, \qquad  \mathsf{n}=\tt, \qquad \delta= \frac{1}{c}, \qquad q=\frac{|z|}{\sqrt{1+\rho}}. 
        \end{equation}
	\end{itemize} 
\end{thm}

\begin{rem}(Consistency check between Theorem~\ref{Thm_TUE moment strong} (a) and  \cite[Theorem 2.7]{BSY24})
Theorem \ref{Thm_TUE moment strong} is valid for $\rho$ fixed. Nevertheless, by taking formally the limit $\rho \to \infty$ in Theorem~\ref{Thm_TUE moment strong} (a), one recovers the result \cite[Theorem 2.7]{BSY24} about the characteristic polynomials of the Ginibre ensemble. Indeed, we have
	\begin{align*}
		\lim_{ \rho \to \infty } \mathcal{H}_0^{ \rm post }= c|z|^2-\frac{3c}{2} +\frac12   \log\Big( \frac{(c+1)^{(c+1)^2} }{ c^{c^2} }\Big), \qquad \lim_{ \rho \to \infty } \mathcal{H}_1^{ \rm post }=0, \qquad \lim_{ \rho \to \infty } \mathcal{G}^{  \rm post  } = \Big(\frac{c}{c+1}\Big)^{ \frac{1}{12} },
	\end{align*}
which is consistent with \cite[Theorem 2.7]{BSY24}.
\end{rem}

\subsection{Jacobi unitary ensemble: large deviation probabilities of the largest eigenvalues}

We now present our second duality formula, which relates $\mathbb{P}(G_{n,m} \leq \ell)$ with the JUE. 
To formulate this, let $p$ and $n$ be two integers such that $p\geq n$ and let $X$ be an $n \times p$ matrix whose entries are independent complex normal random variables $X_{ij} \sim N_{\mathbb{C}}(0,1)$. The matrix $W = XX^{*}$ is called a complex Wishart matrix with $p$ degrees of freedom. (In fact, it is well known that the largest eigenvalue of $W$ is closely related to the LPP model, but in the case where the weights in \eqref{weights} are exponential rather than geometric \cite[Proposition 1.4]{Joh00}.) For the model with geometric weights studied here, we consider the random matrix 
\begin{equation} \label{def of JUE matrix}
J := W_{1}(W_{1}+W_{2})^{-1},
\end{equation}
where $W_{1}$ and $W_{2}$ are independent Wishart matrices with $p_{1}$ and $p_{2}$ degrees of freedom, respectively. The $n$ eigenvalues $x_{1},\ldots,x_{n}$ of $J$ all lie in the interval $[0,1]$ and have the joint probability density function
\begin{equation}  
P(x_{1},\ldots,x_{n}) = \frac{1}{Z_{n}(\lambda_1,\lambda_2)}\,\prod_{1 \leq i < j \leq n}(x_{j}-x_{i})^{2} \prod_{j=1}^{n}x_{j}^{\lambda_{1}}(1-x_{j})^{\lambda_{2}} ,  \label{juepdf}
\end{equation}
where $\lambda_{1} = p_{1}-n$ and $\lambda_{2} = p_{2}-n$ (see e.g. \cite[Chapter 3]{Fo10}), and $Z_{n}(\lambda_1,\lambda_2)$ is a Selberg integral (see e.g. \cite[Eq.~(1.1)]{FW08}) that can be written as 
\begin{align} \label{def of JUE partition function}
\begin{split}
Z_n(\lambda_1,\lambda_2) &=  \prod_{ j=0 }^{n-1}   \frac{ (j+1)!\, \Gamma( \lambda_1+1+j   )  \Gamma( \lambda_2+1+j   )    }{ \Gamma(  \lambda_1+\lambda_2 +1 +  n+j   )    }
\\
&= G(n+2) \frac{ G( \lambda_1+n+1 ) }{ G(\lambda_1+1) } \frac{ G( \lambda_2+n+1 ) }{ G(\lambda_2+1) } \frac{  G( \lambda_1+\lambda_2+n+1 )  }{ G( \lambda_1+\lambda_2+2n+1 ) },
\end{split}
\end{align}
where $G$ is the Barnes $G$-function \eqref{def of Barnes G}. 
The probability density \eqref{juepdf} is known as the \textit{Jacobi Unitary Ensemble} (JUE) with $n$ particles and parameters $\lambda_{1}, \lambda_{2}$. 
Let $$x^{(n)}_{\mathrm{max}}(\lambda_1,\lambda_2) = \mathrm{max}\{x_{i}\}_{i=1}^{n},$$ where $x_{1},\ldots,x_{n}$ are drawn with respect to \eqref{juepdf}. The following is our second duality formula.

\begin{thm}[\textbf{Duality between the LPP and JUE}]
\label{th:jue-lpp}
Consider the last passage time $G_{n,m}$ in a geometric environment with parameter $q^{2}$ and assume without loss of generality that $n \geq m$. Then, for any $\ell \in \N$,
\begin{equation}
\mathbb{P}(G_{n,m} \leq \ell) = \mathbb{P}(x^{(m)}_{\mathrm{max}}(n-m,\ell) \leq 1-q^2). \label{lppjue}
\end{equation} 
\end{thm}
The advantage of formula \eqref{lppjue} is that, as previously mentioned, the right-hand side can be evaluated asymptotically using results from \cite{CG21, Fo12}. 

We mention that large deviation probabilities in Hermitian random matrix ensembles (such as the JUE) have been extensively studied using several approaches, including the Coulomb gas method \cite{DM08}, loop equations and $\frac{1}{N}$ expansions \cite{BEMN11,BG13} and orthogonal polynomials \cite{NM11,CFWW25}. Other such results, for one-dimensional log-gases, can be found in e.g. \cite{BDG01, DM06, DM08, AGKWW14, KC10, RKC12, PS16, VMB07, WKG15, MV09, Fo12, FW12, MS14}; most of these works are based on the Coulomb gas method and primarily focus on the leading-order asymptotics of such probabilities, see also \cite{BF25a}.  

For the JUE, the leading order term of the lower tail probability (of order $n^{2}$) was obtained in \cite{RKC12} via the Coulomb gas method. In \cite{CG21}, formulas for all the terms up to and including the constant term were derived for a wide class of potentials, with some of these terms expressed as integrals against an equilibrium measure. By evaluating these integrals explicitly and applying Theorem \ref{th:jue-lpp}, we obtain the closed-form expressions stated in Theorem \ref{th:lpp-lt}.  

For the upper tail probability, the leading order term (of order $n$) was obtained in \cite{RKC12}. The results in \cite{CG21} do not apply to the upper tail probability; hence, to derive subleading terms, we instead use results from \cite{Fo12} on the eigenvalue density of the JUE outside the support of the equilibrium measure (see also \cite[Eq. (14.136)]{Fo10}).  

When the parameters $\lambda_1$ and $\lambda_2$ satisfy $\lambda_1/n \to \alpha \ge 0$ and $\lambda_2/n \to \beta \ge 0$, the empirical measure $\frac{1}{n}\sum_{j=1}^{n}\delta_{x_{j}}$ of the JUE converges almost surely, as $n\to + \infty$, to the deterministic measure 
\begin{equation} \label{def of Wachter distribution}
d\mu_{ \rm W }(x):=\frac{ \alpha+\beta+2  }{2\pi} \frac{ \sqrt{ (b-x)(x-a) } }{ x(1-x) } \,dx, \qquad x\in[a,b],  
\end{equation}
where $a,b\in [0,1]$ are given by \eqref{x pm lpp} with $\alpha=\gamma-1$ and $\beta=\delta$. 
The distribution \eqref{def of Wachter distribution} was obtained in \cite{Wach80}, and is sometimes called the Wachter distribution, see also \cite[Corollary 4.1]{Co05} and \cite{MSV97} where this distribution appears in  different contexts.

\begin{figure}[t]
	\begin{subfigure}{0.32\textwidth}
	\begin{center}	
		\includegraphics[width=\textwidth]{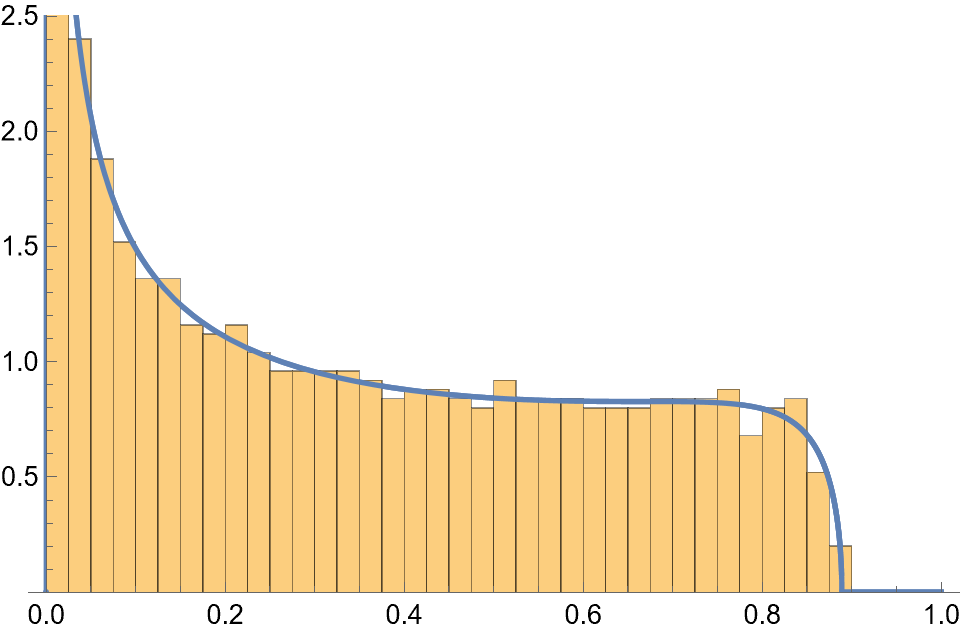}
	\end{center}
	\subcaption{$\alpha=0, \beta=1$}
\end{subfigure}	\qquad
\begin{subfigure}{0.32\textwidth}
	\begin{center}	
		\includegraphics[width=\textwidth]{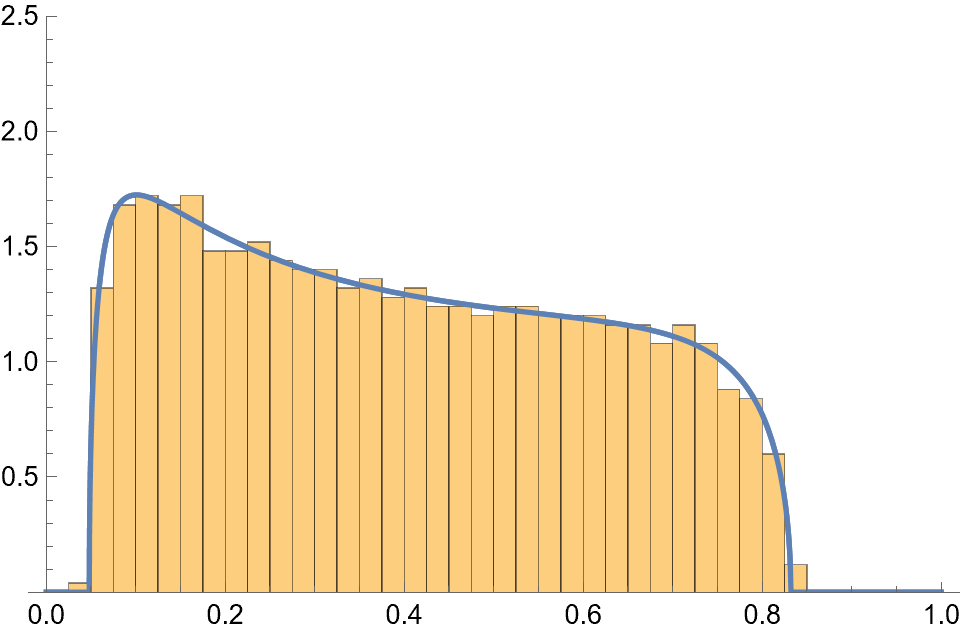}
	\end{center}
	\subcaption{$\alpha=1, \beta=2$}
\end{subfigure}	 
	\caption{Histograms of the eigenvalues of a JUE matrix \eqref{def of JUE matrix} of size \( n = 1000 \), for various values of \( \alpha \) and \( \beta \). The corresponding limiting densities \eqref{def of Wachter distribution} are the solid curves.} \label{Fig_JUE}
\end{figure} 
 
By Theorem~\ref{th:jue-lpp}, the almost-square and rectangular regimes in the LPP problem correspond, for the JUE, to $m\to + \infty$ and $\lambda_1 = n - m$ being of order $1$ or proportional to $m$, respectively. In the first case, $\alpha=0$ and $d\mu_{\rm W}$ has one hard edge and one soft edge; in the second case, $\alpha>0$ and $d\mu_{\rm W}$ has two soft edges, see also Figure~\ref{Fig_JUE}.

\begin{thm}[\textbf{Large deviation probabilities for the JUE with one hard edge and one soft edge}] \label{Thm_JUE LDP weak}
	Let $\alpha>-1$ and $\beta>0$ be fixed, and let 
    \begin{equation}
    b=\frac{4(1+\beta)}{(2+\beta)^{2}}.
    \end{equation} 
   (This is same as the definition of $b$ in \eqref{x pm lpp}, with $\gamma=1$ and $\delta=\beta$.)
	\begin{itemize}
		\item[(a)] \textup{\textbf{(Pulled regime)}}  If $d \in (b,1)$, then
		\begin{align}
			\mathbb{P} \Big( x_{ \rm max }^{ (n) }(\alpha,\beta n)  > d \Big) = \exp\bigg(U_1 n + U_2 \log n +U_3 +o(1) \bigg), \label{largedev-a N}
		\end{align}
        as $n \to \infty$. Here, $U_k$'s ($k=1,2,3$) are as in Theorem~\ref{th:lpp-ut} with $\gamma=1$ and 
\begin{equation} \label{choice of parameters for JUE LPP square}
\mathsf{n}=\alpha, \qquad \delta=\beta, \qquad q=\sqrt{1-d}. 
\end{equation} 
		\item[(b)] \textup{\textbf{(Pushed regime)}} If $d\in (0,b)$, then
		\begin{align}\label{lol4 N}
			\mathbb{P} \Big( x_{ \rm max }^{ (n) }(\alpha,\beta n) < d \Big) = \exp\bigg( \mathsf{L}_1 n^{2} + \mathsf{L}_2 n + \mathsf{L}_3 \log n + \mathsf{L}_4 + O \Big( \frac{\log n}{n } \Big)\bigg), 
		\end{align}
        as $n \to \infty.$
     Here, $\mathsf{L}_k$'s ($k=1,2,3,4$) are as in Theorem~\ref{Thm_LPP weak} with \eqref{choice of parameters for JUE LPP square}. 
	\end{itemize}
\end{thm}

\begin{thm}[\textbf{Large deviation probabilities for the JUE with two soft edges}] \label{Thm_JUE LDP}
	Let $\alpha,\beta>0$ and $t\in \R$ be fixed.  
    Let 
    \begin{equation}
     b =\Big( \frac{  \sqrt{(\alpha+1)(\alpha+\beta+1) } +\sqrt{\beta+1}   }{\alpha+\beta+2} \Big)^2. 
    \end{equation}  
    (This is same as the definition of $b$ in \eqref{x pm lpp}, with $\gamma=\alpha+1$ and $\delta=\beta$.)
	\begin{itemize}
		\item[(a)] \textup{\textbf{(Pulled regime)}}  If $d \in (b,1)$, then
		\begin{align} 
			\mathbb{P} \Big( x_{ \rm max }^{(n)}( \alpha n+t,\beta n ) > d \Big)  = \exp\bigg(U_1 n + U_2 \log n +U_3 +o(1) \bigg), \label{largedev-a}
		\end{align}
	   as $n \to \infty$. Here, $U_k$'s ($k=1,2,3$) are as in Theorem~\ref{th:lpp-ut} with \begin{equation} \label{choice of parameters for JUE LPP rectangle}
        \gamma=\alpha+1, \qquad \mathsf{n} =t, \qquad \delta= \beta, \qquad q=\sqrt{1-d}. 
        \end{equation} 
		\item[(b)] \textup{\textbf{(Pushed regime)}} If $d\in (0,b)$, then
		\begin{align}\label{lol4}
			\mathbb{P} \Big( x_{ \rm max }^{(n)}( \alpha n+t,\beta n ) < d \Big) = \exp\bigg( L_1 n^{2} + L_2 n + L_3 \log n + L_4 + O \Big( \frac{ \log n }{n } \Big)\bigg),  
		\end{align}
        as $n\to \infty.$
    Here, $L_k$'s ($k=1,2,3,4$) are as in Theorem~\ref{th:lpp-lt} with \eqref{choice of parameters for JUE LPP rectangle}. 
	\end{itemize}
\end{thm}

\begin{rem}(Consistency check between Theorem~\ref{Thm_JUE LDP} (b) and \cite{RKC12}) 
By \cite[Subsection 4.3]{RKC12}, as $n\to+\infty$ we have
\begin{equation}
\mathbb{P} \Big( x_{ \rm max }^{(n)}( \alpha n+t,\beta n ) < d \Big) =	\exp \Big( -n^2 ( S( \mathfrak{m} , d )-S( a,  b ) ) +o(n^2)\Big),
\end{equation}
where $a$ and $b$ are given by \eqref{x pm lpp} with the substitution \eqref{choice of parameters for JUE LPP rectangle}, and 
\begin{align}
	\begin{split} \label{def of S(x,y)}
		S(x,y) &=- \log\frac{y-x}{4}  + \alpha \beta \log \frac{ \sqrt{x(1-y)}+\sqrt{y(1-x)} }{ 2}+ \frac{\alpha^2}{4}\log(xy) + \frac{\beta^2}{4}\log\Big((1-x)(1-y) \Big)   
        \\
        &\quad - ( 2+\alpha+\beta ) \Big( \alpha \log \frac{\sqrt{x}+\sqrt{y}}{ 2 } +  \beta \log \frac{\sqrt{1-x}+\sqrt{1-y}}{ 2 }  \Big). 
	\end{split}
\end{align}
It is straightforward to check this formula is consistent with \eqref{lol4}, namely 
\begin{equation}
L_1 = -( S( \mathfrak{m} , d )-S( a,  b ) ), 
\end{equation}
see also \eqref{widehat D1 evaluation} below. 
\end{rem}

\subsection{Related works}

\subsubsection{CUE averages}

A well-known duality formula expresses the largest eigenvalue distribution of the JUE in terms of certain CUE averages (as studied in \cite{FW01, FW02, FW04, BR01}), see also Lemmas \ref{lem:cuetojue} and \ref{lem:cuejue-duality} below. Consequently, the right-hand side of \eqref{lppjue} can be written in terms of these CUE averages, see also \cite{Fo25}. This connection will be further discussed in Subsection~\ref{Subsec_JUE CUE}.

\subsubsection{Free energy expansion of two-dimensional Coulomb gas}

Theorems~\ref{Thm_TUE moment weak} and~\ref{Thm_TUE moment strong} admit a reformulation as precise free energy expansion of certain two-dimensional Coulomb gases. Indeed, by \eqref{def of tUE}, the partition function
\begin{equation}
Z_{N}:=  \int_{ \mathbb{C}^N } \prod_{1\le j<k \le N} |\lambda_j-\lambda_k|^2 \prod_{j=1}^N (1-|\lambda_j|^2)^{M-N-1}  |\lambda_j-z|^{2cN}  \mathbbm{1}_{ \{ |\lambda_j| \le 1 \} } \,dA(\lambda_j), \qquad |z|<1,
\end{equation}
can be rewritten as
\begin{equation}
Z_{N} = Z_{N,M}^{ \rm trunc }\, \mathbb{E}  | \det (  \TT_{N,M} -z ) |^{2cN}  . 
\end{equation}
Therefore, Theorems~\ref{Thm_TUE moment weak} and~\ref{Thm_TUE moment strong} directly yield an expansion of the free energy $\log Z_N$.

Recent advances in free energy expansion for general temperature include \cite{LS17, BBNY19}. The best known error estimates were obtained in \cite{Se23} for general equilibrium measures, and in \cite{AS21} when the equilibrium measure is uniform. An important aspect of this expansion is that it captures the underlying geometric structure of the Coulomb gas \cite{JMP94, TF99, ZW06, CFTW15}. In the determinantal case, more precise results are available \cite{Ch22,Ch23,BP24,AFLS25, BSY24, BKS23, BKSY25, ACC23, ACCL22, AL2025, Noda2025, Roug2025,Byun25}. In particular, the work \cite{BSY24} treats a situation where the potential is non-radially symmetric and the equilibrium measure is uniform. Our results provide the first example of a precise free energy expansion in a situation where the potential is \textit{non-radially symmetric} and the limiting density is \textit{non-uniform}. 

\subsubsection{Moments of the characteristic polynomial of induced spherical ensembles}

In addition to the TUE model, there exists an equivalent formulation within non-Hermitian random matrix theory known as the \emph{induced spherical ensemble} \cite[Section 2.5]{BF25}. From a plasma perspective, the TUE can be viewed as an extension of the Ginibre ensemble to a space of constant negative curvature called the pseudosphere, see \cite[Section 15.7]{Fo10}. In contrast, the induced spherical ensemble corresponds to an extension of the Ginibre ensemble on a surface of constant positive curvature, namely, the Riemann sphere.

In \cite[Proposition 5.4]{Fo25}, it was shown that the moments of the characteristic polynomial of the induced spherical ensemble also exhibit a duality with the large deviation probabilities of the JUE. Hence, this seemingly distinct model turns out to be equivalent to the problem studied here. From the two-dimensional Coulomb gas perspective, this corresponds to a gas of equally charged particles on the Riemann sphere interacting with two external point charges. The associated equilibrium measure and leading-order energy---equivalent to the dominant term in Theorem~\ref{Thm_TUE moment strong}---were recently obtained in \cite{BFL25}. The corresponding planar orthogonal polynomials were analysed in \cite{BFKL25} using the Riemann--Hilbert method.

\section{Duality relations} \label{Section_dualities}

In this section we prove Theorems~\ref{th:trunc-lpp} and ~\ref{th:jue-lpp}. We start from the expression \eqref{gmnsum} below that arises directly after applying the Robinson–Schensted–Knuth (RSK) correspondence to the LPP problem. This is the expression Johansson used in \cite{Joh00} to write the distribution of the last passage time in terms of the Meixner ensemble, as we shall discuss below. Here we show how manipulating this expression in an alternative way leads to the JUE and to the TUE.

\subsection{Duality formulas relating LPP, TUE and JUE}
We begin with an introduction to partitions and Schur functions, see \cite{Macdonald} for a standard reference. A partition $\lambda$ is a weakly decreasing sequence of non-negative integers $\lambda = ( \lambda_1, \lambda_2,\ldots, \lambda_m)$. The number of non-zero terms in such a sequence is called the length of the partition and is denoted as $l(\lambda)$. The weight of the partition is $|\lambda| = \sum_{i=1}^{m} \lambda_i$. The Schur functions $s_{\lambda}(\bm x)$ in $m$ variables $x_{1},\ldots,x_{m}$ are symmetric polynomials defined as
\begin{equation}
s_{\lambda}(\bm x) = \frac{1}{\Delta(\bm x)} \det\Big\{x_{i}^{\lambda_{j}+m-j}\Big\}_{i,j=1}^{m}, \qquad \Delta(\bm x) = \prod_{1 \leq i < j \leq m}(x_{i}-x_{j}). \label{sdef}
\end{equation}
Note that $\Delta(\bm x)$ differs from the classical definition of a Vandermonde determinant by a factor $(-1)^{\frac{m(m-1)}{2}}$.
If a partition $\lambda$ has length $l(\lambda) > m$ then $s_{\lambda}(x_{1},\ldots,x_{m})=0$. If $l(\lambda) < m$ then \eqref{sdef} continues to be a valid definition by appending an appropriate number of zeros to the end of the partition. The Schur functions are homogeneous in the sense that if $q$ is a scalar, then $s_{\lambda}(q \bm x) = q^{|\lambda|}s_{\lambda}(\bm x)$. Partitions are represented by their Young diagram consisting of $m$ rows of boxes, with each row starting on the left with $\lambda_{j}$ boxes in row $j$, for $j=1,\ldots, m$. The total number of boxes in the diagram is the weight $|\lambda|$. The conjugate of $\lambda$, denoted $\lambda'$, is the partition obtained by transposing the Young diagram, that is row $j$ of $\lambda'$ is built from column $j$ of $\lambda$ and vice versa. For vectors $\bm x \in \mathbb{C}^{m}$ and $\bm y \in \mathbb{C}^{\ell}$, the dual Cauchy identity states that (see e.g. \cite[Chapter 1, Eq. (4.3$^\prime$)]{Macdonald} or \cite[Exercise 10.1.4]{Fo10})
\begin{align} 
\prod_{j=1}^{m}\prod_{i=1}^{\ell}(1+x_{j}y_{i}) = \sum_{\lambda, l(\lambda) \leq m, \lambda_{1} \leq \ell} s_{\lambda}(\bm x) s_{\lambda'}(\bm y). \label{dci}
\end{align}
Note the appearance of the conjugate partition $\lambda'$ on the right-hand side which has length at most $\ell$ and first coordinate at most $m$. We use the notation $s_{\lambda}(1_{m}) = s_{\lambda}(1,\ldots,1)$, i.e. $1_{m}$ has all $m$ coordinates equal to $1$. The following theorem presents the known expressions for the probability distribution of $G_{n,m}$ due to \cite{Joh00,BR01,FW04}.
\begin{thm}
\label{thm:meix}
Assume $n \geq m$ and that $n,m,\ell \in \N$. Suppose $q\in (0,1)$, and let $G_{n,m}$ be the last passage time in geometric environment with parameter $q^{2}$. The following formulas express $\mathbb{P}(G_{n,m} \leq \ell)$ in different ways:
\begin{itemize}
\item A Schur function expansion \cite[Eqs. (2.2)--(2.5)]{Joh00}
\begin{equation}
\mathbb{P}(G_{n,m} \leq \ell) = (1-q^{2})^{nm}\sum_{\lambda, l(\lambda) \leq m, \lambda_{1} \leq \ell}s_{\lambda}(1_{n})s_{\lambda}(1_{m})q^{2|\lambda|}. \label{gmnsum}
\end{equation}
\item The largest particle distribution in the Meixner ensemble \cite[Proof of Proposition 1.3]{Joh00}
\begin{equation}
\mathbb{P}(G_{n,m} \leq \ell)  = \mathbb{P}(h^{(m)}_{\mathrm{max}}(n,q) \leq \ell+m-1). \label{lpmeix}
\end{equation}  
\item An average over Haar distributed unitary matrices $U$ of size $\ell \times \ell$ \cite[Eq.(5.53)]{FW04} (see also \cite{BR01})
\begin{equation}
 \mathbb{P}(G_{n,m} \leq \ell) = (1-q^{2})^{nm}\mathbb{E}(\det(1+U)^{n}\det(1+q^{2}U^{\dagger})^{m}). \label{cuebr01}
\end{equation}
\end{itemize}
\end{thm}
We do not give the full proof of the above theorem here but we give some remarks. The Schur function expansion in \eqref{gmnsum} is a consequence of applying the RSK correspondence to geometric LPP \cite[Eq. (2.4)]{Joh00}. To derive \eqref{lpmeix}, Johansson rewrites \eqref{gmnsum} using the identity
\begin{equation}
s_{\lambda}(1_{n}) = \prod_{1 \leq i < j \leq n}\frac{\lambda_{i}-\lambda_{j}+j-i}{j-i}
\end{equation}
and introduces new variables $h_{j} = \lambda_{j}+m-j$ for $j=1,\ldots,m$. Assume that $n \geq m$ (the other case is analogous by symmetry).
This leads to
\begin{equation}
\begin{split}
s_{\lambda}(1_{m})s_{\lambda}(1_{n}) & = s_{\lambda}(1_{m})^{2}\prod_{j=1}^{m}\frac{(\lambda_{j}+n-j)!}{(\lambda_{j}+m-j)!}\,\frac{(j-1)!}{(n-m+j-1)!}, \qquad \mbox{if } l(\lambda) \leq m,
\\
&=\prod_{1 \leq i < j \leq m}(h_{j}-h_{i})^{2}\prod_{j=1}^{m}\frac{(h_{j}+n-m)!}{h_{j}!}\,\frac{1}{(j-1)!(n-m+j-1)!}. \label{meixnerid}
\end{split}
\end{equation} 
The sum over partitions becomes a sum over the set of $h_{j}$'s and this establishes \eqref{lpmeix}. 

Formula \eqref{cuebr01} can be proven quickly in the following way. Consider a general correlator of characteristic polynomials
\begin{equation}
\mathbb{E}\bigg(\prod_{j=1}^{n}\det(1+Uz_{j})\det(1+U^{\dagger}w_{j})\bigg), \label{cue-correlator}
\end{equation}
where $U$ is an $\ell \times \ell$ Haar distributed unitary matrix. By the dual Cauchy identity \eqref{dci}, we can expand \eqref{cue-correlator} as
\begin{equation}
\sum_{\lambda, l(\lambda) \leq n, \lambda_{1} \leq \ell}\sum_{\mu, l(\mu) \leq n, \mu_{1} \leq \ell}s_{\lambda}(\bm z)s_{\mu}(\bm w)\mathbb{E}(s_{\lambda'}(U)s_{\mu'}(U^{\dagger})) = \sum_{\lambda, l(\lambda) \leq n, \lambda_{1} \leq \ell}s_{\lambda}(\bm z)s_{\lambda}(\bm w)
\end{equation}
where we used orthogonality of Schur functions $\mathbb{E}(s_{\lambda'}(U)s_{\mu'}(U^{\dagger})) = \delta_{\lambda \mu}$ (see e.g. \cite[Eq. (10.21)]{Fo10}). Now choose $w_{j}=0$ for $j > m$. This restricts the length of the partitions in the sum to $m$. We also choose $w_{j} = q^{2}$ for $j\leq m$ and $z_{j}=1$ for $j=1,\ldots,n$. Then 
\begin{equation}
\mathbb{E} (\det(1+U)^{n}\det(1+q^{2}U^{\dagger})^{m}) = \sum_{\lambda, l(\lambda) \leq m, \lambda_{1} \leq \ell}s_{\lambda}(1_{n})s_{\lambda}(1_{m})q^{2|\lambda|}
\end{equation}
and we arrived again at the Schur function expansion in \eqref{gmnsum}. This shows \eqref{cuebr01}. The relation between averages of unitary matrices and last passage percolation was extensively discussed in \cite[Eq.(5.53)]{FW04}.

\medskip 

While Theorem \ref{thm:meix} contained the known exact relations between the last passage time and random matrix theory, below we present two further relations we believe to be new. To present them, recall that the Jacobi ensemble $\mathrm{JUE}_m(\alpha,\beta)$ is the joint probability density function of $m$ points $s_{1},\ldots,s_{m}$ on $[0,1]$ of the form
\begin{equation}
\frac{1}{  Z_m(\alpha,\beta) } \prod_{j=1}^{m}s_{j}^{ \alpha } (1-s_{j})^{ \beta } \Delta(\bm s)^{2}, \label{jensemble}
\end{equation}
where $\alpha,\beta>-1$ and the normalization constant is
\begin{equation}
Z_m(\alpha,\beta)  = \int_{[0,1]^{m}}\prod_{j=1}^{m}ds_{j}\,s_{j}^{ \alpha } (1-s_{j})^{ \beta } \Delta(\bm s)^{2}. 
\end{equation} 
We denote the largest point $x^{ (m) }_{\mathrm{max}}(\alpha,\beta) = \mathrm{max}(s_{1},\ldots,s_{m})$. By definition we have
\begin{equation} \label{defeigjue}
\mathbb{P}(  x^{ (m) }_{\mathrm{max}}(\alpha,\beta) \leq x) = 
 \frac{1}{  Z_m(\alpha,\beta) } 
 \int_{[0,x]^{m}}\prod_{j=1}^{m}ds_{j}\,s_{j}^{ \alpha } (1-s_{j})^{ \beta } \Delta(\bm s)^{2}. 
\end{equation}

We now prove Theorems~\ref{th:trunc-lpp} and ~\ref{th:jue-lpp}. 

\begin{proof}[Proof of Theorems~\ref{th:trunc-lpp} and ~\ref{th:jue-lpp}] We first show \eqref{lppjue}. 
We will use \eqref{gmnsum} and an identity obtained in \cite[Eq. (28)]{FK07} 
\begin{equation}\label{lol10}
s_{\lambda}(1_{m})s_{\lambda}(1_{n}) = \frac{s_{\lambda'}(1_{\ell })}{k_{n,m}^{\ell}}\int_{[0,\infty)^{m}}\prod_{j=1}^{m}dx_{j}\,\frac{x_{j}^{n-m}}{(1+x_{j})^{\ell +n+m}}s_{\lambda}(\bm x) \Delta(\bm x)^{2},
\end{equation}
where $\lambda'$ is the conjugate partition, and where (see \cite[Eqs. (87)--(88)]{FK07})
\begin{align*}
k_{n,m}^{\ell} = \int_{[0,\infty)^{m}}\prod_{j=1}^{m}dx_{j}\,\frac{x_{j}^{n-m}}{(1+x_{j})^{\ell +n+m}} \Delta(\bm x) ^{2} = Z^{(m)}(n-m,\ell). 
\end{align*}
Substituting \eqref{lol10} in \eqref{gmnsum}, and then simplifying using the dual Cauchy identity \eqref{dci}, we obtain
\begin{equation}
\mathbb{P}(G_{n,m} \leq \ell) = \frac{(1-q^{2})^{nm}}{k_{n,m}^{\ell}}\int_{[0,\infty)^{m}}\prod_{j=1}^{m}dx_{j}\,(1+q^{2}x_{j})^{\ell }\frac{x_{j}^{n-m}}{(1+x_{j})^{\ell+n+m}} \Delta(\bm x)^{2} .
\end{equation}
Now for each $j=1,\ldots,m$, we make the change of variables $x_{j} = \frac{s_{j}}{1-s_{j}}$ so that $dx_{j} = \frac{1}{(1-s_{j})^{2}}\,ds_{j}$ and $\frac{1}{1+x_{j}} = 1-s_{j}$. We also have
\begin{equation}
 \Delta(\bm x) ^{2}\prod_{j=1}^{m}dx_{j} = \Delta(\bm s)^{2}\prod_{j=1}^{m}\frac{ds_{j}}{(1-s_{j})^{2m}}.
\end{equation}
Therefore
\begin{equation}
\mathbb{P}(G_{n,m} \leq \ell) = \frac{(1-q^{2})^{nm}}{k_{n,m}^{\ell }}\int_{[0,1]^{m}}\prod_{j=1}^{m}ds_{j}\,s_{j}^{n-m}(1-s_{j}(1-q^{2}))^{\ell} \Delta(\bm s)^{2}. 
\end{equation}
Changing variables $s_{j} \to s_{j}/(1-q^{2})$ for all $j=1,\ldots,m$ cancels the prefactor and gives \eqref{defeigjue} with $x=1-q^{2}$, $\alpha=n-m$ and $\beta = \ell$. This completes the proof of \eqref{lppjue}.

Next, we show \eqref{momtrunc}. By \cite[Eq. (2.25)]{SS23}, we have
\begin{equation}
\begin{split}
\mathbb{E}(|\det(\mathsf{T}_{ \ell,\ell+n-m  }-q)|^{2m}) &= \prod_{j=1}^{m}\frac{(\ell+j-1)!}{(\ell+n-m+j-1)!}\sum_{\lambda, l(\lambda) \leq m, \lambda_{1} \leq \ell}s_{\lambda}(1_{m})^{2}q^{2|\lambda|}\prod_{j=1}^{m}\frac{(\lambda_{j}+n-j)!}{(\lambda_{j}+m-j)!}\\
&=\prod_{j=1}^{m}\frac{(n-m+j-1)!(\ell+j-1)!}{(\ell+n-m+j-1)!(j-1)!}\sum_{\lambda, l(\lambda) \leq m, \lambda_{1} \leq \ell}s_{\lambda}(1_{m})s_{\lambda}(1_{n})q^{2|\lambda|},
\end{split}
\end{equation}
where we used the first identity in \eqref{meixnerid}. The result now follows from \eqref{gmnsum}.
\end{proof}
We remark that Theorem \ref{th:jue-lpp} also implies the equivalence of the right-hand sides in \eqref{lppjue} and \eqref{momtrunc}, i.e.
\begin{equation}
\mathbb{E}(|\det(\mathsf{T}_{ \ell,\ell+n-m  }-q)|^{2m}) = \frac{1}{c_{\ell,n,m}}(1-q^{2})^{-nm} \mathbb{P}(x^{(m)}_{\mathrm{max}}(n-m,\ell) \leq 1-q^2).
\end{equation}
This identity first appeared in \cite[Theorem 1.10]{DS22}.

\subsection{Dualities of JUE and CUE averages} \label{Subsec_JUE CUE}
There are several duality relations between JUE and CUE averages that are extensively discussed and used in the series of works \cite{FW01,FW02,FW04}. We provide some of them for consistency checking with Theorem~\ref{th:jue-lpp}. 
For a partition $\lambda$ we introduce the hypergeometric coefficient
\begin{equation}
[u]_{\lambda} = \prod_{j=1}^{l(\lambda)}\frac{\Gamma(\lambda_{j}+u-j+1)}{\Gamma(u-j+1)}.
\end{equation}
These coefficients obey a transposition property  (see e.g. \cite[Exercise 12.4.2 with $\alpha=1$]{Fo10})
\begin{equation}
[-u]_{\lambda} = (-1)^{|\lambda|}[u]_{\lambda'}. \label{transposition_property}
\end{equation}
It will be useful to recall a particular case of Selberg integral (see e.g. \cite[Chapter 4]{Fo10} and \cite{FW08})
\begin{equation}
\begin{split}
S_{N}( \alpha,\beta) &:= \int_{[0,1]^{N}}\prod_{j=1}^{N}ds_{j}\,s_{j}^{ \alpha }(1-s_{j})^{ \beta }\Delta(\bm s)^{2} = \prod_{j=0}^{N-1}\frac{\Gamma( \alpha +j+1)\Gamma( \beta +j+1)\Gamma(j+2)}{\Gamma(\alpha+\beta+N+j+1)}, \label{selberg}
\end{split}
\end{equation}
valid for $\re  \alpha  > -1, \re  \beta >-1$. We also recall a trigonometric version of this integral which is often referred to as the Morris integral (see e.g. \cite[Eqs. (1.17)--(1.18) and page 496]{FW08}):
\begin{equation}
\begin{split}
M_{N}(\alpha,\beta) &:= \int_{\mathcal{C}^{N}}\prod_{j=1}^{N}  \frac{dz_{j}}{i z_{j}}  \,z_{j}^{\frac{ \alpha - \beta }{2}}|1+z_{j}|^{ \alpha + \beta }|\Delta(\bm z)|^{2} = (2\pi)^{N}\prod_{j=0}^{N-1}\frac{\Gamma( \alpha + \beta +j+1)\Gamma(j+2)}{\Gamma( \alpha +j+1)\Gamma(\beta+j+1)}, \label{morris}
\end{split}
\end{equation}
valid for $\mathrm{Re}( \alpha + \beta )>-1$, with $\mathcal{C}$ denoting the unit circle oriented positively.  The works \cite{K93,K97} independently obtained a generalization of \eqref{selberg} where the integrand is weighted by a Schur function $s_{\lambda}(\bm s)$. The trigonometric or Morris version of this Schur function average was obtained by Forrester \cite{Fo94}. Finally, a spherical version of the Schur polynomial averages was obtained in \cite{FK07} and generalized further in \cite{FS09}. We summarise these results in the following.
\begin{lem}
Let $\lambda$ be a partition. 
\begin{itemize}
    \item Morris type \cite[Eq. above (2.46)]{FW08}: for any $ \alpha , \beta $ such that $\mathrm{Re}(\alpha+\beta)>-1$, we have
\begin{equation}
\int_{\mathcal{C}^{N}}\prod_{j=1}^{N}  \frac{dz_{j}}{i z_{j}} \,z_{j}^{\frac{\alpha-\beta}{2}}|1+z_{j}|^{ \alpha + \beta } s_{\lambda}(\bm z)|\Delta(\bm z)|^{2} = M_{N}(\alpha,\beta)s_{\lambda}(1_N)\frac{[-\beta]_{\lambda}}{[\alpha+N]_{ \lambda }} (-1)^{|\lambda|}. \label{cue-kaneko}
\end{equation}
\item Spherical type \cite[Eq. (28)]{FK07} and \cite[Lemma 2.4]{FS09}:  for any $\alpha,\beta,\lambda$ such that $\re \alpha >-1$, $\re \beta >-1$, $l(\lambda) \leq m$, $l(\lambda')< 1+\re \beta$, 
\begin{equation}
\int_{[0,\infty)^{m}}\prod_{j=1}^{m}dx_{j}\,\frac{x_{j}^{\alpha}}{(1+x_{j})^{\alpha+\beta+2m}}s_{\lambda}(\bm x)\Delta(\bm x)^{2} = S_{m}(\alpha,\beta)s_{\lambda}(1_{m})\frac{[\alpha+m]_{\lambda}}{[-\beta]_{\lambda}}(-1)^{|\lambda|}. \label{spherical-kaneko}
\end{equation}
\item Jacobi type \cite[Eq. (2.46)]{FW08}: for any $\alpha,\beta$ such that $\mathrm{Re}\,\alpha>-1$ and $\mathrm{Re}\,\beta>-1$, we have
\begin{equation}
\int_{[0,1]^{m}}\prod_{j=1}^{m}ds_{j}\,s_{j}^{ \alpha }(1-s_{j})^{ \beta }s_{\lambda}(\bm s)\Delta(\bm s)^{2} = S_{m}(\alpha,\beta)s_{\lambda}(1_m)\frac{[\alpha+m]_{\lambda}}{[\alpha+\beta+2m]_{\lambda}}. \label{jac-kaneko}
\end{equation}
\end{itemize}
\end{lem}

From these integration formulas, we establish the following.
\begin{lem}
\label{lem:cuetojue}
Assume that $m,N \in \mathbb{N}$,  $t\in \C$, $\re \eta_{2} > m-1$ and $\re \eta_{1}>-1$. Then
\begin{align}
\begin{split}
&\quad \int_{[0,1]^{m}} \prod_{j=1}^{m}ds_{j}\,s_{j}^{\eta_{2}-m}(1-s_{j})^{\eta_{1}}(1-(1-t)s_{j})^{N}\Delta(\bm s)^{2}
\\
&= \frac{S_{m}(\eta_{2}-m,\eta_{1})}{M_{m}(\eta_{1}+\eta_{2},-\eta_{2})}\int_{\mathcal{C}^{m}}\prod_{j=1}^{m}  \frac{dz_{j}}{i z_{j}} \,z_{j}^{\frac{\eta_{1}+2\eta_{2}}{2}}|1+z_{j}|^{\eta_{1}}(1+(1-t)z_{j})^{N}|\Delta(\bm z)|^{2}. \label{CUEtoJUE}
\end{split}
\end{align} 
\end{lem}
\begin{proof}
We expand both powers of $N$ in \eqref{CUEtoJUE} using the dual Cauchy identity \eqref{dci}. Integrating term by term and evaluating the resulting Schur function averages using \eqref{jac-kaneko} and \eqref{cue-kaneko} results in identical expressions and completes the proof.
\end{proof}

\begin{lem}
\label{lem:cuejue-duality}
Assume that $m,N \in \mathbb{N}$, $t\in \C$, $\re \eta_{2} > m-1$ and $\re \eta_{1}>-1$. Then
\begin{equation}
\begin{split}
& \quad \frac{1}{M_{N}(\eta_{1},\eta_{2})}\int_{\mathcal{C}^{N}} \prod_{j=1}^{N}\frac{dz_{j}}{iz_{j}}\,z_{j}^{\frac{\eta_{1}-\eta_{2}}{2}}|1+z_{j}|^{\eta_{1}+\eta_{2}}(1+tz_{j})^{m}|\Delta(\bm z)|^{2}
\\
&=\frac{1}{S_{m}(\eta_{2}-m,\eta_{1}+N)}\int_{[0,1]^{m}}\prod_{j=1}^{m}ds_{j}\,s_{j}^{\eta_{2}-m}(1-s_{j})^{\eta_{1}}(1-(1-t)s_{j})^{N}\Delta(\bm s)^{2}. \label{CUEtoJUE_duality}
\end{split}
\end{equation}
\end{lem}
\begin{proof}
To establish this duality, we start from the integral on the right-hand side of \eqref{CUEtoJUE_duality} and make the change of variables $s_{j} = \frac{x_{j}}{1+x_{j}}$ for $j=1,\ldots,m$, leading to
\begin{equation}
\begin{split}
&\quad  \int_{[0,1]^{m}} \prod_{j=1}^{m}ds_{j}\,s_{j}^{\eta_{2}-m}(1-s_{j})^{\eta_{1}}(1-(1-t)s_{j})^{N}\Delta(\bm s)^{2} 
\\ 
&=\int_{[0,\infty)^{m}} \prod_{j=1}^{m}dx_{j}\,\frac{x_{j}^{\eta_{2}-m}}{(1+x_{j})^{\eta_{1}+\eta_{2}+m+N}}(1+tx_{j})^{N}\Delta(\bm x)^{2}\\
&=\sum_{\lambda, l(\lambda) \leq m, \lambda_{1} \leq N}t^{|\lambda|}s_{\lambda'}(1_{N})\int_{[0,\infty)^{m}}\prod_{j=1}^{m}dx_{j}\,\frac{x_{j}^{\eta_{2}-m}}{(1+x_{j})^{\eta_{1}+\eta_{2}+m+N}}s_{\lambda}(\bm x)\Delta(\bm x)^{2}\\
&=S_{m}(\eta_{2}-m,\eta_{1}+N)\sum_{\lambda, l(\lambda) \leq m, \lambda_{1} \leq N}(-t)^{|\lambda|}s_{\lambda'}(1_{N})s_{\lambda}(1_{m})\frac{[\eta_{2}]_{\lambda}}{[-\eta_{1}-N]_{\lambda}}, \label{dual_pf_eq1}
\end{split}
\end{equation}
where we applied \eqref{dci} and \eqref{spherical-kaneko}. Applying the same Schur function expansion to the left-hand side of \eqref{CUEtoJUE_duality} followed by use of \eqref{cue-kaneko}, we arrive at
\begin{equation}
\begin{split}
&\quad \int_{\mathcal{C}^{N}} \prod_{j=1}^{N}\frac{dz_{j}}{iz_{j}}\,z_{j}^{\frac{\eta_{1}-\eta_{2}}{2}}|1+z_{j}|^{\eta_{1}+\eta_{2}}(1+tz_{j})^{m}|\Delta(\bm z)|^{2}
\\
&  = \sum_{\lambda, l(\lambda) \leq N, \lambda_{1} \leq m}t^{|\lambda|}s_{\lambda'}(1_{m}) \int_{\mathcal{C}^{N}} \prod_{j=1}^{N}\frac{dz_{j}}{iz_{j}}\,z_{j}^{\frac{\eta_{1}-\eta_{2}}{2}}|1+z_{j}|^{\eta_{1}+\eta_{2}}s_{\lambda}(\bm z) |\Delta(\bm z)|^{2}
\\
&= M_{N}(\eta_{1},\eta_{2})\sum_{\lambda, l(\lambda) \leq N, \lambda_{1} \leq m}(-t)^{|\lambda|}s_{\lambda'}(1_{m})s_{\lambda}(1_{N})\frac{[-\eta_{2}]_{\lambda}}{[\eta_{1}+N]_{\lambda}}\\
&=M_{N}(\eta_{1},\eta_{2})\sum_{\lambda, l(\lambda) \leq N, \lambda_{1} \leq m}(-t)^{|\lambda|}s_{\lambda'}(1_{m})s_{\lambda}(1_{N})\frac{[\eta_{2}]_{\lambda'}}{[-\eta_{1}-N]_{\lambda'}},\label{dual_pf_eq2}
\end{split}
\end{equation}
where we applied  \eqref{transposition_property} in the last line. Finally, note that the sums in the last line of \eqref{dual_pf_eq1} and \eqref{dual_pf_eq2} are identical. This follows from the fact that $\lambda_{1} = l(\lambda')$ and summing over $\lambda'$ in place of $\lambda$. 
\end{proof}
Comparing Lemmas \ref{lem:cuejue-duality} and \ref{lem:cuetojue}, we immediately get the following duality between CUE averages of dimension $N$ and $m$.
\begin{cor}
For any $\eta_{1},\eta_{2}$ such that $\mathrm{Re}(\eta_{1}+\eta_{2}) > -1$ and $\mathrm{Re}(\eta_{1})>-1$,  and any $t\in \C$,  we have
\begin{equation}
\begin{split}
\label{cue-duality}
&\quad \frac{S_{m}(\eta_{2}-m,\eta_{1}+N)}{M_{N}(\eta_{1},\eta_{2})}\int_{\mathcal{C}^{N}}\prod_{j=1}^{N}\frac{dz_{j}}{iz_{j}}\,z_{j}^{\frac{\eta_{1}-\eta_{2}}{2}}|1+z_{j}|^{\eta_{1}+\eta_{2}}(1+tz_{j})^{m}|\Delta(\bm z)|^{2}
\\
&= \frac{S_{m}(\eta_{2}-m,\eta_{1})}{M_{m}(\eta_{1}+\eta_{2},-\eta_{2})}\int_{\mathcal{C}^{m}}\prod_{j=1}^{m}\frac{dz_{j}}{iz_{j}}\,z_{j}^{\frac{\eta_{1}+2\eta_{2}}{2}}|1+z_{j}|^{\eta_{1}}(1+(1-t)z_{j})^{N}|\Delta(\bm z)|^{2}.
\end{split}
\end{equation}
\end{cor}

\begin{rem}
Setting $t=0$ in \eqref{cue-duality} and using the identity $(1+z_{j})^{N} = z_{j}^{N/2}|1+z_{j}|^{N}$, we get the identity
\begin{equation}
\frac{S_{m}(\eta_{2}-m,\eta_{1}+N)}{S_{m}(\eta_{2}-m,\eta_{1})} = \frac{M_{m}(\eta_{1}+\eta_{2}+N,-\eta_{2})}{M_{m}(\eta_{1}+\eta_{2},-\eta_{2})} \label{morris-selberg-identity}
\end{equation}
which can be checked directly by substituting the evaluations \eqref{morris} and \eqref{selberg}. Substituting $t=1$ in \eqref{cue-duality} and using $(1+z_{j})^{m} = z_{j}^{m/2}|1+z_{j}|^{m}$ gives another less obvious relation
\begin{equation}
 \frac{S_{m}(\eta_{2}-m,\eta_{1})}{S_{m}(\eta_{2}-m,\eta_{1}+N)} = \frac{M_{N}(\eta_{1}+m,\eta_{2})}{M_{N}(\eta_{1},\eta_{2})}.
\end{equation}
This can also be verified by direct substitution of \eqref{morris} and \eqref{selberg}, combined with appropriate manipulations of the resulting product of Gamma functions.
\end{rem}

\section{Large deviation probabilities of the JUE} \label{Section_large deviation JUE}

In this section, we prove Theorems~\ref{Thm_JUE LDP weak} and ~\ref{Thm_JUE LDP}. 

\subsection{Equilibrium measure of the constrained JUE}

Suppose $\alpha,\beta \ge 0$. Let $x_{1},\ldots,x_{n}$ be distributed according to \eqref{juepdf} with $\lambda_{1}=\alpha n$ and $\lambda_{2}=\beta n$ and \textit{conditioned} on the event that $\max\{x_{1},\ldots,x_{n}\}< d$ for some $d\in (0,1]$. In other words, the joint probability density function of $x_{1},\ldots,x_{n}$ is proportional to 
\begin{align}\label{density JUE n 2}  
\prod_{1\leq i < j \leq n} (x_{j}-x_{i})^{2} \prod_{j=1}^{n}e^{-nV(x_{j})} \, dx_{j}, \qquad x_{1},\ldots,x_{n}\in (0,1),
\end{align}
where  
\begin{align}\label{def of V}
V(x) := \begin{cases}
-\alpha \log x - \beta \log(1-x), & \mbox{if } x\in (0,d), 
\smallskip 
\\
+\infty, & \mbox{if } x\in [d,1).
\end{cases}
\end{align}
The equilibrium measure, denoted $\mu_V$, is the unique minimizer of the functional
\begin{align*}
\sigma \mapsto \iint \log \frac{1}{|x-y|} \, d\sigma(x) \, d\sigma(y) + \int V(x) \, d\sigma(x)
\end{align*}
among all Borel probability measures $\sigma$ supported on $[0,1]$. 

Recall that the Wachter distribution $\mu_{\rm W}$ is defined in \eqref{def of Wachter distribution}. 
Let \( d \in (0, b) \), where \( b \) denotes the right endpoint of the support of the Wachter distribution.
Let $\mathfrak{m}\in (0,a)$ be the unique solution of
\begin{align}  \label{eq of mathfrak m d}
2 = \alpha \bigg( \frac{\sqrt{d} }{\sqrt{ \mathfrak{m}  }}-1 \bigg) + \beta \bigg( \frac{\sqrt{1-d}}{\sqrt{1- \mathfrak{m} }}-1 \bigg).
\end{align} 
This is same as \eqref{a eq 4} with $\alpha=\gamma-1$, $\beta=\delta$ and $d=1-q^{2}$. 
The following is shown in \cite[Section 4]{RKC12}, see also Remark~\ref{Rem_explicit m} and Figure~\ref{Fig_JUE constrained}. 
\begin{prop}[\textbf{Constrained JUE distribution}] \label{Prop_constrained Wacther} 
Let $\alpha \ge 0$ and $\beta>0$.
\begin{itemize}
    \item[\textup(a)] \textup{\textbf{(Pulled regime)}} Suppose $d \in (b,1]$. Then $\mu_V= \mu_{ \rm W }$, where $\mu_{ \rm W }$ is the Wachter distribution \eqref{def of Wachter distribution}.  
    \smallskip
    \item[\textup (b)] \textup{\textbf{(Pushed regime)}} Suppose $d \in (0,b)$. \smallskip 
    \begin{itemize}
        \item[\textup(b1)] For $\alpha=0$, we have 
\begin{align}\label{def of mu d small N}
d\mu_V(x) = \frac{\psi(x)  }{\sqrt{x(d-x)}}\, dx, \qquad \psi(x):=\frac{1}{\pi} \bigg( 1+\frac{\beta}{2}-\frac{\beta \sqrt{1-d}}{2(1-x)} \bigg), \qquad x\in [0,d].
\end{align}
        \item[\textup(b2)] For $\alpha>0$, we have 
\begin{align}\label{def of mu d small}
d\mu_V(x) = \psi(x)\sqrt{\frac{x-\mathfrak{m}}{d-x}}\,dx, \qquad \psi(x):=\frac{1}{2\pi} \bigg( \frac{\alpha}{x}\frac{\sqrt{d}}{\sqrt{\mathfrak{m}}} - \frac{\beta}{1-x} \frac{\sqrt{1-d}}{\sqrt{1-\mathfrak{m}}} \bigg), \qquad x\in [\mathfrak{m},d]. 
\end{align}
    \end{itemize}
\end{itemize}
\end{prop}

In the case $\alpha = 0$, the left endpoint $0$ in \eqref{def of mu d small N} is a hard edge, whereas for $\alpha > 0$, the left endpoint $\mathfrak{m}$ in \eqref{def of mu d small} is a soft edge. In both cases, the right endpoint $d$ is a hard edge. Thus, \eqref{def of mu d small N} has two hard edges, while \eqref{def of mu d small} has one hard edge and one soft edge. These correspond, respectively, to the general settings described in \eqref{def of eq msr general two hard edge} and \eqref{def of eq msr general one hard edge}.

\begin{figure}
	\begin{subfigure}{0.35\textwidth}
	\begin{center}	
		\includegraphics[width=\textwidth]{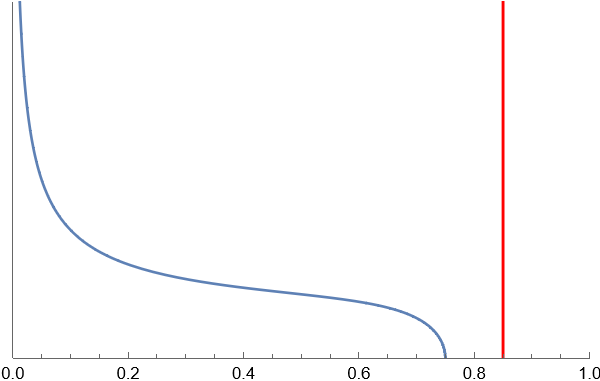}
	\end{center}
	\subcaption{$\lambda_1=0, \lambda_2=2n$, $d=0.85$}
\end{subfigure}	 \qquad 
\begin{subfigure}{0.35\textwidth}
	\begin{center}	
		\includegraphics[width=\textwidth]{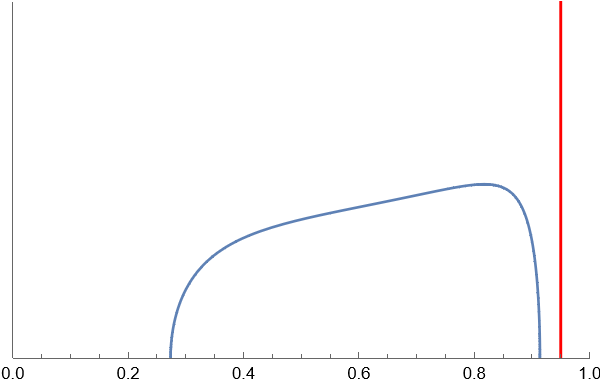}
	\end{center}
	\subcaption{$\lambda_1=4n, \lambda_2=2n$, $d=0.9$}
\end{subfigure}	

\begin{subfigure}{0.35\textwidth}
	\begin{center}	
		\includegraphics[width=\textwidth]{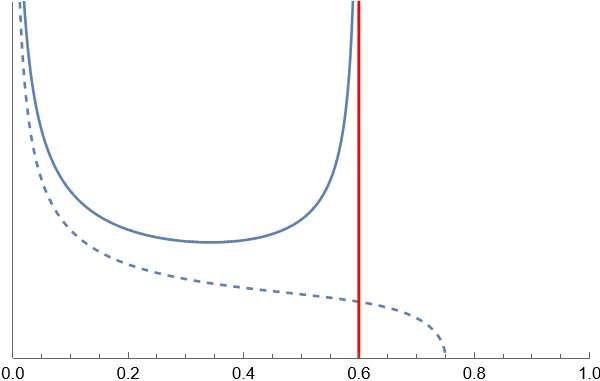}
	\end{center}
	\subcaption{$\lambda_1=0, \lambda_2=2n$, $d=0.6$}
\end{subfigure}	 \qquad 
\begin{subfigure}{0.35\textwidth}
	\begin{center}	
		\includegraphics[width=\textwidth]{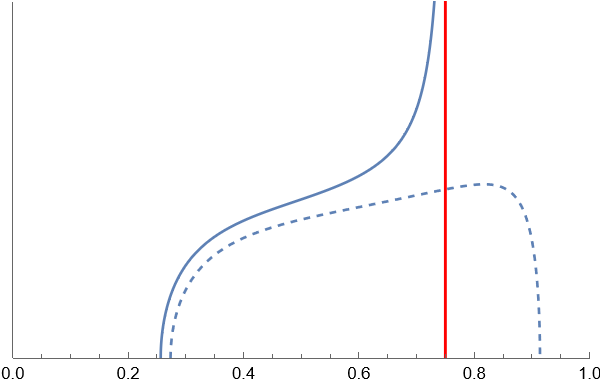}
	\end{center}
	\subcaption{$\lambda_1=4n, \lambda_2=2n$, $d=0.75$}
\end{subfigure}	
	\caption{Plots show limiting densities of the constrained JUE. (A) and (B) correspond to the pulled regime of Proposition \ref{Prop_constrained Wacther}, and (C) and (D) to the pushed regime. } \label{Fig_JUE constrained}
\end{figure}

\subsection{Asymptotics of Hankel determinants with a Laguerre-type or Jacobi-type potential}

In this subsection, we recall results from \cite{CG21} and adapt them to our setting.  
Here, we consider general potentials \( V \) and \( W \). Let
\begin{align}
\begin{split}
L_{n}(V,W) & := \det \bigg( \int_{\mathcal{I}} x^{j+k}e^{-nV(x)}e^{W(x)} \,dx \bigg)_{j,k=0,\ldots,n-1}
\\
& = \frac{1}{n!}\int_{\mathcal{I}}\ldots\int_{\mathcal{I}}\prod_{1\leq j < k \leq n} (x_{k}-x_{j})^{2} \prod_{j=1}^{n}e^{-nV(x_{j})}e^{W(x_j)}\, dx_{j}, \label{def of Ln}
\end{split}
\end{align}
where $\mathcal{I}\subset \R$.

Let $\mu \equiv \mu_V$ be the equilibrium measure associated with the potential $V$ and write $\mathcal{S}:=\mathrm{supp} \, \mu$. As is well-known \cite{ST97}, $\mu$ is completely determined by the following Euler-Lagrange variational conditions:
\begin{align}
& 2\int \log|x-y| \, d\mu(y) = V(x)-\ell, \qquad \mbox{for } x \in \mathcal{S}, \label{EL equality 2} 
\\
& 2\int \log|x-y| \, d\mu(y) \leq V(x)-\ell, \qquad \mbox{for } x \in \R\setminus \mathcal{S}, \label{EL inequality 2}
\end{align}
where $\ell$ is a real constant. 

In the statement of the following result from \cite{CG21}, ``\emph{one-cut regular}" means that $V:\mathcal{I}\to \R$ is real-analytic, satisfies the growth condition $\lim_{x\to \pm \infty, x\in \mathcal{I}}V(x)/\log|x| = +\infty$, that $\mathcal{S}=[a,b]$ consists of a single interval, that the Euler-Lagrange inequality \eqref{EL inequality 2} is assumed to hold strictly outside $[a,b]$, and that the density $\psi$ of $\mu$ is strictly positive on $[a,b]$ and admits an analytic continuation in a neighbourhood of $[a,b]$.

The precise asymptotic behaviour of the Hankel determinant $L_n(V,W)$ depends on the nature of the associated equilibrium measure, particularly on the number of hard edges. Theorem~\ref{thm:general a and b} treats the case with one hard edge and one soft edge, while Theorem~\ref{thm:general a and b Jacobi-type} addresses the case with two hard edges.

The following result can be seen as a special case of \cite[Theorem 1.2]{CG21}. 
However, deriving it involves a series of elementary but careful transformations, 
relying on \cite[Remark 1.4]{Ch19}. 
For the reader’s convenience, we provide the details of this procedure in Appendix~\ref{Appendix_general results on Hankel}.

In what follows, $W$ will be a function satisfying the following assumption.
\begin{assump}\label{ass:W}
The function $W:\mathcal{I}\to \R\cup \{+\infty\}$ satisfies the following.
\begin{itemize}
\item $W$ is analytic in a neighbourhood of $\mathcal{S}$,
\smallskip 
\item $e^{W}$ is integrable on $\mathcal{I}$,
\smallskip 
\item there exists $c'>0$ such that $e^{W(x)-c'V(x)}\to 0$ as $x\to \pm \infty$, $x\in \mathcal{I}$.
\end{itemize}
\end{assump}

\begin{thm}[\textbf{Partition function expansion: one soft edge and one hard edge in the equilibrium measure}]  \label{thm:general a and b}
Let $b>a$, $\mathcal{I}=(-\infty,b]$, and $V:\mathcal{I}\to \R$ be a one-cut regular potential whose equilibrium measure is of the form
\begin{equation} \label{def of eq msr general one hard edge}
d\mu(x)= \psi(x)\sqrt{\frac{x-a}{b-x}} \, \mathbbm{1}_{ [a,b] }(x) dx. 
\end{equation}
Let $W$ be a function satisfying Assumption \ref{ass:W}.
Then as $n \to \infty$, we have
	\begin{equation}
		L_{n}(V,W) = \exp\bigg(C_{1} n^{2} + C_{2} n + C_{3} \log n + C_{4} +  O \Big( \frac{\log n}{n} \Big)\bigg),
	\end{equation}
	where
	\begin{align*}
		& C_{1}  = \log \frac{b-a}{4} - \frac{3}{2} - \frac{1}{2} \int_{a}^{b} \Big(V(x)-4\frac{b-x}{b-a}\Big)\Big( \frac{2}{(b-a)\pi} + \psi(x) \Big)\sqrt{\frac{x-a}{b-x}} \, dx, 
		\\
		& C_{2} = \log(2\pi) + \int_{a}^{b}W(x)\psi(x)\sqrt{\frac{x-a}{b-x}}\, dx, \nonumber \\
		& C_{3} = - \frac{1}{6}, \\
		& C_{4} = 2\zeta^{\prime}(-1) - \frac{1}{8}\log \big( \tfrac{b-a}{2}\pi\psi(b) \big) - \frac{1}{24}\log \big( \tfrac{b-a}{2}\pi\psi(a) \big) \\
	& \qquad + \frac{1}{4\pi^{2}} \int_{a}^{b}  \frac{W(x)}{\sqrt{(b-x)(x-a)}} \bigg(\dashint_{a}^{b} \frac{W^{\prime}(y)\sqrt{(b-y)(y-a)}}{x-y} \, dy \bigg) \, dx.
	\end{align*}
\end{thm} 

Extending \eqref{def of Ln}, we define 
\begin{equation}
L_{n}(V,W,\alpha) := \frac{1}{n!}\int_{\mathcal{I}}\ldots\int_{\mathcal{I}}\prod_{1\leq j < k \leq n} (x_{k}-x_{j})^{2} \prod_{j=1}^{n}(x_j-a)^\alpha e^{-nV(x_{j})}e^{W(x)} \, dx_{j}. 
\end{equation}
The following can be similarly derived from \cite[Theorem 1.3]{CG21}.

\begin{thm}[\textbf{Partition function expansion: two hard edges in the equilibrium measure}]\label{thm:general a and b Jacobi-type}
Let $\alpha>-1$, $b>a$, $\mathcal{I}=[a,b]$, and $V:\mathcal{I}\to \R$ be a one-cut regular potential whose equilibrium measure is of the form 
\begin{equation} \label{def of eq msr general two hard edge}
d\mu(x)= \frac{\psi(x)}{\sqrt{(b-x)(x-a)}} \,\mathbbm{1}_{ [a,b] }(x).
\end{equation}
Then as $n \to \infty$, we have 
	\begin{equation}
		L_{n}(V,W=0,\alpha) = \exp\bigg( C_{1} n^{2} + C_{2} n + C_{3} \log n + C_{4} +  O \Big( \frac{\log n}{n} \Big) \bigg),
	\end{equation}
	where
	\begin{align*}
		& C_{1}   = \log \frac{b-a}{4} - \frac{1}{2} \int_{a}^{b} V(x)\Big( \frac{1}{\pi} + \psi(x) \Big)\frac{dx}{\sqrt{(b-x)(x-a)}}, \\
		& C_{2} = \alpha \log \frac{b-a}{2} + \log(2\pi) - \alpha \log 2 - \frac{\alpha}{2\pi}\int_{a}^{b} \frac{V(x)}{\sqrt{(x-a)(b-x)}} \, dx + \frac{\alpha}{2}V(a) \\ 
		& C_{3} = - \frac{1}{4}+\frac{\alpha^{2}}{2}, \\
		& C_{4} = 3\zeta^{\prime}(-1) + \frac{\log 2}{12} - \frac{1}{8}\log \big( \pi\psi(b) \big) - \frac{1-4\alpha^{2}}{8}\log \big( \pi\psi(a) \big) + \frac{\alpha}{2}\log(2\pi) - \frac{\alpha^{2}}{2}\log 2 -\log G(1+\alpha).
	\end{align*}
\end{thm}

\subsection{Proof of Theorem \ref{Thm_JUE LDP} (b)}

We first establish a few useful identities, which are presented in the next two lemmas.

\begin{lem}\label{lemma:some identities 1}
Let $W(x):=t \log x$, $t\in \R$, and let $b>a>0$. The following relations hold:
\begin{align}
& \int_{a}^{b} \frac{  W(x)  }{x}\frac{\sqrt{x-a}}{\sqrt{b-x}} \, dx = 2\pi t \bigg(   \log\Big(\frac{\sqrt{a}+\sqrt{b}}{2}\Big) +    \frac{\sqrt{a}}{\sqrt{b}}   \log\Big(\frac{\sqrt{a}+\sqrt{b}}{ 2\sqrt{ab} } \Big) \bigg),  \label{lol6} 
\\
& \int_{a}^{b} \frac{  W(x)  }{1-x}\frac{\sqrt{x-a}}{\sqrt{b-x}} \, dx = 2 \pi t \bigg(  \log \Big( \frac{2}{ \sqrt{a}+\sqrt{b} } \Big) +  \frac{\sqrt{1-a}}{\sqrt{1-b}} \log \Big( \frac{ \sqrt{a(1-b)} +\sqrt{b(1-a)}  }{  \sqrt{1-a}+\sqrt{1-b}   } \Big) \bigg). \label{lol7}
\end{align}
\end{lem}
\begin{proof}
Since $W$ is analytic in a neighbourhood of $(a,b)$,
\begin{align*}
\int_{a}^{b}   \frac{ W(x) }{x}\frac{\sqrt{x-a}}{\sqrt{b-x}} \, dx = \frac{1}{2i}\oint_{\mathcal{C}}  \frac{ W(z) }{z}\frac{\sqrt{z-a}}{\sqrt{z-b}} \, dz = \frac{1}{2i}\oint_{\mathcal{C}} \frac{ W(z) }{z}\bigg( \frac{\sqrt{z-a}}{\sqrt{z-b}} - 1 \bigg)\, dz,
\end{align*}
where the principal branch is used for the roots, and $\mathcal{C}$ is a closed loop oriented anticlockwise and surrounding $(a,b)$ but not $0$. We next deform the contour $\mathcal{C}$ into $C(0,M)\cup (-C(0,\epsilon)) \cup \big((-M,-\epsilon)+i0_{+}\big) \cup \big((-\epsilon,-M)-i0_{+}\big)$, where $M>b$ and $\epsilon \in (0,a)$ are arbitrary and given $z_{0}\in \C$ and $R>0$, $C(z_{0},R)$ denotes the circle centred at $z_{0}$ of radius $R$ oriented counterclockwise. See Figure~\ref{Fig_contour deform} for an illustration.  

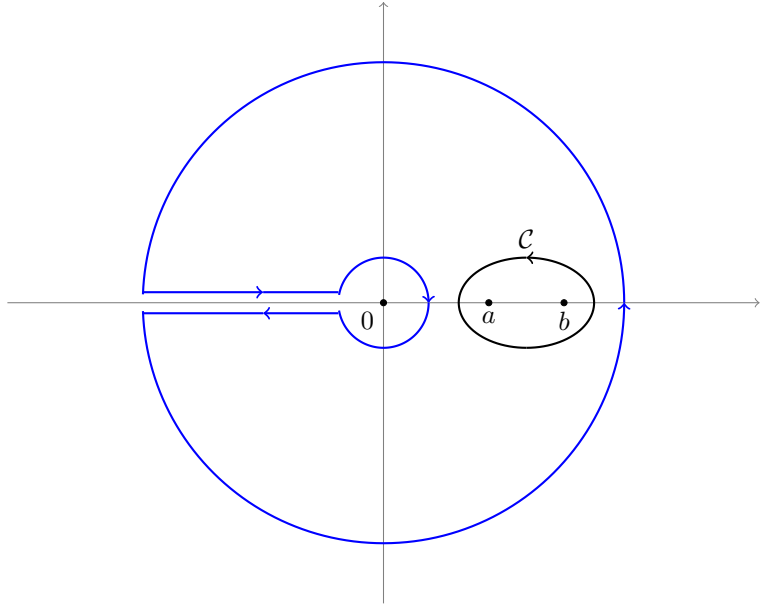
\begin{figure}[t]
\begin{center}

\begin{tikzpicture}[scale=2]

% Parameters
\def\e{0.3}    % small radius (epsilon < a)
\def\R{1.6}    % large radius (R > b)
\def\a{0.7}
\def\b{1.2}
\def\h{0.07}
\def\xr{0.45} % x radius
\def\yr{0.3} % y radius
\def\xc{(\a + \b)/2} % x center

% Axes
\draw[->, gray] (-2.5, 0) -- (2.5, 0); %node[right] {\(\operatorname{Re} z\)};
\draw[->, gray] (0, -2) -- (0, 2); %node[above] {\(\operatorname{Im} z\)};

\draw[thick, blue, ->] (-\R, \h) -- (-\R/2, \h);   % upper connecting line
\draw[thick, blue] (-\R/2, \h) -- (-\e, \h);

\draw[thick, blue, ->] (-\e, -\h) -- (-\R/2, -\h); % lower connecting line
\draw[thick, blue] (-\R/2, -\h) -- (-\R, -\h); % lower connecting line

 \draw[thick, blue, ->, domain=170:0, samples=100] 
    plot ({\e*cos(\x)}, {\e*sin(\x)});

\draw[thick, blue, domain=0:-170, samples=100] 
    plot ({\e*cos(\x)}, {\e*sin(\x)});

 \draw[thick, blue, ->, domain=-178:0, samples=100] 
    plot ({\R*cos(\x)}, {\R*sin(\x)});

 \draw[thick, blue, domain=0:178, samples=100] plot ({\R*cos(\x)}, {\R*sin(\x)});

%\draw[thick, red, ->]    ({(\a + \b)/2}, 0) ellipse [x radius={(\b - \a)/2 + 0.3}, y radius=0.25];

% Overlay a partial red arc with arrow
\draw[thick, ->] 
    plot[domain=-90:90, samples=50] 
    ({\xc + \xr*cos(\x)}, {\yr*sin(\x)});

\draw[thick] 
    plot[domain=90:270, samples=50] 
    ({\xc + \xr*cos(\x)}, {\yr*sin(\x)});

% Points on real axis
%\filldraw (0, 0) circle (0.02) node[below left] {\(0\)};
\filldraw (\a, 0) circle (0.02) node[below] {\(a\)};
\filldraw (\b, 0) circle (0.02) node[below] {\(b\)};

\filldraw (0, 0) circle (0.02) node[below left] {\(0\)};

\node at ({\xc}, {\yr + 0.125}) {\(\mathcal{C}\)};

\end{tikzpicture}
\end{center}
    \caption{Illustration of the contour deformation used in the proof of Lemma \ref{lemma:some identities 1}.}
    \label{Fig_contour deform}
\end{figure}

Since the integrand is $O(z^{-2}\log z)$ as $z\to \infty$, the contribution along $C(0,M)$ tends to $0$ as $M\to +\infty$. Hence, for any $\epsilon \in (0,a)$,
\begin{align}\label{lol5}
\int_{a}^{b}  \frac{ W(x) }{x}\frac{\sqrt{x-a}}{\sqrt{b-x}} \, dx = \pi t \int_{-\infty}^{-\epsilon} \frac{1}{x}\bigg( \frac{\sqrt{x-a}}{\sqrt{x-b}} -1\bigg) \, dx - \frac{t}{2}\int_{-\pi}^{\pi} \log(\epsilon e^{i\theta}) \bigg( \frac{\sqrt{a-\epsilon e^{i\theta}}}{\sqrt{b-\epsilon e^{i\theta}}} - 1 \bigg)\, d\theta.
\end{align}
The first integral in \eqref{lol5} can be evaluated explicitly using primitives, and as $\epsilon \to 0$ we find
\begin{align*}
\int_{-\infty}^{-\epsilon} \frac{1}{x}\bigg( \frac{\sqrt{x-a}}{\sqrt{x-b}} -1\bigg) \, dx = \bigg( \frac{\sqrt{a}}{\sqrt{b}} - 1 \bigg) \log \epsilon  -\log 4 - \frac{\sqrt{a}}{\sqrt{b}}\log(4ab) + 2\Big( 1 + \frac{\sqrt{a}}{\sqrt{b}} \Big) \log(\sqrt{a}+\sqrt{b}) + O(\epsilon).
\end{align*}
For the second integral in \eqref{lol5}, we use
\begin{align*}
\log(\epsilon e^{i\theta}) \bigg( \frac{\sqrt{a-\epsilon e^{i\theta}}}{\sqrt{b-\epsilon e^{i\theta}}} - 1 \bigg) = \bigg( \frac{\sqrt{a}}{\sqrt{b}} - 1 \bigg) (\log \epsilon + i \theta) + O(\epsilon \log \epsilon), \qquad \mbox{as } \epsilon \to 0,
\end{align*}
uniformly for $\theta \in (-\pi,\pi)$ to deduce that
\begin{align*}
\int_{-\pi}^{\pi} \log(\epsilon e^{i\theta}) \bigg( \frac{\sqrt{a-\epsilon e^{i\theta}}}{\sqrt{b-\epsilon e^{i\theta}}} - 1 \bigg) \, d\theta = 2\pi \bigg( \frac{\sqrt{a}}{\sqrt{b}} - 1 \bigg)\log \epsilon + O(\epsilon \log \epsilon), \qquad \mbox{as } \epsilon \to 0.
\end{align*}
Substituting the above in \eqref{lol5} and taking $\epsilon \to 0$, we find \eqref{lol6}. The proof of \eqref{lol7} is similar (albeit slightly more involved) and we omit it.
\end{proof}
\begin{lem}\label{lemma:some identities 2}
Let $W(x):=t \log x$, $t\in \R$, and let $b>a>0$. The following relation hold:
\begin{align}\label{lol9}
\int_{a}^{b}  \frac{W(x)}{\sqrt{(b-x)(x-a)}} \bigg(\dashint_{a}^{b} \frac{W^{\prime}(y)\sqrt{(b-y)(y-a)}}{x-y} \, dy \bigg) \,dx = (2\pi t)^{2} \log \bigg( \frac{\sqrt{a}+\sqrt{b}}{2 (ab)^{\frac{1}{4}}} \bigg) . 
\end{align}
\end{lem}
\begin{proof}
Let $x\in (a,b)$ be fixed. Since $W'(y)=t/y$, we have
\begin{align}\label{lol8}
\dashint_{a}^{b} \frac{W^{\prime}(y)\sqrt{(b-y)(y-a)}}{x-y} \, dy = \frac{-t}{2i}\oint_{\mathcal{C}} \frac{1}{y}\frac{\sqrt{y-b}\sqrt{y-a}}{x-y} \, dy = \pi t \bigg( 1-\frac{\sqrt{ab}}{x} \bigg),
\end{align}
where the principal branch is used for the roots, and $\mathcal{C}$ is a closed loop oriented anticlockwise and surrounding $(a,b)$ but not $0$. The last equality in \eqref{lol8} follows after deforming $\mathcal{C}$ towards $\infty$ and using the residues
\begin{align*}
\mbox{Res}\bigg( \frac{1}{y}\frac{\sqrt{y-b}\sqrt{y-a}}{x-y};y=\infty \bigg) = 1, \qquad \mbox{Res}\bigg( \frac{1}{y}\frac{\sqrt{y-b}\sqrt{y-a}}{x-y};y=0 \bigg) = -\frac{\sqrt{ab}}{x}.
\end{align*}
Substituting \eqref{lol8} in the left-hand side of \eqref{lol9} yields
\begin{align*}
\int_{a}^{b}  \frac{W(x)}{\sqrt{(b-x)(x-a)}} \bigg(\dashint_{a}^{b} \frac{W^{\prime}(y)\sqrt{(b-y)(y-a)}}{x-y} \, dy \bigg) \, dx = \frac{\pi t^{2}}{2i}\oint_{\mathcal{C}} \frac{\log x}{\sqrt{x-b}\sqrt{x-a}} \bigg( 1 - \frac{\sqrt{ab}}{x} \bigg) \, dx.
\end{align*}
As in the previous lemma, by deforming the contour $\mathcal{C}$ into $C(0,M)\cup (-C(0,\epsilon)) \cup \big((-M,-\epsilon)+i0_{+}\big) \cup \big((-\epsilon,-M)-i0_{+}\big)$, where $M>b$ and $\epsilon \in (0,a)$ are arbitrary (see Figure~\ref{Fig_contour deform}), we obtain 
\begin{align*}
\int_{a}^{b}  \frac{W(x)}{\sqrt{(b-x)(x-a)}} \bigg(\dashint_{a}^{b} \frac{W^{\prime}(y)\sqrt{(b-y)(y-a)}}{x-y} \, dy \bigg) \, dx = \pi^{2}t^{2} \big( I_{\epsilon,M} + I_{\epsilon} + I_{M}\big),
\end{align*}
where 
\begin{align*}
&  I_{\epsilon,M} = \int_{-M}^{-\epsilon} \bigg( 1-\frac{\sqrt{ab}}{x} \bigg) \frac{dx}{\sqrt{x-b}\sqrt{x-a}},
\\
& I_{\epsilon} = \epsilon \int_{-\theta}^{\theta}\frac{\log(\epsilon e^{i\theta})}{\sqrt{b-\epsilon e^{i\theta}}\sqrt{a-\epsilon e^{i\theta}}} \bigg( 1-\frac{\sqrt{ab}}{\epsilon e^{i\theta}} \bigg)   \, e^{i\theta}\frac{d\theta}{2\pi}, 
\\
& I_{M} = M \int_{-\theta}^{\theta}\frac{\log(M e^{i\theta})}{\sqrt{M e^{i\theta}-b}\sqrt{M e^{i\theta}-a}} \bigg( 1-\frac{\sqrt{ab}}{M e^{i\theta}} \bigg)   e^{i\theta} \, \frac{d\theta}{2\pi}.
\end{align*}
As $M\to +\infty$,
\begin{align*}
I_{M} = \int_{-\theta}^{\theta} \Big( \log M + i \theta + O(M^{-1}\log M) \Big)\frac{d\theta}{2\pi} = \log M +  O(M^{-1}\log M),
\end{align*}
where we have used that the first $O$-term above is uniform for $\theta \in (-\pi,\pi)$. Similarly, as $\epsilon \to 0$, 
\begin{align*}
I_{\epsilon} = \int_{-\theta}^{\theta} \Big( -\log \epsilon - i \theta + O(\epsilon \log \epsilon) \Big) \frac{d\theta}{2\pi} = -\log \epsilon + O(\epsilon \log \epsilon),
\end{align*}
where we have used that the first $O$-term above is uniform for $\theta \in (-\pi,\pi)$. The integral $I_{\epsilon,M}$ can be evaluated explicitly using primitives. As $M\to + \infty$, we obtain
\begin{align*}
I_{\epsilon,M} = \log \bigg( \frac{(a+b+2\sqrt{ab})\epsilon}{2ab+(a+b)\epsilon + 2\sqrt{ab(a+\epsilon)(b+\epsilon)}} \bigg) - \log \bigg( \frac{4M}{a+b+2(\epsilon + \sqrt{(a+\epsilon)(b+\epsilon)})} \bigg) + O(M^{-1}),
\end{align*}
where the error term is independent of $\epsilon$. Hence
\begin{align*}
\lim_{\epsilon\to 0}\lim_{M\to +\infty} \big( I_{\epsilon,M} + \log M - \log \epsilon \big) = \lim_{\epsilon\to 0}\lim_{M\to +\infty} \big( I_{\epsilon,M} + I_{M} - I_{\epsilon} \big) = 4 \log \bigg( \frac{\sqrt{a}+\sqrt{b}}{2 (ab)^{\frac{1}{4}}} \bigg) ,
\end{align*} 
and the claim follows.
\end{proof}

We are now in a position to prove Theorem ~\ref{Thm_JUE LDP} (b).

\begin{proof}[Proof of Theorem~\ref{Thm_JUE LDP} (b)]
By \eqref{juepdf} with $\lambda_1=\alpha n+t$ and $\lambda_2= \beta n$, we have 
\begin{align}\label{lol1}
\mathbb{P} \Big( x_{ \rm max }^{ (  n )  }(\alpha n +t, \beta n ) < d \Big) = \frac{1}{Z_{n}(\alpha n+t,\beta n) }  \int_{0}^{d} \ldots \int_{0}^{d} \prod_{1\leq i < j \leq n} (x_{j}-x_{i})^{2} \prod_{j=1}^{n}e^{-nV(x_{j})}e^{W(x_{j})} \, dx_{j},
\end{align}
where $V$ is given by \eqref{def of V} and $W(x) := t \log x$.  
Recall that the associated equilibrium measure $\mu \equiv \mu_V$ is given by \eqref{def of mu d small}, where the analytic part of the density function $\psi$ is also defined.

Proposition \ref{Prop_constrained Wacther} (b2) shows that the equilibrium measure $\mu_{V}$ has one soft edge and one hard edge. Hence, we want to apply Theorem \ref{thm:general a and b} with $\mathcal{I}=(-\infty,d]$.
The potential $V$ is defined for $x\in (0,1)$ by \eqref{def of V}, and we extend it to $(-\infty,d]$ by setting $V(x) = +\infty$ if $x \leq 0$. Strictly speaking, $V$ does not satisfy the assumptions of Theorem \ref{thm:general a and b} because it is discontinuous at $x=0$. However, since $V$ is analytic in a neighbourhood of $(\mathfrak{m},d)$, the proof of \cite[Theorem 1.2]{CG21} shows that Theorem \ref{thm:general a and b} still applies to our potential $V$. Moreover, $W$ also satisfies the assumptions of Theorem \ref{thm:general a and b}. Thus by combining  Proposition \ref{Prop_constrained Wacther} (b2) and Theorem \ref{thm:general a and b}, we obtain 
\begin{align}\label{lol2}
\frac{1}{n!} \int_{0}^{d} \ldots \int_{0}^{d} \prod_{1\leq i < j \leq n} (x_{j}-x_{i})^{2} \prod_{j=1}^{n}e^{-nV(x_{j})}e^{W(x_{j})} \, dx_{j} = \exp\bigg(\widehat{D}_{1} n^{2} + \widehat{D}_{2} n + \widehat{D}_{3} \log n + \widehat{D}_{4} + O \Big( \frac{\log n}{n} \Big)\bigg)
\end{align}
as $n \to \infty$, where
\begin{align*}
& \widehat{D}_{1} = \log \frac{d-\mathfrak{m}  }{4} - \frac{3}{2} - \frac{1}{2} \int_{ \mathfrak{m} }^{ d } \bigg(V(x)-4\frac{d-x}{d-\mathfrak{m}} \bigg)\bigg( \frac{2}{(d-\mathfrak{m})\pi} + \psi(x) \bigg)\sqrt{\frac{x-\mathfrak{m}}{d-x}}\, dx, 
\\
& \widehat{D}_{2} = \log(2\pi) + \int_{\mathfrak{m}}^{d}W(x)\psi(x)\sqrt{\frac{x-\mathfrak{m}}{d-x}}\, dx, \nonumber 
\\
& \widehat{D}_{3} = - \frac{1}{6}, 
\\
& \widehat{ D }_{4} = 2\zeta^{\prime}(-1) - \frac{1}{8}\log \big( \tfrac{d- \mathfrak{m} }{2}\pi\psi(d) \big) - \frac{1}{24}\log \big( \tfrac{d- \mathfrak{m} }{2}\pi\psi(\mathfrak{m}) \big) \\
& \qquad + \frac{1}{4\pi^{2}}\int_{\mathfrak{m}}^{d}  \frac{W(x)}{\sqrt{(d-x)(x-\mathfrak{m})}} \bigg(\dashint_{\mathfrak{m}}^{d} \frac{W^{\prime}(y)\sqrt{(d-y)(y-\mathfrak{m})}}{x-y}\, dy \bigg) dx.
\end{align*}
Lemmas~\ref{lemma:some identities 1} and~\ref{lemma:some identities 2}, together with \eqref{def of mu d small}, allow us to simplify $\widehat{D}_{2}$ and $\widehat{D}_{4}$ as follows:
\begin{align*}
\widehat{D}_{2} & = \log(2\pi) + t \bigg( \alpha \frac{\sqrt{d}}{\sqrt{ \mathfrak{m} }} + \beta \frac{\sqrt{1-d}}{\sqrt{1-\mathfrak{m}}} \bigg) \log\frac{\sqrt{\mathfrak{m}}+\sqrt{d}}{2}  + t\alpha \log \bigg( \frac{1}{2\sqrt{\mathfrak{m}}}+\frac{1}{2\sqrt{d}} \bigg) - t \beta \log \frac{\sqrt{\mathfrak{m}(1-d)}+\sqrt{d(1-\mathfrak{m})}}{\sqrt{1-\mathfrak{m}}+\sqrt{1-d}}, \\
\widehat{D}_{4} & = 2\zeta^{\prime}(-1) - \frac{1}{8}\log \big( \tfrac{d- \mathfrak{m} }{2}\pi\psi(d) \big) - \frac{1}{24}\log \big( \tfrac{d- \mathfrak{m} }{2}\pi\psi(\mathfrak{m}) \big) + t^{2}  \log \bigg( \frac{\sqrt{\mathfrak{m}}+\sqrt{d}}{2(\mathfrak{m}d)^{\frac{1}{4}}}\bigg).
\end{align*}
Also, a long computation (similar to the proof of Lemma \ref{lemma:some identities 1}, so we omit details) allows to simplify the integral appearing in $\widehat{D}_{1}$ as follows
\begin{align*}
& \int_{\mathfrak{m}}^{d} \bigg(V(x)-4\frac{d-x}{d-\mathfrak{m} }\bigg)\bigg( \frac{2}{(d-\mathfrak{m})\pi} + \psi(x) \bigg)\sqrt{\frac{x-\mathfrak{m}}{d-x}}\, dx 
= -1 + \alpha \bigg( 1-\frac{\sqrt{d}}{\sqrt{\mathfrak{m}}} \bigg) + \beta \bigg( 1 - \frac{\sqrt{1-d}}{\sqrt{1-\mathfrak{m}}} \bigg)
\\
&\quad + \frac{\alpha^{2}}{2}\log(\mathfrak{m}d) + 2\alpha \beta \log \bigg( \frac{\sqrt{\mathfrak{m}(1-d)}+\sqrt{d(1-\mathfrak{m})}}{2} \bigg) + \frac{\beta^{2}}{2}\log \big( (1-\mathfrak{m})(1-d) \big) \nonumber \\
& \quad - \bigg( 2+\alpha+\beta+\alpha \frac{\sqrt{d}}{\sqrt{\mathfrak{m}}} + \beta \frac{\sqrt{1-d}}{\sqrt{1-\mathfrak{m}}} \bigg) \bigg( \alpha \log \frac{\sqrt{\mathfrak{m}}+\sqrt{d}}{2} + \beta \log \frac{\sqrt{1-\mathfrak{m}}+\sqrt{1-d}}{2} \bigg). \nonumber
\end{align*} 
Then by \eqref{eq of mathfrak m d}, this can be further simplified as 
\begin{align*}
\begin{split}
&\quad \int_{\mathfrak{m}}^{d} \bigg(V(x)-4\frac{d-x}{d-\mathfrak{m} }\bigg)\bigg( \frac{2}{(d-\mathfrak{m})\pi} + \psi(x) \bigg)\sqrt{\frac{x-\mathfrak{m}}{d-x}}\, dx
\\
&= -3+ \frac{\alpha^{2}}{2}\log(\mathfrak{m}d) + 2\alpha \beta \log \bigg( \frac{\sqrt{\mathfrak{m}(1-d)}+\sqrt{d(1-\mathfrak{m})}}{2} \bigg) + \frac{\beta^{2}}{2}\log \big( (1-\mathfrak{m})(1-d) \big) 
\\
&\quad  - 2(2+\alpha+\beta) \bigg( \alpha \log \frac{\sqrt{\mathfrak{m}}+\sqrt{d}}{2} + \beta \log \frac{\sqrt{1-\mathfrak{m}}+\sqrt{1-d}}{2} \bigg).
\end{split}
\end{align*}
This gives rise to 
\begin{align}
\begin{split} \label{widehat D1 evaluation}
\widehat{D}_{1} &=  \log \frac{d-\mathfrak{m}  }{4} - \alpha \beta \log \bigg( \frac{\sqrt{\mathfrak{m}(1-d)}+\sqrt{d(1-\mathfrak{m})}}{2} \bigg) - \frac{\alpha^{2}}{4}\log(\mathfrak{m}d) -   \frac{\beta^{2}}{4}\log \big( (1-\mathfrak{m})(1-d) \big) 
\\
&\quad + (2+\alpha+\beta) \bigg( \alpha \log \frac{\sqrt{\mathfrak{m}}+\sqrt{d}}{2} + \beta \log \frac{\sqrt{1-\mathfrak{m}}+\sqrt{1-d}}{2} \bigg).  
\end{split}
\end{align}
(It can be checked that $\widehat{D}_1= -S(\mathfrak{m},d)$, where $S(x, y)$ defined in \eqref{def of S(x,y)}.)

On the other hand, using \eqref{def of JUE partition function} and the asymptotic expansion \eqref{Barnes G asymp}, a direct computation yields
\begin{align}
\frac{1}{n!}\, Z_n( \alpha n+t,\beta n )  = \exp\bigg(\widetilde{D}_{1} n^{2} + \widetilde{D}_{2} n + \widetilde{D}_{3} \log n + \widetilde{D}_{4} +  O \Big( \frac{1}{n } \Big)\bigg), \qquad \mbox{as } n \to \infty,\label{lol3}
\end{align}
where 
\begin{align*}
& \widetilde{D}_{1} = \frac{(1+\alpha)^{2}}{2}\log(1+\alpha) + \frac{(1+\beta)^{2}}{2}\log(1+\beta) + \frac{(1+\alpha+\beta)^{2}}{2}\log(1+\alpha+\beta) \\
& \hspace{0.8cm} - \frac{\alpha^{2}}{2}\log \alpha - \frac{\beta^{2}}{2}\log \beta - \frac{(2+\alpha+\beta)^{2}}{2}\log(2+\alpha+\beta), 
\\
& \widetilde{D}_{2} = \log(2\pi) + t \Big( (1+\alpha)\log(1+\alpha) + (1+\alpha+\beta)\log(1+\alpha+\beta)  \\
& \hspace{0.8cm} - (2+\beta)\log(2+\alpha+\beta) - \alpha \log(\alpha(2+\alpha+\beta)) \Big), \nonumber \\
& \widetilde{D}_{3} = - \frac{1}{12}, \\
& \widetilde{D}_{4} = \zeta^{\prime}(-1) + \frac{1}{12}\log  \frac{\alpha \beta (2+\alpha+\beta)}{(1+\alpha)(1+\beta)(1+\alpha+\beta)} + \frac{t^{2}}{2}\log \frac{(1+\alpha)(1+\alpha+\beta)}{\alpha(2+\alpha+\beta)}.
\end{align*}
Substituting \eqref{lol2} and \eqref{lol3} into \eqref{lol1}, and simplifying, we obtain
\begin{equation}
L_k= \widehat{D}_k-\widetilde{D}_k, \qquad k=1,2,3,4,
\end{equation}
where $L_k$ is as in \eqref{lol4}.
\end{proof}
 
\subsection{Proof of Theorem \ref{Thm_JUE LDP weak} (b)}
As in the previous subsection, we start by proving some useful identities.
\begin{lem}\label{lemma:integrals with two hard edges}
Let $\beta>0$, $V(x):=- \beta \log(1-x)$, $d\in (0,1)$, and let $\psi$ be as in \eqref{def of mu d small N}. The following relations hold:
\begin{align}
& \int_{0}^{d}\frac{\log(1-x)}{\sqrt{x(d-x)}}dx = -2\pi \log \bigg( \frac{2}{1+\sqrt{1-d}} \bigg), \label{lol11} \\
& \int_{0}^{d}\frac{1}{1-x}\frac{\log(1-x)}{\sqrt{x(d-x)}}dx = \frac{\pi}{\sqrt{1-d}} \log\bigg( \frac{4(1-d)}{d} \frac{1-\sqrt{1-d}}{1+\sqrt{1-d}} \bigg), \label{lol12} \\
& \int_{0}^{d} V(x) \bigg( \frac{1}{\pi}+\psi(x) \bigg) \frac{dx}{\sqrt{x(d-x)}} = \beta \bigg( (\beta+4)\log\bigg( \frac{2}{1+\sqrt{1-d}} \bigg) + \frac{\beta}{2} \log \bigg( \frac{4(1-d)}{d} \frac{1-\sqrt{1-d}}{1+\sqrt{1-d}} \bigg) \bigg). \label{lol13}
\end{align}
\end{lem}
\begin{proof}
By contour deformations, we have
\begin{align*}
\int_{0}^{d}\frac{\log(1-x)}{\sqrt{x(d-x)}}dx = \frac{1}{2i} \int_{\mathcal{C}}\frac{\log(1-x)}{\sqrt{x}\sqrt{x-d}}\, dx,
\end{align*}
where the principal branch is used for each function in the integrand of the right-hand side, and $\mathcal{C}$ is a positively oriented contour surrounding $[0,d]$ once and not crossing $[1,+\infty)$. Next, we write
\begin{align*}
\frac{1}{2i} \int_{\mathcal{C}}\frac{\log(1-x)}{\sqrt{x}\sqrt{x-d}}dx = \frac{1}{2i} \int_{\mathcal{C}}\log(1-x)\bigg(\frac{1}{\sqrt{x}\sqrt{x-d}}-\frac{1}{x}\bigg) dx + \frac{1}{2i} \int_{\mathcal{C}}\frac{\log(1-x)}{x}dx =:I_{1}+I_{2}.
\end{align*}
The function $\frac{\log(1-x)}{x}$ is analytic inside $\mathcal{C}$, and thus $I_{2}=0$. Since
\begin{align*}
\frac{1}{\sqrt{x}\sqrt{x-d}}-\frac{1}{x} = O(x^{-2}), \qquad \mbox{as } x \to \infty,
\end{align*}
the integral $I_{1}$ can be evaluated by first deforming $\mathcal{C}$ into $C(0,M)\cup(-C(1,\epsilon))\cup \big( (1+\epsilon,M)+i0_{+} \big) \cup \big( (M,1+\epsilon)-i0_{+} \big)$ for some $M>1+\epsilon>1$, and then taking the limits $M\to + \infty$ and $\epsilon\to 0$. The contribution along $C(0,M)\cup(-C(1,\epsilon))$ vanish in these limits, and we find
\begin{align*}
\frac{1}{2i} \int_{\mathcal{C}}\frac{\log(1-x)}{\sqrt{x}\sqrt{x-d}}dx = \frac{1}{2i} \int_{1}^{+\infty} (-2\pi i) \bigg( \frac{1}{\sqrt{x}\sqrt{x-d}}-\frac{1}{x} \bigg) dx = -2\pi \log \bigg( \frac{2}{1+\sqrt{1-d}} \bigg),
\end{align*}
which is \eqref{lol11}. For \eqref{lol12}, we use the same contour deformation and obtain (after taking $M\to + \infty$)
\begin{align}\label{lol16}
\int_{0}^{d}\frac{1}{1-x}\frac{\log(1-x)}{\sqrt{x(d-x)}}dx = \frac{1}{2i} \int_{1+\epsilon}^{+\infty}  \frac{-2\pi i}{1-x}\frac{dx}{\sqrt{x}\sqrt{x-d}} - \frac{1}{2i} \int_{C(1,\epsilon)} \frac{\log(1-x)}{1-x}\frac{dx}{\sqrt{x}\sqrt{x-d}} =: I_{3}+I_{4}.
\end{align}
The first term $I_{3}$ can be evaluated exactly using primitives, from which we readily obtain
\begin{align}\label{lol14}
I_{3} = -\frac{\pi \log \epsilon}{\sqrt{1-d}} + \frac{\pi}{\sqrt{1-d}} \log\bigg( \frac{4(1-d)}{d} \frac{1-\sqrt{1-d}}{1+\sqrt{1-d}} \bigg) + O(\epsilon), \qquad \mbox{as } \epsilon \to 0,
\end{align}
while for $I_{4}$ we have
\begin{align}\label{lol15}
I_{4} = \frac{\pi \log \epsilon}{\sqrt{1-d}} + O(\epsilon), \qquad \mbox{as } \epsilon \to 0.
\end{align}
Taking $\epsilon \to 0$ in \eqref{lol16} using \eqref{lol14} and \eqref{lol15} yields \eqref{lol12}. Formula \eqref{lol13} directly follows by substituting $\psi$ and using \eqref{lol11} and \eqref{lol12}.
\end{proof}

Next, we show Theorem~\ref{Thm_JUE LDP weak} (b). 

\begin{proof}[Proof of Theorem~\ref{Thm_JUE LDP weak} (b)]
By \eqref{juepdf} with $\lambda_1=\alpha  $ and $\lambda_2= \beta n$, we have 
\begin{align}\label{lol1 N}
\mathbb{P} \Big( x_{ \rm max }^{ ( n )  }(\alpha, \beta n) < d \Big) = \frac{1}{Z_n(\alpha, \beta n) }   \int_{0}^{d} \ldots \int_{0}^{d} \prod_{1\leq i < j \leq n} (x_{j}-x_{i})^{2} \prod_{j=1}^{n}e^{-nV(x_{j})} \, dx_{j},
\end{align}
where 
\begin{align}\label{def of V alpha0}
V(x) := \begin{cases}
- \beta \log(1-x), & \mbox{if } x\in (0,d), 
\smallskip 
\\
+\infty, & \mbox{if } x\in [d,1).
\end{cases}
\end{align}
By Proposition~\ref{Prop_constrained Wacther} (b1), the associated equilibrium measure is given by \eqref{def of mu d small N}. Furthermore, $V$ satisfies the assumptions of Theorem \ref{thm:general a and b Jacobi-type}. It follows that 
\begin{align}\label{lol2 N}
\frac{1}{n!} \int_{0}^{d} \ldots \int_{0}^{d} \prod_{1\leq i < j \leq n} (x_{j}-x_{i})^{2} \prod_{j=1}^{n}e^{-nV(x_{j})} \, dx_{j} = \exp\bigg(\widehat{C}_{1} n^{2} + \widehat{C}_{2} n + \widehat{C}_{3} \log n + \widehat{C}_{4} +  O \Big( \frac{\log n}{n} \Big)\bigg)
\end{align}
as $n \to \infty$, where
\begin{align*}
& \widehat{C}_{1} = \log \frac{d}{4} - \frac{1}{2} \int_{0}^{d} V(x)\Big( \frac{1}{\pi} + \psi(x) \Big)\frac{dx}{\sqrt{x(d-x)}}, 
\\
& \widehat{C}_{2} = \alpha \log \frac{d}{2} + \log(2\pi) - \alpha \log 2 - \frac{\alpha}{2\pi}\int_{0}^{d} \frac{V(x)}{\sqrt{x(d-x)}}\, dx + \frac{\alpha}{2}V(0), \nonumber \\
& \widehat{C}_{3} = - \frac{1}{4}+\frac{\alpha^{2}}{2}, \\
& \widehat{C}_{4} = 3\zeta^{\prime}(-1) + \frac{\log 2}{12}  - \frac{1}{8}\log \big( \pi\psi(d) \big) - \frac{1-4\alpha^{2}}{8}\log \big( \pi\psi(0) \big) + \frac{\alpha}{2}\log(2\pi) - \frac{\alpha^{2}}{2}\log 2 -\log G(1+\alpha).
\end{align*}
It also follows from \eqref{def of JUE partition function} and \eqref{Barnes G asymp} that 
\begin{align}
Z_n( \alpha,\beta n )  = \exp\bigg(\widetilde{C}_{1} n^{2} + \widetilde{C}_{2} n + \widetilde{C}_{3} \log n + \widetilde{C}_{4} + O \Big( \frac{1}{n } \Big)\bigg), \qquad \mbox{as } n \to \infty,\label{lol3 N}
\end{align}
where 
\begin{align*}
& \widetilde{C}_{1} = (1+\beta)^{2} \log(1+\beta) - \frac{\beta^{2}}{2}\log \beta - \frac{(2+\beta)^{2}}{2}\log(2+\beta), \\
& \widetilde{C}_{2} = \log(2\pi) + \alpha (1+\beta)\log(1+\beta)-\alpha(2+\beta)\log(2+\beta), \nonumber \\
& \widetilde{C}_{3} = \frac{3\alpha^{2}-1}{6}, \\
& \widetilde{C}_{4} = 2\zeta^{\prime}(-1) - \alpha^{2} \mathrm{arccoth}(3+2\beta)+\frac{\alpha}{2}\log(2\pi)+\frac{1}{12}\log \bigg( \frac{\beta(2+\beta)}{(1+\beta)^{2}} \bigg)-\log G(1+\alpha).
\end{align*}
Substituting \eqref{lol2 N} and \eqref{lol3 N} in \eqref{lol1 N}, and simplifying using Lemma \ref{lemma:integrals with two hard edges}, we find \eqref{lol4 N}.
\end{proof}

\subsection{Proof of Theorem~\ref{th:lpp-ut} on the upper tail LPP probability}

Let $\rho_{n}(x)$ denote the one-point eigenvalue density of \eqref{juepdf}, defined such that for any continuous function $f$,
\begin{equation}
\mathbb{E}\bigg[\sum_{j=1}^{n}f(x_{j})\bigg] = \int_{0}^{1}f(x)\rho_{n}(x) \, dx.
\end{equation}
The eigenvalues $x_{1},\ldots,x_{n}$ in \eqref{juepdf} form a determinantal point process, meaning that all correlation functions can be represented as determinants involving a scalar function of two variables $K_{n}(x,y)$ known as the correlation kernel (see e.g. \cite[Section 3.3.1]{Co05} for some background). The eigenvalue density corresponds to the first correlation function and can be represented as $\rho_{n}(x) = K_{n}(x,x)$.

In \cite{Fo12}, the large $n$ asymptotics of $\rho_{n}$ outside the support is computed for the JUE up to terms of order $1$. Moreover, \cite[Eq. (14.136)]{Fo10} gives the relation between the asymptotics of $\rho_{n}$ outside its support and the upper tail probability of $x^{(n)}_{\mathrm{max}}(\lambda_1,\lambda_2)$. For completeness, we give a quick proof of this fact.
\begin{lem}
\label{le:juedens}
Consider the JUE \eqref{juepdf}, and recall from \eqref{def of Wachter distribution} that the support of the equilibrium measure is $[a,b]\subset [0,1]$. Suppose $b<1$ and that $d \in (b,1)$ is fixed. Assume also that
\begin{equation}
\int_{d}^{1}  \rho_{n}(x)\,dx \sim e^{-nA_{1}+A_{2}}, \qquad n \to \infty \label{evdasympt}
\end{equation}
for some constants $A_{1}>0$ and $A_{2}$. Then
\begin{equation*}
\mathbb{P}(x^{(n)}_{\mathrm{max}}(\lambda_1,\lambda_2) < d) = 1-e^{-nA_{1}+A_{2}} + O(e^{-2nA_{1}}), \qquad n \to \infty.
\end{equation*} 
\end{lem}

\begin{proof}
The distribution of $x^{(n)}_{\mathrm{max}}(\lambda_1,\lambda_2)$ can be written as
\begin{align}
	\begin{split}
\mathbb{P}(x^{(n)}_{\mathrm{max}}(\lambda_1,\lambda_2) < d) & = 1+\sum_{l=1}^{n}\frac{(-1)^{l}}{l!}\int_{[d,1]^{l}}  \mathrm{det}\bigg\{K_{n}(x_i,x_j)\bigg\}_{i,j=1}^{l} \, dx_{1}\ldots dx_{l}   \\
 & = 1-\int_{d}^{1}\rho_{n}(x)\, dx+\sum_{l=2}^{n}\frac{(-1)^{l}}{l!}\int_{[d,1]^{l}}  \mathrm{det}\bigg\{K_{n}(x_i,x_j)\bigg\}_{i,j=1}^{l} \, dx_{1}\ldots dx_{l}. \label{lol20}
 \end{split}
\end{align}
Let $E_{n}(d)$ denote the last sum (with $l$ running from $2$ to $n$). Since the matrices $\{K_{n}(x_i,x_j)\}_{i,j=1}^{\ell}$ are Hermitian positive definite, Hadamard’s inequality (see e.g. \cite[Proposition 9.1.4]{Fo10}) implies that
\begin{align*}
\mathrm{det}\bigg\{K_{n}(x_i,x_j)\bigg\}_{i,j=1}^{l} \leq \prod_{i=1}^{l} K_{n}(x_{i},x_{i}), \qquad \mbox{for any } l \in \N, \, x_{1},\ldots,x_{l}\in [d,1].
\end{align*}
Hence,
\begin{equation*}
|E_{n}(d)| \leq \sum_{l=2}^{\infty}\frac{1}{l!}\bigg(\int_{d}^{1} \rho_{n}(x)\,dx \bigg)^{l} \leq \bigg(\int_{d}^{1} \rho_{n}(x)\,dx \bigg)^{2}
\end{equation*}
for all sufficiently large $n$, where we used $e^{x}-1-x \leq x^{2}$ for $x \in [0,1]$. Inserting \eqref{evdasympt} completes the proof. 
\end{proof}

Now we recall the result \cite[Eq. (1.5)]{Fo12} here in the specific case of the JUE$_n(\alpha n,\beta n)$: 
\begin{equation}\label{lol21}
\rho_{n}(x) \sim \mathrm{exp}(-2n\widetilde{\Phi}_{\alpha,\beta}(x))\,C_{\alpha,\beta}(x)
\end{equation}
where $\widetilde{\Phi}_{\alpha,\beta}(x)$ and $C_{\alpha,\beta}(x)$ are $C^{1}$ functions defined on $x\in(b,1]$ where $b$ is the upper end point of the support. Furthermore, $\widetilde{\Phi}_{\alpha,\beta}(x)$ is the monotonically increasing function of $x$ given by
\begin{equation}
\begin{split}
\widetilde{\Phi}_{\alpha,\beta}(x)& = -\log\bigg(\frac{2x-a-b+2\sqrt{(x-b)(x-a)}}{b-a}\bigg)+\alpha\log\bigg(\frac{\sqrt{ab}+x-\sqrt{(x-b)(x-a)}}{(\sqrt{a}+\sqrt{b})\sqrt{x}}\bigg)
\\
&\quad +\beta\log\bigg(\frac{\sqrt{(1-a)(1-b)}-x+1+\sqrt{(x-b)(x-a)}}{(\sqrt{1-a}+\sqrt{1-b})\sqrt{1-x}}\bigg)
\end{split}
\end{equation}
and $C_{\alpha,\beta}(x)$ is given by
\begin{equation}
C_{\alpha,\beta}(x) = \frac{b-a}{8\pi}\,\frac{1}{(x-a)(x-b)}.
\end{equation}
Now applying Lemma \ref{le:juedens} we obtain
\begin{equation*}
\begin{split}
\mathbb{P}(x_{\mathrm{max}}^{(n)} (\alpha n,\beta n) > x) &\sim \int_{x}^{1}\rho_{n}(y)\, dy
\sim\int_{x}^{1}e^{-2n\widetilde{\Phi}_{\alpha,\beta}(y)}C_{\alpha,\beta}(y)\,dy, \qquad \mbox{as } n \to \infty.
\end{split}
\end{equation*}
Since $\widetilde{\Phi}_{\alpha,\beta}(y)$ is a positive and monotonically increasing function of $y > b$, the dominant contribution comes from a small neighbourhood of $x$, the lower end point of integration. By standard saddle point arguments we obtain
\begin{equation*}
\begin{split}
\mathbb{P}(x_{\mathrm{max}}^{(n)} (\alpha n,\beta n) > x) &\sim \frac{e^{-2n\widetilde{\Phi}_{\alpha,\beta}(x)}C_{\alpha,\beta}(x)}{2n\widetilde{\Phi}'_{\alpha,\beta}(x)}, \qquad \mbox{as } n \to \infty.
\end{split}
\end{equation*}

Then, replacing $\alpha$ by $\alpha_{1} + \frac{\alpha_{2}}{n}$ in the above and applying a Taylor series expansion (using that \eqref{lol21} is uniform for $\alpha$ in compact subsets of $[0,+\infty)$), we obtain Theorem~\ref{Thm_JUE LDP} (a) by setting $\alpha_{1}=\alpha$ and $\alpha_{2}=t$, and Theorem \ref{Thm_JUE LDP weak} (a) by setting $\alpha_{1}=0$ and $\alpha_{2}=\mathrm{n}$. Theorem \ref{th:lpp-ut} now directly follows by combining Theorems \ref{Thm_JUE LDP weak}, \ref{Thm_JUE LDP} and \ref{th:jue-lpp} (note that $\widetilde{\Phi}_{\alpha,\beta}$ is related to \eqref{utexpan} as $\widetilde{\Phi}_{\alpha,\beta} = \Phi_{\alpha+1,\beta}$).

\appendix

\section{Results from \cite{CG21}} \label{Appendix_general results on Hankel}

\begin{thm}[particular case of Theorem 1.2 in \cite{CG21}]\label{theorem L} 
	Let $\mathcal{I}=[-1,+\infty)$ and $V:\mathcal{I}\to \R$ be a one-cut regular potential whose equilibrium measure is supported on $[-1,1]$ with density $\psi(x)\sqrt{\frac{1-x}{1+x}}$. Let $W:\mathcal{I}\to \R$ be analytic in a neighbourhood of $[-1,1]$, locally H\"{o}lder continuous on $\mathcal{I}$ and such that $W(x) = O(V(x))$ as $x\to + \infty$. As $n \to \infty$, we have
	\begin{equation*}
		L_{n}(V,W) = \exp\bigg(C_{1} n^{2} + C_{2} n + C_{3} \log n + C_{4} +  O \Big( \frac{\log n}{n} \Big)\bigg),
	\end{equation*}
	where
	\begin{align*}
		& C_{1} = - \log 2 - \frac{3}{2} - \frac{1}{2} \int_{-1}^{1} (V(x)-2(x+1))\Big( \frac{1}{\pi} + \psi(x) \Big)\sqrt{\frac{1-x}{1+x}}\, dx, \\
		& C_{2} = \log(2\pi) + \int_{-1}^{1}W(x)\psi(x)\sqrt{\frac{1-x}{1+x}} \, dx, \nonumber \\
		& C_{3} = - \frac{1}{6}, \\
		& C_{4} = 2\zeta^{\prime}(-1) - \frac{1}{8}\log ( \pi\psi(-1) ) - \frac{1}{24}\log ( \pi\psi(1) ) + \frac{1}{4\pi^{2}}\int_{-1}^{1}  \frac{W(x)}{\sqrt{1-x^{2}}} \bigg(\dashint_{-1}^{1} \frac{W^{\prime}(y)\sqrt{1-y^{2}}}{x-y}\, dy \bigg) \,dx.
	\end{align*}
\end{thm}

Using Theorem \ref{theorem L} and the change of variables $x_{j}\to -x_{j}$ in \eqref{def of Ln}, we get the following.
\begin{thm}\label{thm:Laguerre type v2}
	Let $\mathcal{I}=(-\infty,1]$ and $V:\mathcal{I}\to \R$ be a one-cut regular potential whose equilibrium measure is supported on $[-1,1]$ with density $\psi(x)\sqrt{\frac{1+x}{1-x}}$. Let $W:\mathcal{I}\to \R$ be analytic in a neighbourhood of $[-1,1]$, locally H\"{o}lder continuous on $\mathcal{I}$ and such that $W(x) =  O(V(x))$ as $x\to - \infty$. As $n \to \infty$, we have
	\begin{equation*}
		L_{n}(V) = \exp\bigg(C_{1} n^{2} + C_{2} n + C_{3} \log n + C_{4} +  O \Big( \frac{\log n}{n} \Big)\bigg),
	\end{equation*}
	where
	\begin{align*}
		& C_{1} = - \log 2 - \frac{3}{2} - \frac{1}{2} \int_{-1}^{1} (V(x)-2(1-x))\Big( \frac{1}{\pi} + \psi(x) \Big)\sqrt{\frac{1+x}{1-x}} \, dx, \\
		& C_{2} = \log(2\pi) + \int_{-1}^{1}W(x)\psi(x)\sqrt{\frac{1+x}{1-x}} \, dx, \nonumber \\
		& C_{3} = - \frac{1}{6}, \\
		& C_{4} = 2\zeta^{\prime}(-1) - \frac{1}{8}\log ( \pi\psi(1) ) - \frac{1}{24}\log ( \pi\psi(-1) ) + \frac{1}{4\pi^{2}}\int_{-1}^{1}  \frac{W(x)}{\sqrt{1-x^{2}}} \bigg(\dashint_{-1}^{1} \frac{W^{\prime}(y)\sqrt{1-y^{2}}}{x-y} \, dy \bigg) \, dx.
	\end{align*}
\end{thm}

Following \cite[Remark 1.4]{Ch19}, we extend the above result to handle equilibrium measures supported on $[a,b]$ with a soft edge at $a$ and a hard edge at $b$.

Suppose $w(x)=e^{-nV(x)}e^{W(x)}$ is such that the associated equilibrium measure is supported on $[a,b]$. Then, the change of variables $y_{j} =( x_{j} - \frac{a+b}{2} )/( \frac{b-a}{2} )$, $j = 1,...,n$ in \eqref{def of Ln} shows that
\begin{equation}\label{L to Ltilde}
	L_{n}(V,W)|_{\mathcal{I}=(-\infty,b]} = \Big( \frac{b-a}{2} \Big)^{n^{2}} L_{n}(\widetilde{V},\widetilde{W})|_{\mathcal{I}=(-\infty,1]},
\end{equation}
where 
\begin{equation*}
	\widetilde{V}(x) = V( \tfrac{a+b}{2}+ \tfrac{b-a}{2} \, x ), \qquad \widetilde{W}(x) = W( \tfrac{a+b}{2}+  \tfrac{b-a}{2}\,x ).
\end{equation*}

Also, if $\ell$ and $\psi(x) \sqrt{\frac{x-a}{b-x}}$ are respectively the Euler-Lagrange constant and the density of the equilibrium measure of $V$, the change of variables $s =  ( y - \frac{a+b}{2} )/ ( \frac{b-a}{2})$ in \eqref{EL equality 2} and \eqref{EL inequality 2} implies that the equilibrium measure of $\widetilde{V}$ is supported on $[-1,1]$ and can be written as $\widetilde{\psi}(x)\sqrt{\frac{1+x}{1-x}}$. Furthermore, $\widetilde{\psi}$ and the Euler-Lagrange constant $\widetilde{\ell}$ of $\widetilde{V}$ are given by
\begin{equation} \label{A2}
	\widetilde{\ell} = \ell + 2\log(\tfrac{b-a}{2}), \qquad \widetilde{\psi}(x) =  \frac{b-a}{2} \psi( \tfrac{a+b}{2} +  \tfrac{b-a}{2} \,x ).
\end{equation}
The assumptions on $W$ in Theorem \ref{thm:Laguerre type v2} can in fact be relaxed; it suffices to assume that $W$ satisfies Assumption \ref{ass:W}. This can be shown straightforwardly following the argument in \cite[Section 1.3]{CFWW25}.
Substituting the above relations in Theorem \ref{thm:Laguerre type v2} and simplifying, we obtain Theorem~\ref{thm:general a and b}.

Similarly, if $\ell$ and $ \frac{\psi(x)}{\sqrt{(b-x)(x-a)}}$ are respectively the Euler-Lagrange constant and the density of the equilibrium measure of $V$, the change of variables $s =  ( y - \frac{a+b}{2} )/ ( \frac{b-a}{2} )$ in \eqref{EL equality 2} and \eqref{EL inequality 2} implies that the equilibrium measure of $\widetilde{V}$ is supported on $[-1,1]$ and can be written as $\frac{\widetilde{\psi}(x)}{\sqrt{1-x^{2}}}$. Furthermore, $\widetilde{\psi}$ and the Euler-Lagrange constant $\widetilde{\ell}$ of $\widetilde{V}$ are given by 
\begin{equation}
	\widetilde{\ell} = \ell + 2\log(\tfrac{b-a}{2}), \qquad \widetilde{\psi}(x) =  \psi( \tfrac{a+b}{2} +  \tfrac{b-a}{2} x).
\end{equation}
Using the above relations in \cite[Theorem 1.3]{CG21} and simplifying, we obtain Theorem~\ref{thm:general a and b Jacobi-type}.

\bigskip

\subsection*{Acknowledgement}
SSB was supported by the National Research Foundation of Korea grant (RS-2023-00301976, RS-2025-00516909). CC is a Research Associate of the Fonds de la Recherche Scientifique - FNRS. CC also acknowledges support from the Swedish Research Council, Grant No. 2021-04626, and from the European Research Council (ERC), Grant Agreement No. 101115687.
PM was supported by the Swedish Foundation for International Cooperation in Research and Higher Education (PD2023-9315).
NS was supported from the Royal Society, grant URF\textbackslash R\textbackslash 231028.
The authors are grateful to Elnur Emrah, Peter Forrester, Kurt Johansson, Pierre Le Doussal, Yuchen Liao, and Matteo Mucciconi for their interest and valuable discussions.


\bigskip












 

\begin{thebibliography}{999}


\bibitem{ABGS25} T. Alberts, R. Basu, S. Groathouse and X. Shen, \emph{Large deviations of geodesic midpoint fluctuations in last-passage percolation with general iid weights}, arXiv:2502.00942 

\bibitem{AC21} T. Alberts and E. Cator, \emph{On the passage time geometry of the last passage percolation problem}, ALEA Lat. Am. J. Probab. Math. Stat. \textbf{18} (2021), 211--247.

\bibitem{AL2025} M. Allard, S. Lahiry, \textit{Birth of a gap: Critical phenomena in 2D Coulomb gas}, arXiv:2509.24529.


\bibitem{AFLS25} M. Allard, P. J. Forrester, S. Lahiry and B.-J. Shen, \emph{Partition function of 2D Coulomb gases with radially symmetric potentials and a hard wall}, arXiv:2506.14738. 
 

\bibitem{ABK21} G. Akemann, S.-S. Byun and N.-G. Kang, \emph{A non-Hermitian generalisation of the Marchenko–Pastur distribution: from the circular law to multi-criticality}, Ann. Henri Poincaré \textbf{22} (2021), 1035--1068.
 

\bibitem{AGKWW14} G. Akemann, T. Guhr, M. Kieburg, R. Wegner and T. Wirtz, \emph{Completing the picture for the smallest eigenvalue of real Wishart matrices}, Phys. Rev. Lett. \textbf{113} (2014), 250201.
 

\bibitem{ACC23}  Y. Ameur, C. Charlier and J. Cronvall, \emph{Free energy and fluctuations in the random normal matrix model with spectral gaps}, Constr. Approx. (Online), https://doi.org/10.1007/s00365-025-09720-9, arXiv:2312.13904.  
 

\bibitem{ACCL22} Y. Ameur, C. Charlier, J. Cronvall and J. Lenells, \emph{Disk counting statistics near hard edges of random normal matrices: the multi-component regime}, Adv. Math. \textbf{441} (2024), 109549.
 

\bibitem{AS21} S. Armstrong and S. Serfaty, \emph{Local laws and rigidity for Coulomb gases at any temperature}, Ann. Probab. \textbf{49} (2021), 46--121.
 
\bibitem{AKW22} T. Assiotis, J. P. Keating and J. Warren, \emph{On the joint moments of the characteristic polynomials of random unitary matrices}, Int. Math. Res. Not. \textbf{2022} (2022), 14564--14603.
 
\bibitem{BR01} J. Baik and E. M. Rains, \emph{Algebraic aspects of increasing subsequences}, Duke Math. J. \textbf{109} (2001), 1--65.
 
\bibitem{BDMMZ01} J. Baik, P. Deift, K. T.-R. McLaughlin, P. Miller and X. Zhou, \emph{Optimal tail estimates for directed last passage site percolation with geometric random variables}, Adv. Theor. Math. Phys. \textbf{5} (2001), no. 6, 1207--1250.

\bibitem{BBLM15} F. Balogh, M. Bertola, S.-Y. Lee and K. D. T.-R. McLaughlin, \emph{Strong asymptotics of the orthogonal polynomials with respect to a measure supported on the plane}, Comm. Pure Appl. Math. \textbf{68} (2015), 112--172.
 
\bibitem{BGM17} F. Balogh, T. Grava and D. Merzi, \textit{Orthogonal polynomials for a class of measures with discrete rotational symmetries in the complex plane}, Constr. Approx. \textbf{46} (2017), 109--169. 

\bibitem{BBBK24} J. Baslingker, R. Basu, S. Bhattacharjee and M. Krishnapur, \emph{Optimal tail estimates in $\beta$-ensembles and applications to last passage percolation}, arXiv:2405.12215. 

\bibitem{BGS21} R. Basu, S. Ganguly and A. Sly, \emph{Upper tail large deviations in first passage percolation}, Comm. Pure Appl. Math. \textbf{74} (2021), 1577--1640. 
 

\bibitem{BBNY19} R. Bauerschmidt, P. Bourgade, M. Nikula and H.-T. Yau, \emph{The two-dimensional Coulomb plasma: quasi-free approximation and central limit theorem}, Adv. Theor. Math. Phys. \textbf{23} (2019), 841--1002.
 
 

\bibitem{BDG01} G. Ben Arous, A. Dembo and A. Guionnet, \emph{Aging of spherical spin glasses}, Probab. Theory Related Fields \textbf{120} (2001), 1--67. 
 

\bibitem{BS25} S. Berezin and E. Strahov, \textit{Last-passage percolation and product-matrix ensembles}, arXiv:2503.22801. 


\bibitem{BEG18} M. Bertola, J. G. Elias Rebelo and T. Grava, \emph{Painlevé IV critical asymptotics for orthogonal polynomials in the complex plane}, SIGMA Symmetry Integrability Geom. Methods Appl. \textbf{14} (2018), Paper No. 091, 34pp.  

\bibitem{BCL2024} E. Blackstone, C. Charlier and J. Lenells, \emph{Toeplitz determinants with a one-cut regular potential and Fisher--Hartwig singularities I. Equilibrium measure supported on the unit circle}, Proc. Roy. Soc. Edinburgh Sect. A \textbf{154} (2024), 1431--1472.

\bibitem{BO00} A. Borodin and A. Okounkov, \emph{A Fredholm determinant formula for Toeplitz determinants}, Integral Equations Operator Theory \textbf{37} (2000), 386396.

\bibitem{BEMN11} G. Borot, B. Eynard, S. N. Majumdar and C. Nadal, \emph{Large deviations of the maximal eigenvalue of random matrices}, J. Stat. Mech. Theory Exp. \textbf{56} (2011), P11024.

\bibitem{BG13} G. Borot and A. Guionnet, \emph{Asymptotic expansion of $\beta$ matrix models in the one-cut regime}, Comm. Math. Phys. \textbf{317} (2013), 447--483.

\bibitem{Byun25} S.-S. Byun, \emph{Anomalous free energy expansions of planar Coulomb gases: multi-component and conformal singularity}, arXiv:2508.00316.

\bibitem{BCMSEquilibrium} S.-S. Byun, C. Charlier, P. Moreillon and N. Simm, \emph{Planar equilibrium measure for  truncated ensembles with a point charge}, preprint.

\bibitem{BF25} S.-S.~Byun and P. J.~Forrester, \emph{Progress on the study of the Ginibre ensembles}, KIAS Springer Ser. Math. \textbf{3} Springer, 2025, 221pp.

\bibitem{BF25a} S.-S.~Byun and P. J.~Forrester, \emph{Electrostatic computations for statistical mechanics and random matrix applications}, arXiv:2510.14334.

\bibitem{BFKL25} S.-S. Byun, P. J. Forrester, A. B. J. Kuijlaars and S. Lahiry, \emph{Orthogonal polynomials in the spherical ensemble with two insertions},  arXiv:2503.15732.

\bibitem{BFL25} S.-S. Byun, P. J. Forrester and S. Lahiry, \emph{Properties of the one-component Coulomb gas on a sphere with two macroscopic external charges}, Pure Appl. Funct. Anal. (to appear), arXiv:2501.05061.

\bibitem{BKS23} S.-S. Byun, N.-G. Kang and S.-M. Seo, \emph{Partition functions of determinantal and Pfaffian Coulomb gases with radially symmetric potentials}, Comm. Math. Phys. \textbf{401} (2023), 1627--1663.

\bibitem{BSY24} S.-S Byun, S.-M. Seo and M. Yang, \emph{Free energy expansions of a conditional GinUE and large deviations of the smallest eigenvalue of the LUE}, Comm. Pure Appl. Math. \textbf{78} (2025), 2245--2502.

\bibitem{BKSY25} S.-S Byun, N.-G. Kang, S.-M. Seo and M. Yang, \emph{Free energy of spherical Coulomb gases with point charges}, J. Lond. Math. Soc. \textbf{112} (2025), e70294.

\bibitem{BP24} S.-S. Byun and S. Park, \emph{Large gap probabilities of complex and symplectic spherical ensembles with point charges}, arXiv:2405.00386. 

\bibitem{BY23} S.-S. Byun and M. Yang, \emph{Determinantal Coulomb gas ensembles with a class of discrete rotational symmetric potentials}, SIAM J. Math. Anal. \textbf{55} (2023), 6867--6897.

\bibitem{BY25} S.-S. Byun and E. Yoo, \emph{Three topological phases of the elliptic Ginibre ensembles with a point charge}, arXiv:2502.02948. 
 
\bibitem{CC22} M. Cafasso and T. Claeys, \emph{A Riemann-Hilbert approach to the lower tail of the Kardar-Parisi-Zhang equation}, Comm. Pure Appl. Math. \textbf{75} (2022), 493--540.

\bibitem{CC2025} M. Cafasso and T. Claeys, \textit{Biorthogonal measures, polymer partition functions, and random matrices}, Ann. Inst. H. Poincare. Prob. Stat. (to appear), arXiv:2401.10130.
 

\bibitem{CFTW15} T. Can, P. J. Forrester, G. Téllez and P. Wiegmann, \emph{Exact and asymptotic features of the edge density profile for the one component plasma in two dimensions}, J. Stat. Phys. \textbf{158} (2015), 1147--1180.


\bibitem{Ch19} C. Charlier, \emph{Asymptotics of Hankel determinants with a one-cut regular potential and Fisher–Hartwig singularities}, Int. Math. Res. Not. \textbf{2019} (2019), 7515--7576.
 

\bibitem{Ch22} C. Charlier, \emph{Asymptotics of determinants with a rotation-invariant weight and discontinuities along circles}, Adv. Math. \textbf{408} (2022), 108600.
 
\bibitem{Ch23} C. Charlier, \emph{Large gap asymptotics on annuli in the random normal matrix model}, Math. Ann. \textbf{388} (2024), 3529--3587.
 

\bibitem{CCR22} C. Charlier, T. Claeys and G. Ruzza, \emph{Uniform tail asymptotics for Airy kernel determinant solutions to KdV and for the narrow wedge solution to KPZ}, J. Funct. Anal. \textbf{283} (2022), Paper No. 109608, 54 pp.
  

\bibitem{CFWW25} C. Charlier, B. Fahs, C. Webb and M. D. Wong, \emph{Asymptotics of Hankel determinants with a multi-cut regular potential and Fisher-Hartwig singularities}, Mem. Amer. Math. Soc. \textbf{310} (2025), 1567.
 

\bibitem{CG21} C. Charlier and R. Gharakhloo, \emph{Asymptotics of Hankel determinants with a Laguerre-type or Jacobi-type potential and Fisher–Hartwig singularities} Adv. Math. \textbf{383} (2021), 107672.
 

\bibitem{Co05} B. Collins, \emph{Product of random projections, Jacobi ensembles and universality problems arising from free probability}, Probab. Theory Relat. Fields \textbf{133} (2005), 315--344.
 
\bibitem{CG20} I. Corwin and P. Ghosal, \emph{Lower tail of the KPZ equation}, Duke Math. J. \textbf{169} (2020), 1329--1395.

\bibitem{CLW16} I. Corwin, Z. Liu and D. Wang, \emph{Fluctuations of TASEP and LPP with general initial data}, Ann. Appl. Probab.  \textbf{26} (2016), 2030--2082. 
 

\bibitem{CK22} J. G. Criado del Rey and A. B. J. Kuijlaars, \emph{A vector equilibrium problem for symmetrically located point charges on a sphere}, Constr. Approx. \textbf{55} (2022), 775--827.
  
 
\bibitem{DDV24} S. Das, D. Dauvergne and B. Virág, \emph{Upper tail large deviations of the directed landscape}, Duke Math. J. (to appear), arXiv:2405.14924.

\bibitem{DD22} S. Das and E. Dimitrov, \textit{Large deviations for discrete $\beta$-ensembles}, J. Funct. Anal. \textbf{283} (2022), 109487. 

\bibitem{DLM23} S. Das, Y. Liao and M. Mucciconi, \emph{Large deviations for the $q$-deformed polynuclear growth}, Ann. Probab. \textbf{53} (2025), 1223--1286.
 
\bibitem{DLM24} S. Das, Y. Liao and M. Mucciconi, \emph{Lower tail large deviations of the stochastic six vertex model}, Int. Math. Res. Not. \textbf{2025} (2025), 1--38. 

\bibitem{DT21} S. Das and L.-C. Tsai, \emph{Fractional moments of the stochastic heat equation}, Ann. Inst. H. Poincar\'e Probab. Stat. \textbf{57} (2021), 778--799. 

\bibitem{DZ22} S. Das and W. Zhu, \emph{Upper-tail large deviation principle for the ASEP}, Electron. J. Probab. \textbf{27} (2022), Paper No. 11, 34 pp.

\bibitem{DOV22} D. Dauvergne, J. Ortmann and B. Virág, \emph{The directed landscape}, Acta Math. \textbf{229} (2022), 201--285.

\bibitem{DM06} D. S. Dean and S. N. Majumdar, \emph{Large deviations of extreme eigenvalues of random matrices}, Phys. Rev. Lett. \textbf{97} (2006), 160201. 


\bibitem{DM08} D. S. Dean and S. N. Majumdar, \emph{Extreme value statistics of eigenvalues of Gaussian random matrices}, Phys. Rev. E \textbf{77} (2008), 041108. 	

\bibitem{DS22} A. Deaño and N. Simm, \emph{Characteristic polynomials of complex random matrices and Painlevé transcendents}, Int. Math. Res. Not. \textbf{2022} (2022), 210--264.

\bibitem{DMMS25} A. Deaño, L. Molag, K. D. T.-R. McLaughlin and N. Simm, \emph{Asymptotics for a class of planar orthogonal polynomials and truncated unitary matrices}, arXiv:2505.12633. 

\bibitem{DIK} P. Deift, A. Its and I. Krasovsky, \textit{Asymptotics of Toeplitz, Hankel, and Toeplitz+Hankel determinants with Fisher-Hartwig singularities}, Ann. of Math. \textbf{174} (2011), 1243--1299.

\bibitem{DZ99} J.-D. Deuschel and O. Zeitouni, \emph{On increasing subsequences of I.I.D. samples}, Combin. Probab. Comput. \textbf{8} (1999), 247--263.
 

\bibitem{FS09} Z. M. Feng and J. P. Song, \emph{Integrals over the circular ensembles relating to classical domains}, J. Phys. A \textbf{42} (2009), 325204.
 

\bibitem{Fo94} P. J. Forrester, \emph{Addendum to: Selberg correlation integrals and the $1/r^2$ quantum many-body system}, Nucl. Phys. B \textbf{416} (1994), 377--385.
 
 
\bibitem{Fo10} P. J. Forrester, \emph{Log-gases and random matrices}, Princeton University Press, Princeton, NJ, 2010.

\bibitem{Fo12} P. J. Forrester, \emph{Large deviation eigenvalue density for the soft edge Laguerre and Jacobi $\beta$-ensembles}, J. Phys. A \textbf{45} (2012), 145201.

\bibitem{Fo25} P. J. Forrester, \emph{Dualities in random matrix theory}, arXiv:2501.07144.

\bibitem{FR05} P. J. Forrester and E. M. Rains, \emph{Interpretations of some parameter dependent generalizations of classical matrix ensembles}, Probab. Theory Related Fields \textbf{131} (2005), 1--61.

\bibitem{FW08} P. J. Forrester and S. O. Warnaar, \emph{The importance of the Selberg integral}, Bull. Amer. Math. Soc. \textbf{45} (2008), 489--534.

\bibitem{FW01} P. J. Forrester and N. S. Witte, \emph{Application of the $\tau$-function theory of Painlev\'e equations to random matrices: PIV, PII and the GUE}, Comm. Math. Phys. \textbf{219} (2001), 357--398.
 
\bibitem{FW02} P. J. Forrester and N. S. Witte, \emph{Application of the $\tau$-function theory of Painlev\'e equations to random matrices: PV , PIII, the LUE, JUE, and CUE}, Comm. Pure Appl. Math. \textbf{55} (2002), 679--727.
 
\bibitem{FW04} P. J. Forrester and N. S. Witte, \emph{Application of the $\tau$-function theory of Painlev\'e equations to random matrices: PVI, the JUE, CyUE, cJUE and scaled limits}, Nagoya Math. J. \textbf{174} (2004), 29--114.
 

\bibitem{FW12} P. J. Forrester and N. S. Witte, \emph{Asymptotic forms for hard and soft edge general $\beta$ conditional gap probabilities}, Nuclear Phys. B \textbf{859}, (2012), 321--340. 
 
\bibitem{FK07} Y. V. Fyodorov and B. A. Khoruzhenko, \emph{A few remarks on colour–flavour transformations, truncations of random unitary matrices, Berezin reproducing kernels and Selberg-type integrals}, J. Phys. A \textbf{40} (2007), 669--699.

\bibitem{HMO25} J. Husson, G. Mazzuca and A. Occelli, \emph{Discrete and continuous Muttalib–Borodin process: large deviations and Riemann–Hilbert analysis}, arXiv:2505.23164.

\bibitem{JMP94}  B. Jancovici, G. Manificat and C. Pisani, \emph{Coulomb systems seen as critical systems: finite-size effects in two dimensions}, J. Stat. Phys. \textbf{76} (1994), 307--329.

\bibitem{Joh00} K. Johansson, \emph{Shape fluctuations and random matrices}, Commun. Math. Phys. \textbf{209} (2000), 437--476.
 
 
\bibitem{K93} K. W. J. Kadell, \emph{The Selberg-Jack symmetric functions}, Adv. Math. \textbf{130} (1997), 33--102.


\bibitem{K97} J. Kaneko, \emph{Selberg integrals and hypergeometric functions associated with Jack polynomials}, SIAM J. Math. Anal. \textbf{24} (1993), 1086--1110.

\bibitem{KC10} E. Katzav and I. P. Castillo, \emph{Large deviations of the smallest eigenvalue of the Wishart–Laguerre ensemble}, Phys. Rev. E \textbf{82} (2010), 040104.

 
\bibitem{KS00} J. P. Keating and N. C. Snaith, \emph{Random Matrix Theory and $\zeta(1/2+it)$}, Comm. Math. Phys. \textbf{214} (2000), 57--89.

\bibitem{KKL25} M. Kieburg, A. B. J. Kuijlaars and S. Lahiry, \textit{Orthogonal polynomials in the normal matrix model with two insertions}, Nonlinearity \textbf{38} (2025), 065013.

\bibitem{KLY25} T. Krüger, S.-Y. Lee and M. Yang, \emph{Local statistics in normal matrix models with merging singularity}, Comm. Math. Phys. \textbf{406} (2025), 13.
 

\bibitem{LS17} T. Leblé and S. Serfaty, \emph{Large deviation principle for empirical fields of log and Riesz gases}, Invent. Math. \textbf{210} (2017), 645--757.


\bibitem{LMS16} P. Le Doussal, S. N. Majumdar and G. Schehr, \emph{Large deviations for the height in 1D Kardar-Parisi-Zhang growth at late times}, Europhys. Lett. \textbf{113} (2016), 60004. 
 

\bibitem{LY17} S.-Y. Lee and M. Yang, \emph{Discontinuity in the asymptotic behavior of planar orthogonal polynomials under a perturbation of the Gaussian weight}, Comm. Math. Phys. \textbf{355} (2017), 303--338.
 

\bibitem{LY23} S.-Y. Lee and M. Yang, \emph{Strong asymptotics of planar orthogonal polynomials: Gaussian weight perturbed by finite number of point charges}, Comm. Pure Appl. Math. \textbf{76} (2023), 2888--2956.
 

\bibitem{LS77} B. F. Logan and L. A. Shepp, \emph{A variational problem for random Young tableaux}, Adv. Math. \textbf{26} (1977), 206--222.

\bibitem{Macdonald} I. G. Macdonald, \emph{Symmetric Functions and Hall Polynomials}, Oxford Mathematical Monographs, 1979.

 
\bibitem{MSV97} D. S. Moak, E. B. Saff and R. S. Varga, \emph{On the zeros of Jacobi polynomials $P_n^{(\alpha_n,\beta_n)}(x)$}, Trans. Am. Math. Soc. \textbf{249} (1979), 159--162. 

	
\bibitem{MS14} S. N. Majumdar and G. Schehr, \emph{Top eigenvalue of a random matrix: Large deviations and third order phase transition}, J. Stat. Mech. \textbf{2014} (2014), P01012.

\bibitem{MV09} S. N. Majumdar and M. Vergassola, \emph{Large deviations of the maximum eigenvalue for Wishart and Gaussian random matrices}, Phys. Rev. Lett. \textbf{102} (2009), 060601.


\bibitem{NM11} C. Nadal and S. N. Majumdar, \emph{A simple derivation of the Tracy-Widom distribution of the maximal eigenvalue of a Gaussian unitary random matrix}, J. Stat. Mech. Theory Exp. \textbf{29} (2011), P04001.
 
\bibitem{Roug2025} N. Rougerie, \textit{Free-energy variations for determinantal 2D plasmas with holes}, arXiv:2510.01745.


\bibitem{NSW20} M. Nikula, E. Saksman and C. Webb, \emph{Multiplicative chaos and the characteristic polynomial of the CUE: the $L^1$-phase}, Trans. Am. Math. Soc. \textbf{373} (2020), 3905--3965.  

\bibitem{Noda2025} K. Noda, \textit{Partition functions of two-dimensional Coulomb gases with circular root- and jump-type singularities}, arXiv:2510.00843.

  
\bibitem{NIST}  F. W. J. Olver, D. W. Lozier, R. F. Boisvert and C. W. Clark, eds. \emph{NIST Handbook of Mathematical Functions}, Cambridge: Cambridge University Press, 2010.
  
\bibitem{PS16} A. Perret and G. Schehr, \emph{Finite $N$ corrections to the limiting distribution of the smallest eigenvalue of Wishart complex matrices}, Random Matrices Theory Appl. \textbf{5} (2016), 1650001.
 

\bibitem{QT21} J. Quastel and L.-C. Tsai, \emph{Hydrodynamic large deviations of TASEP}, Comm. Pure Appl. Math \textbf{78} (2025), 913--994.

\bibitem{RKC12} H. M. Ramli, E. Katzav and I. P. Castillo, \emph{Spectral properties of the Jacobi ensembles via the Coulomb gas approach}, J. Phys. A \textbf{45} (2012), 465005.
 

\bibitem{ST97} E. B. Saff and V. Totik, \emph{Logarithmic Potentials with External Fields}, Grundlehren der Mathematischen Wissenschaften, Springer-Verlag, Berlin, 1997.
 
\bibitem{SS23} A. Serebryakov and N. Simm, \emph{Schur function expansion in non-Hermitian ensembles and averages of characteristic polynomials}, Ann. Henri Poincaré \textbf{26} (2025), 1927--1974.
	

\bibitem{Se23} S. Serfaty, \emph{Gaussian fluctuations and free energy expansion for Coulomb gases at any temperature}, Ann. Inst. Henri Poincaré Probab. Stat. \textbf{59} (2023), 1074--1142.
 

\bibitem{Sep98} T. Seppäläinen, \emph{Coupling the totally asymmetric simple exclusion process with a moving interface}, Markov Process. Related Fields \textbf{4} (1998), 593--628.
 

\bibitem{Sep98a} T. Seppäläinen, \emph{Large deviations for increasing sequences on the plane}, Probab. Theory Related Fields \textbf{112} (1998), 221--244. 
 

\bibitem{SF25} B.-J. Shen and P. J. Forrester, \emph{Moments of characteristic polynomials for classical $\beta$ ensembles}, arXiv:2502.07142. 

\bibitem{SW24} N. Simm and F. Wei, \emph{On moments of the derivative of CUE characteristic polynomials and the Riemann zeta function}, arXiv:2409.03687. 

\bibitem{TF99} G. Téllez and P. J. Forrester, \emph{Exact finite-size study of the 2D OCP at $\Gamma=4$ and $\Gamma= 6$}, J. Stat. Phys. \textbf{97} (1999), 489--521.
 

\bibitem{Tsai22} L.-C. Tsai, \emph{Exact lower-tail large deviations of the KPZ equation}, Duke Math. J. \textbf{171} (2022), no.9, 1879--1922.

  
\bibitem{VMB07} P. Vivo, S. N. Majumdar and O. Bohigas, \emph{Large deviations of the maximum eigenvalue in Wishart random matrices}, J. Phys. A \textbf{40} (2007), 4317--4337. 
 

\bibitem{Wach80} K. W. Wachter, \emph{The limiting empirical measure of multiple discriminant ratios}, Ann. Stat. \textbf{8} (1980), 937--957.
 

\bibitem{WW19} C. Webb and M. D. Wong, \emph{On the moments of the characteristic polynomial of a Ginibre random matrix}, Proc. Lond. Math. Soc. \textbf{118} (2019), 1017--1056.


\bibitem{WKG15} T. Wirtz, M. Kieburg and T. Guhr, \emph{Limiting statistics of the largest and smallest eigenvalues in the correlated Wishart model}, Europhys. Lett. \textbf{109} (2015), 20005. 
 
\bibitem{ZW06} A. Zabrodin and P. Wiegmann, \emph{Large-$N$ expansion for the 2D Dyson gas}, J. Phys. A \textbf{39} (2006), 8933--8964.

\bibitem{ZS2000} K. \.{Z}yczkowski and H.-J. Sommers, \textit{Truncations of random unitary matrices}, J. Phys. A \textbf{33} (2000), no. 10, 2045--2057.
 
\end{thebibliography}
\end{document}